\documentclass[final]{article}

\title{SIGLE: valid Selective Inference procedure for Generalized~Linear~Lasso}
\date{April 2023}
\usepackage{natbib}

\usepackage{tkz-tab}
\usepackage{afterpage}
\usepackage{fullpage}

\usepackage[utf8]{inputenc} 
\usepackage[T1]{fontenc}    
\usepackage{hyperref}
\definecolor{burgundy}{rgb}{0.5, 0.0, 0.13}
\definecolor{camel}{rgb}{0.76, 0.6, 0.42}
\definecolor{chamoisee}{rgb}{0.63, 0.47, 0.35}
\definecolor{grey1}{RGB}{128,128,128}

\usetikzlibrary{bayesnet}
\usetikzlibrary{arrows}
 \hypersetup{colorlinks=true,linkcolor=burgundy,citecolor=chamoisee,urlcolor=burgundy,linktoc=page}
\usepackage{url}            
\usepackage{booktabs}       
\usepackage{amsfonts}       
\usepackage{nicefrac}       
\usepackage{microtype}      
\usepackage{dsfont}         

\author{%
  Quentin Duchemin\\
  Swiss Data Science Center, École polytechnique fédérale de Lausanne\\
  1015, Lausanne, Switzerland\\
  \texttt{quentin.duchemin@epfl.ch} \\
  $\And$\\
   Yohann De Castro\\
  Univ. Lyon, École Centrale de Lyon, CNRS UMR 5208\\ 
  Institut Camille Jordan\\ 
  36 Avenue Guy de Collongue, 69134 Écully, France\\
  Institut Universitaire de France (IUF)\\
  \texttt{yohann.de-castro@ec-lyon.fr}}

\usepackage{comment}
\usepackage{tcolorbox}
\usepackage{amsmath}
\usepackage{amsthm}

\newtheorem{thm}{Theorem}
\newtheorem{Lemma}{Lemma}
\newtheorem{proposition}{Proposition}

\newtheorem{assumption}{Assumption}

\theoremstyle{remark}

\usepackage{amssymb}
\usepackage{mathtools}
\DeclarePairedDelimiter\ceil{\lceil}{\rceil}

\usepackage{diagbox}
\usepackage{subcaption}
\usepackage{bbold}
\usepackage{tikz}
\usetikzlibrary{%
  arrows,%
  calc,%
  shapes.geometric,%
  shapes.misc,%
  shapes.symbols,%
  shapes.arrows,%
  automata,%
  through,%
  positioning,%
  scopes,%
  decorations.shapes,%
  decorations.text,%
  decorations.pathmorphing,%
  shadows}
\usetikzlibrary{positioning}
\usepackage{multirow}
\usepackage{makecell}
\usepackage{enumitem}
\usepackage{empheq}
\makeatletter
\newcommand{\opnorm}{\@ifstar\@opnorms\@opnorm}
\makeatother

\usepackage{colortbl}
\usepackage{multirow}

\usepackage{algorithm,algorithmic,refcount}
\usepackage{tikz}
\usetikzlibrary{arrows,calc,shapes,decorations.pathreplacing}
\usepackage{subcaption}

\newcolumntype{C}[1]{>{\centering\arraybackslash}m{#1}}



\usepackage[leftbars]{changebar}

\usepackage{pifont}
\newcommand{\cmark}{\ding{51}}%
\newcommand{\xmark}{\ding{55}}%

\usepackage{algorithm,algorithmic,refcount}
\definecolor{light-gray}{gray}{0.8}

\definecolor{auburn}{rgb}{0.43, 0.21, 0.1}

\begin{document}

\maketitle

\begin{abstract}
This article investigates uncertainty quantification of the generalized linear lasso~(GLL), a popular variable selection method in high-dimensional regression settings. In many fields of study, researchers use data-driven methods to select a subset of variables that are most likely to be associated with a response variable. However, such variable selection methods can introduce bias and increase the likelihood of false positives, leading to incorrect conclusions. 

In this paper, we propose a post-selection inference framework that addresses these issues and allows for valid statistical inference after variable selection using GLL. 
We show that our method provides accurate $p$-values and confidence intervals, while maintaining high statistical power. 

In a second stage, we focus on the sparse logistic regression, a popular classifier in high-dimensional statistics. We show with extensive numerical simulations that SIGLE is more powerful than state-of-the-art PSI methods. SIGLE relies on a new method to sample states from the distribution of observations conditional on the selection event. This method is based on a simulated annealing strategy whose energy is given by the first order conditions of the logistic lasso.

\end{abstract}

\section{Introduction}
\label{sec:intro}

In modern statistics, the number of predictors can far exceed the number of observations available. In this high-dimensional context, $\ell_1$ regularisation leads to a small number of predictors to be selected (referred to as the selected support) while allowing for a minimax optimal prediction error, see for instance \cite[Chapter 2]{van2016estimation}. The estimated parameters and support are not explicitly known and are obtained by solving a convex optimisation program in practice. This makes inference of the model parameters difficult if not impossible. 

In this context, the application of standard inference methods without taking into account the use of data to select the model usually leads to undesirable statistical properties. Post-selection inference (PSI) is designed to address this issue. It consists of constructing inference procedures considering that the vector of observations $Y$ is distributed according to the distribution conditional on the so-called selection event. In the literature, the problem of post-selection inference has been studied mainly for linear regression with Gaussian noise and assuming that the model has been selected using LASSO. Leaving this specific framework is an essential step for applications and more challenges can be expected in the study of PSI procedures for a generalized linear model (GLM). Moreover, the ubiquity of the logistic model to solve practical regression problems and the surge of high dimensional data-sets make the sparse logistic regression~(SLR) more and more attractive. In this frame, it becomes crucial to provide certifiable guarantees on the output of the SLR, e.g. confidence intervals.

Inference procedures with statistical guarantees in the Generalized Linear Model (GLM) are few, if any. The practitioner is often left with no valid option to quantifies the uncertainty of predictions in high-dimensional GLMs. To the best of our knowledge, she might use the recent work of \cite{taylorGLM} for inference with the Generalized Linear LASSO (GLL). Based on a heuristic argument, \cite{taylorGLM} quantifies the uncertainty of the solutions of GLL. 

The main contribution of this article is three fold. First, we introduce {\bf SIGLE} (Selective Inference for Generalized Linear Estimation), {a new conditional MLE approach to provide testing procedures and confidence regions for the solutions of GLL. SIGLE relies one a new sampling scheme from the distribution of observations conditional on the selection event. 

Second, we focus on the SLR and we introduce a new method to sample states according to the conditional distribution, allowing the use of SIGLE in this context. We empirically witness that SIGLE is more powerful than current state-of-the-art methods. On Figure~\ref{fig:pvals-PSI}, we observe that our testing procedure (SIGLE) is correctly calibrated and we compare its power with the method from \cite{taylorGLM} and with a {\it weak learner}. This weak learner is a two-sided test based on the statistic $\sum_{i=1}^n |\overline \pi_i^{\theta_0} - y_i|$ where $\overline \pi^{\theta_0}$ is the expectation of the vector of observations under the null conditional on the selection event (cf. Eq.\eqref{eq:gradbar}).~More experiments can be found in Section~\ref{sec:hypo-testing}.

Last but not least, we prove a new conditional Central Limit Theorem (CLT) that exhibits conditions under which the SIGLE statistic is asymptotically normal. These assumptions hold under considerations similar to those commonly used in the study of asymptotic properties of subset selection via the Lasso in linear models (cf. \cite{taylorGLM,bunea}) and are not of particular interest for practical applications. This conditional CLT is a significant contribution and can be read at three levels of granularity. First it motivates the choice of the SIGLE statistic in this work. Second, it opens new perspective regarding the theoretical analysis of PSI methods in GLL. Indeed, while~\cite{taylorGLM} focus first on getting unconditional asymptotic result before considering the distribution of the limit distribution conditional on the selection event, we directly consider the conditional distribution of the SIGLE statistic before analyzing its asymptotic limit. Let us stress out that the asymptotic result stated in~\cite{taylorGLM} relies on non rigorous computations. Third, we believe that the proof of our conditional CLT might be of independent interest. In particular, we are--as far as we know--the first to correct the proof from~\cite{Liang2012MaximumLE} which has been reported as false (cf.  \cite{zhang18}). }

\begin{figure}[ht!]
\begin{center}
 \centering
    \begin{subfigure}[b]{0.49\textwidth}
        \centering
        \includegraphics[width=\textwidth]{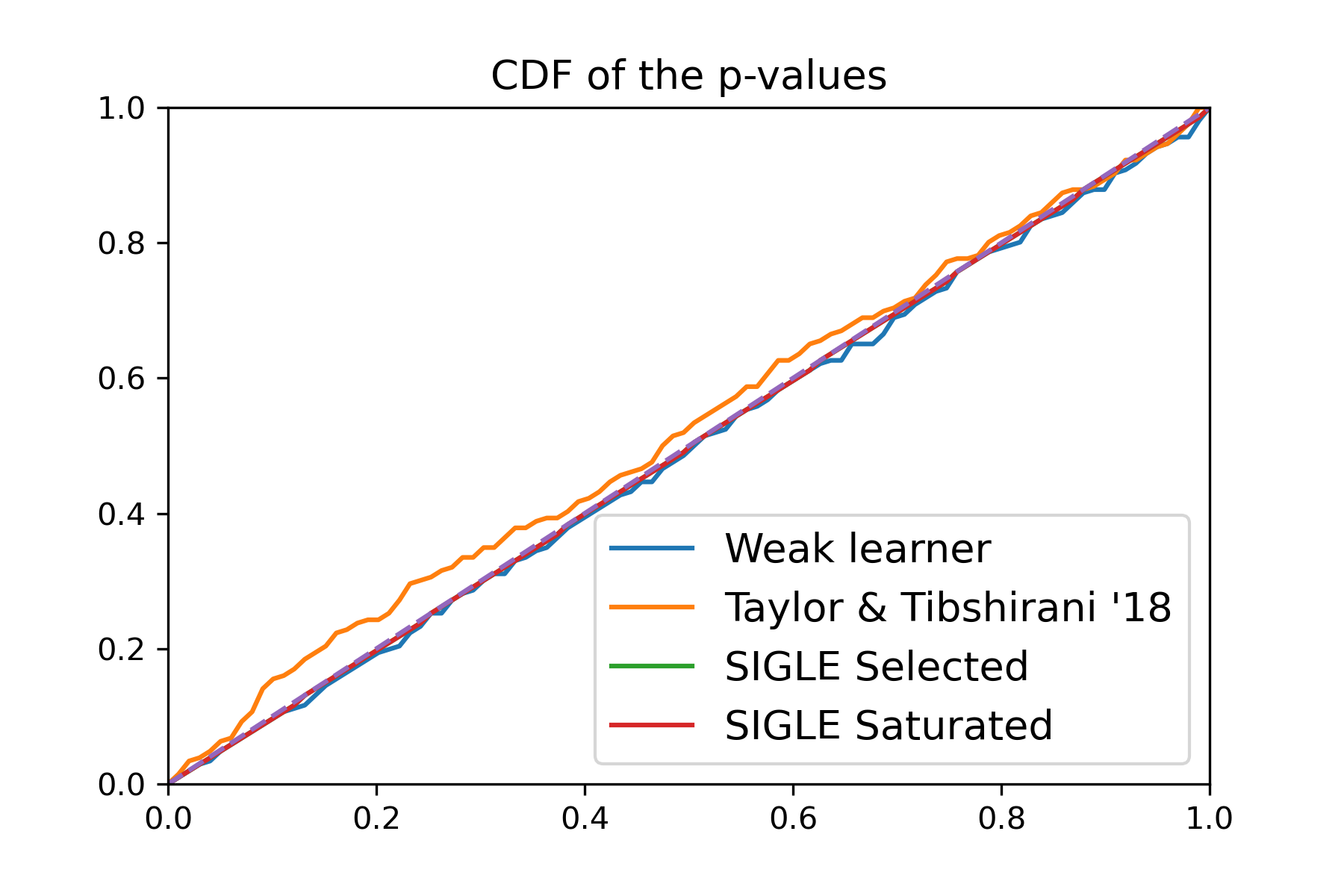}
        \caption[]%
        {{\small  $\vartheta^*= [0 , \ldots , 0]$.}}    
    \end{subfigure}
    \hfill
     \centering
    \begin{subfigure}[b]{0.49\textwidth}
        \centering
        \includegraphics[width=\textwidth]{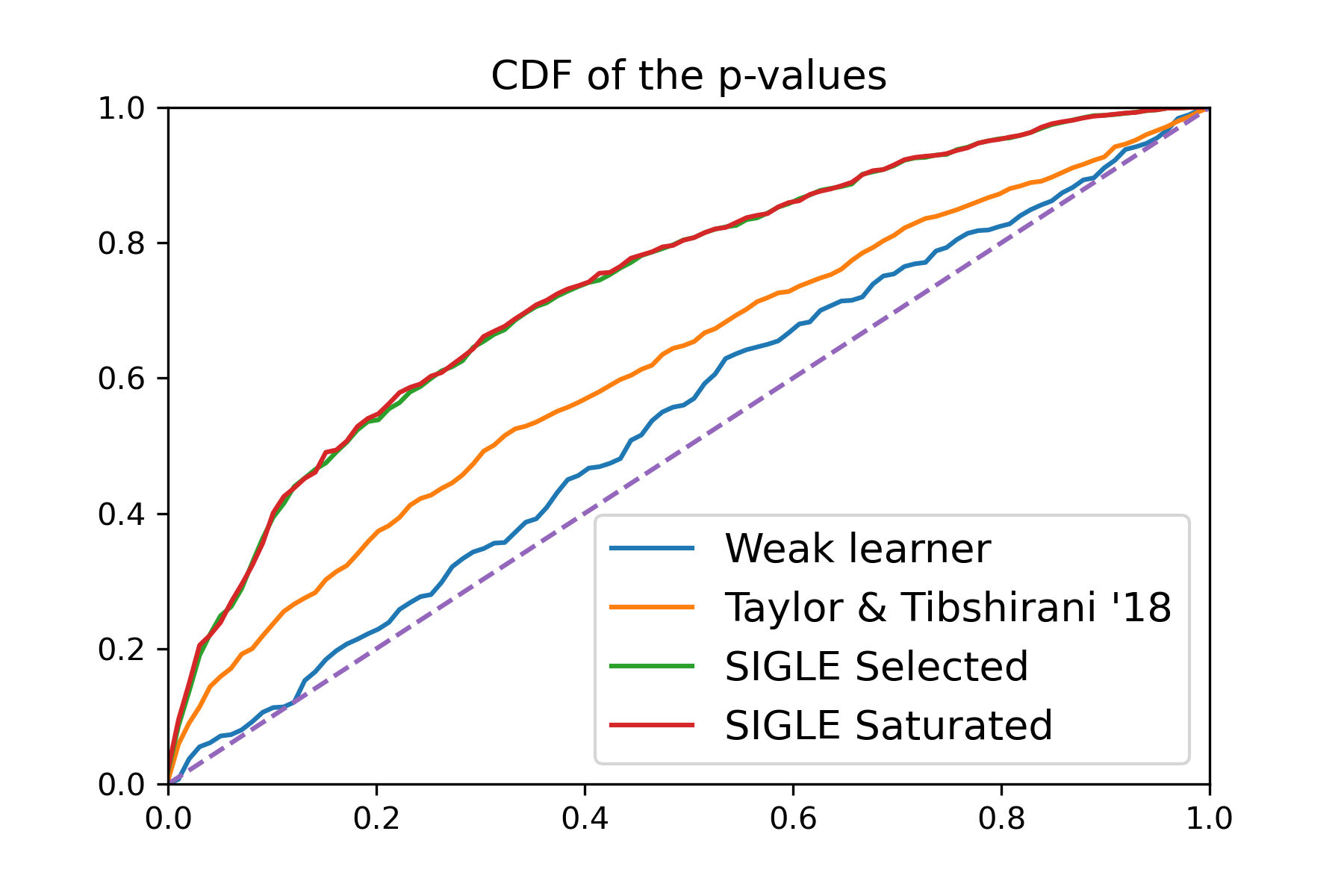}
        \caption[]%
        {{\small $\vartheta^*= [0.2,0.2,0,\ldots,0]$.}}    
    \end{subfigure}
\caption{In the logistic model, we consider a design matrix $\mathbf X \in \mathds R^{200\times 10}$ where we sample independently each entry with respect to a standard normal distribution. On Figures $(a)$ and $(b)$, we show the cumulative distribution function (CDF) of the p-values under obtained from $i)$ a {\it weak learner}, $ii)$ the procedure TT-1 (cf. Section~\ref{sec:hypo-testing}) adapted from \cite{taylorGLM} and $iii)$ SIGLE. On Figure $(a)$ we work under the global null showing that SIGLE and the weak leaner are correctly calibrated. The method from \cite{taylorGLM} does not show $p$-values systematically larger than uniform. On Figure $(b)$ we work under the alternative $\vartheta^*=[0.2,0.2,0,\dots]\in \mathds R^{10}$. 
}
\label{fig:pvals-PSI}\end{center}
\vskip -0.2in
\end{figure}

\subsection{Post-Selection Inference for high-dimensional GLM}
\label{sec:1.1}

We are interested in a target parameter $\vartheta^\star\in\Theta\subseteq\mathds R^d$ attached to the distribution~$\mathds P_{\vartheta^\star}$ of $N$ independent response variables $Y:=(y_1,\ldots,y_N)\in\mathcal Y^N\subseteq\mathds R^N$ given by the data  $Z:=(z_1,\ldots,z_N)$ where $z_i=(\mathbf x_i,y_i)\in\mathcal X\times \mathcal Y$ with $\mathbf x_i\in\mathcal X\subseteq\mathds R^d$ a covariate, namely a vector of $d$ predictors. The family of generalized linear models, or GLMs for short, is based on modeling the conditional distribution of the responses $y_i\in\mathcal Y$ given the covariate $\mathbf x_i\in\mathcal X$ in an exponential family form, namely
\[
\mathds P_{\vartheta^\star}(y|\mathbf x)=h_{v}(y)\exp
\Big\{
\frac{y\langle\mathbf x,\vartheta^\star\rangle-\xi(\langle\mathbf x,\vartheta^\star\rangle)}{v}
\Big\}\,,
\]
where $v>0$ is a scale parameter, and $\xi\,:\,\mathds R\to\mathds R$ is the partition function which is assumed to be of class~$\mathcal C^{m+1}$ (with $m$ a non-negative integer). For sake of readability, the dependence on $\mathbf X$ will be omitted when it is clear from the context, and we will simply denote $\mathds P_{\vartheta^*}(\cdot \,|\, \mathbf x)$ by $\mathds P_{\vartheta^*}(\cdot)$. Standard examples are $\xi(t)=t^2/2$ for the Gaussian linear model with noise variance~$v$ and observation space $\mathcal Y=\mathds R$, or $v=1$, $\xi(t)=\exp(t)$ and $\mathcal Y=\{0,1,2,\ldots\}$ for the Poisson regression. Last but not least, we will consider in this paper the logistic regression where $v=1$, $\xi(t)=\log(1+\exp(t))$ and $\mathcal Y=\{0,1\}$. 

The negative log-likelihood takes the form
\begin{equation}\label{eq:GLM}
\forall \vartheta\in\Theta\,,\ 
\mathcal L_N(\vartheta,Z)
:=\sum_{i=1}^N
\xi(\langle\mathbf x_i,\vartheta\rangle)
-\langle y_i\mathbf x_i,\vartheta\rangle\,.
\end{equation}
Assume that the partition function $\xi$ is differentiable, then the score function is 
\[
\forall \vartheta\in\Theta\,,\ 
\nabla_\vartheta\mathcal L_N(\vartheta,Z)
=\mathbf X^\top
\big(
\sigma(\mathbf X\vartheta)-Y
\big)\,,
\]
where $\sigma=\xi'$ is the derivative of the partition function and $\mathbf X\in\mathds R^{N\times d}$ is referred to as the design matrix whose rows are the covariates and the columns are the predictors. Note that $\sigma(\mathbf X \vartheta)$ should be understood as applying entrywise the function $\sigma$ to the vector $\mathbf X \vartheta$. In a high-dimensional context one has more predictors than observations ({\it i.e.,} $N\ll d$), and one would like to select a small number of predictors to explain the response. We use an $\ell_1$-regularization to enforce a structure of sparsity in $\vartheta$. Our overall estimator is based on solving the Generalized Linear Lasso (GLL)
\begin{equation}
\label{e:generalized_lasso}
\hat\vartheta^{\lambda}\in\arg\min_{\vartheta\in\Theta}
\big\{
\mathcal L_N(\vartheta,Z)+\lambda \|\vartheta\|_1
\big\} ,
\end{equation}
where $\lambda>0$ is a user-defined regularization hyperparameter. We assume that the negative log-likelihood is strictly convex. This assumption is satisfied for instance in the Gaussian linear model or logistic regression. In this case, it is necessary and sufficient that the solutions $\hat\vartheta^{\lambda}$ to \eqref{e:generalized_lasso} satisfy
the following Karush–Kuhn–Tucker (KKT) conditions 
 \begin{subequations}\label{eq:gKKTconditions}
    \begin{empheq}[left=\empheqlbrace]{align}
     \mathbf X^{\top} \Big( Y -\sigma(\mathbf X \hat{\vartheta}^{\lambda})\Big)&=\lambda \widehat{{S}},&&\label{eq:KKTY}\\
     \widehat {S}_k&=\mathrm{sign}(\hat{\vartheta}_k^{\lambda})\quad &&\text{if } \widehat{\vartheta}_k^{\lambda}\neq0,\label{subeqn-1:gbinomi} \\
     \widehat {S}_k&\in [-1,1]\quad && \text{if } \hat{\vartheta}_k^{\lambda}=0.\label{subeqn-2:gbinomi}    \end{empheq}
  \end{subequations}
Given any $Y\in\mathcal Y^N$ and $\lambda>0$, Proposition~\ref{prop:signs} shows that there exists one and only one vector of signs~$\widehat S \in \mathds R^d$ such that~$(\hat \vartheta^{\lambda}, \hat S)$ satisfies the KKT conditions for some~$\hat \vartheta^{\lambda} \in \Theta$. The proof of Proposition~\ref{prop:signs} can be found in Section~\ref{proof-prop:signs}.

\begin{proposition}\label{prop:signs} Let $Y \in \mathcal Y^N$ and let the partition function~$\xi$ be strictly convex. Then, there exists a unique~$\widehat S(Y)$ such that for any couple $(\hat \vartheta^{\lambda}, \hat S)$ satisfying the KKT conditions $($cf. Eq.\eqref{eq:gKKTconditions} with $Y$ in Eq.\eqref{eq:KKTY}$)$, it holds that~$\hat S=\widehat S(Y)$. Furthermore, one has 
\[\widehat S(Y):=\frac{1}{\lambda}\mathbf X^{\top}(Y-\sigma(\mathbf X \hat \vartheta^{\lambda})),\]
where $\hat \vartheta^{\lambda}$ is any solution of the generalized linear Lasso as defined in~\eqref{e:generalized_lasso}.
\end{proposition}

\noindent
We define the {\it equicorrelation set} as
\[\widehat M(Y) := \{ k \in [d] \; | \; |\widehat{S}_k(Y)|=1\}.\]
In the following, we will identify the equicorrelation set and the set of predictors with nonzero coefficients $\{k \in [d] \; | \; \hat \vartheta_k^{\lambda}\neq 0\},$ also called ‘selected' model.
Since~$|\widehat {S}_k(Y)| = 1$ for any~$\hat \vartheta_k^{\lambda} \neq 0$, the equicorrelation set does in fact contain all predictors with nonzero coefficients, although it may also include some predictors with
zero coefficients. However, we work in this paper with Assumption~\ref{assumption:equi}, ensuring that the equicorrelation set is precisely
the set of predictors with nonzero coefficients. 
\begin{assumption}\label{assumption:equi} Problem~\eqref{e:generalized_lasso} is non degenerate: $\widehat S(Y) \in \mathrm{relint}\,  \partial \|\cdot\|_1$, where $\mathrm{relint}$ denotes the relative~interior.
\end{assumption} \noindent Let us highlight that this assumption has already been used in the context of GLMs \citep[cf.][Assumption 8]{massias2020dual}, and is common in works on support identification (cf. \cite{candes13,vaiter15}).

For any set of indexes~$M \subseteq[d]$ with cardinality $s$, we denote by~$\Theta_M$ the set of target parameters induced on the support~$M$ namely,
\[\Theta_M:= \{ \vartheta_M \, |\, \vartheta \in \Theta\} \subseteq \mathds R^s.\]
We aim at making inference conditionally on the {\it selection event}~$E_{M}$ defined as
\begin{equation}
\label{def:EM0}
E_{M}:= \left\{ Y \in \mathcal Y^N \; | \; \widehat M(Y) = M \right\}\,,
\end{equation}
namely, the set of all observations $Y$ that induced the same equicorrelation set~$M$ with the generalized linear lasso.

\subsection{A useful characterization of the selection event}

Following the approach of~\cite{sun16}, given some $M\subseteq[d]$ with~$|M|=s$ and~$S_M \in \{-1,+1\}^s$, we first characterize the event
\begin{equation}\label{def:EMSM}E_M^{S_M}:= \{ Y \in E_M \; |\; \widehat S_M(Y) = S_M\},\end{equation}
and we obtain $E_M$ as a corollary by taking a union over all possible vectors of signs~$S_M$. Proposition~\ref{prop:EMSM-deformed-polytope} gives a first description of $E_M^{S_M}$ and its proof is postponed to Section~\ref{proof-prop:EMSM-deformed-polytope}.

\begin{proposition} \label{prop:EMSM-deformed-polytope}
Let us consider $M \subseteq[d]$ with $| M|=s$ and $S_{M} \in \{-1,+1\}^s$. It holds
\begin{align}E_M^{S_M}=\Big\{ Y \in \mathcal Y^N \; |\; \exists \theta \in \Theta_M \text{ s.t. }  (i)&\;\; \mathbf X_{M}^{\top}\left(Y - \sigma(\mathbf X_{M} \theta) \right)=\lambda S_{M}\label{def:ES}\\
(ii)&\;\; \mathrm{sign}(\theta) = S_{M}\notag\\
(iii)&\;\; \left\| \mathbf X^{\top}_{-{M}} \left( Y- \sigma(\mathbf X_{M} \theta) \right)\right\|_{\infty} < \lambda \Big\}\,,\notag
\end{align}
were $\mathbf X_{M}\in\mathds R^{N\times s}$ $($resp. $\mathbf X_{-M}\in\mathds R^{N\times (d-s)})$ is the submatrix obtained from $\mathbf X$ by keeping the columns indexed by $M$ $($resp. its complement$)$.
\end{proposition}

With Proposition~\ref{prop:signs}, we proved the uniqueness of the vector of signs satisfying the KKT conditions as soon as~$\xi$ is strictly convex. By considering additionally that~$\mathbf X_M$ has full column rank, we claim that there exists a unique~$\theta \in \Theta_M$ that satisfies the condition~$(i)$ in the definition of the selection event~$E_M^{S_M}$ (see Eq.\eqref{def:ES}). This statement will be a direct consequence of Proposition~\ref{prop:diffeo-g} (proved in Section~\ref{proof-prop:diffeo-g}) which ensures that the map~$\Xi$ arising in Eq.\eqref{def:ES} and defined by
\begin{align}\label{eq:Xi}
\Xi:\Theta_M &\to \mathds R^s\\
\theta  &\mapsto \mathbf X_M^{\top} \sigma(\mathbf X_M \theta) \notag
\end{align}
is a $\mathcal C^{m}$-diffeomorphism whose inverse is denoted by $\Psi$. 

\begin{proposition} 
\label{prop:diffeo-g} 
We consider that the partition function $\xi$ is strictly convex and we further assume that the set $M \subseteq[d]$ is such that~$\mathbf X_M$ has full column rank. Then~$\Xi$ is a $\mathcal C^{m}$-diffeomorphism from $\Theta_M$ to $\mathrm{Im}(\Xi)=\{ \mathbf X_M^{\top} \sigma(\mathbf X_M\theta) \; | \; \theta \in \Theta_M \}$.
\end{proposition}

\noindent
Using Propositions~\ref{prop:EMSM-deformed-polytope} and~\ref{prop:diffeo-g}, we  are able to provide a new description of the selection event~$E_M^{S_M}$ which can be understood as the counterpart of~\cite[Proposition 4.2]{sun16}.

\begin{thm}\label{thm:ES} Suppose that $\xi$ is strictly convex. Given some $M\subseteq [d]$ with cardinal~$s$ such that $\mathbf X_M$ has full column rank and $S_M\in \{-1,1\}^s$, it holds
\begin{align}E_M^{S_M}= \Big\{ Y \in \mathcal Y^N \; |\; \text{s.t. }& \rho = -\lambda S_M + \mathbf X_M^{\top} Y \text{ satisfies} \label{def2:ES}\\
(a)\;& \rho \in \mathrm{Im}(\Xi)\notag \\
(b)\;&  \mathrm{Diag}(S_M) \Psi(\rho)\geq 0 \notag \\
(c)\; &  \left\| \mathbf X^{\top}_{-M} \left( Y- \sigma(\mathbf X_M  \Psi(\rho)) \right)\right\|_{\infty} < \lambda \Big\}.\notag
\end{align}
\end{thm}
 \noindent {\bf Remark.} In the linear model, $\Xi:\theta \mapsto  \mathbf X_M ^{\top}\mathbf X_M\theta$ has full rank and thus condition~$(a)$ from Eq.\eqref{def2:ES} always holds.

\subsection{Which parameters can be inferred?}
\label{sec:sec:selec-satur}

Once a model 
$M$ has been selected, two different modeling assumptions are generally considered when we derive post-selection inference procedures, see for instance \cite[Section 4]{fithian2014optimal}. This choice appears to be essential since it determines the parameters on which inference is conducted. In the following, we consider the mean value
\begin{equation}
    \label{eq:pistar}
    \pi^*:=\mathds E_{\vartheta^*}[Y]=\sigma( \mathbf X \vartheta^*)\,,
\end{equation}
as the parameter of interest. To support this choice, note that the Bayes predictor in the logistic or the linear model is defined from $\mathds E_{\vartheta^*}[Y]$.

As presented in~\cite{fithian2014optimal}, the analyst should decide whether the model $M$ belongs to the so-called class of {\it saturated models} or {\it selected models}. In the following, we discuss these concepts for arbitrary GLMs and Table~\ref{table:question-asked} summarizes the key concepts.

\renewcommand{\arraystretch}{1.5}
\begin{table}[!ht]
\centering
\begin{tabular}{C{1.65cm}||C{3.1cm}|C{3.7cm}|c}
Model & Selected & Weak selected &  Saturated\\\hline \hline
Assumption & $\sigma^{-1}(\pi^*)\in \mathrm{Im}(\mathbf X_M)$ &$\mathbf X_M^{\top}\pi^*\in \mathrm{Im}(\Xi)$ & None\\\hline
Statistic of interest & $\Psi(\mathbf X_M^{\top} Y)$ &  $\Psi(\mathbf X_M^{\top} Y)$ &  $\mathbf X_M^{\top} Y$\\\hline
Inferred parameter  & $\theta^*\in \Theta_M$ s.t. $\pi^*=\sigma(\mathbf X_M\theta^*)$ &  $\theta^*\in \Theta_M$ s.t. $\pi^*$ and $\sigma(\mathbf X_M\theta^*)$ have the same projections on the column span of $\mathbf X_M$ &  $\mathbf X_M^{\top} \pi^*$ 
\end{tabular}
\caption{Once a model has been selected, we may infer some parameters assuming one of the three modeling: selected model, weak selected model, and saturated model respectively based on the assumptions described in the first row. In this case, inference on the quantities described on the third row can be done from the statistic described in the second row. }
\label{table:question-asked}
\end{table}

\paragraph{The (weak) selected model: Parameter inference.}

In the {\it weak selected model}, we consider that the data have been sampled from the GLM (cf. Eq.\eqref{eq:GLM}) and we assume that the selected model~$M$ is such that 
\begin{equation}
\label{eq:selectedmodel}
\mathbf X_M^{\top}\sigma( \mathbf X \vartheta^*)\in \mathrm{Im}(\Xi)\,,\end{equation}
and recall that $\mathbf X_M^{\top}\pi^*=\mathbf X_M^{\top}\mathds E_{\vartheta^*}[Y]=\mathbf X_M^{\top}\sigma( \mathbf X \vartheta^*)$. This is equivalent to state that there exists some vector~$\theta^* \in \Theta_M$ satisfying 
\begin{equation*}
\mathbf X_M^{\top}\pi^*=\Xi(\theta^*)\,,
\end{equation*}
and recall that $\Xi(\theta^*)=\mathbf X_M^{\top}\sigma(\mathbf X_M\theta^*)$. In this framework, we have the possibility to make inference on the parameter vector $\theta^*:=\Psi(\mathbf X_M^{\top}\pi^*)$ itself. 

In the {\it selected model}, we replace the condition from Eq.\eqref{eq:selectedmodel} by the stronger assumption that there exists $\theta^*\in \Theta_M$ such that 
\begin{equation}
\label{e:selectedmodel2}
\mathbf X_M \theta^* = \mathbf X \vartheta^*.
\end{equation}
This assumption is always satisfied for the global null hypothesis $\vartheta^*=0$ for which the aforementioned condition holds with $\theta^*=0$.

\paragraph{The saturated model: Mean value inference.} 
The assumption from Eq.\eqref{eq:selectedmodel} or~\eqref{e:selectedmodel2} can be understood as too restrictive since the analyst can never check in practice that this condition holds, except for the global null. This is the reason why one may prefer to consider the so-called {\it saturated model} where we only assume that the data have been sampled from the GLM.

In this case it remains meaningful to provide  post-selection inference procedure for transformation of $\pi^*$. A typical choice is to consider linear transformation of~$\pi^*$ and among them, one may focus specifically on transformation of $\mathbf X_M^{\top}\pi^*$. This choice is motivated by remarking that this quantity characterizes the first order optimality condition for the unpenalized MLE $\widehat \theta$ for the design matrix~$\mathbf X_M$ through $\mathbf X_M^{\top}Y=\Xi(\widehat \theta)$, or by considering the example of linear model (as presented below).

\paragraph{The example of the linear model.}

Note that in linear regression, $\sigma=\mathrm{Id}$ and $\Psi:\rho\mapsto\big(\mathbf X_M^{\top}\mathbf X_M\big)^{-1}\rho$. Hence, Eq.\eqref{eq:selectedmodel} is equivalent to Eq.\eqref{e:selectedmodel2} meaning that the selected and the weak selected models coincide. Moreover, in both the saturated and the selected models, we aim at making inference on transformations of $\Psi(\mathbf X_M^{\top}\pi^*)=\mathbf X_M^+\pi^*$ (where $\mathbf X_M^+$ is the pseudo-inverse of $\mathbf X_M$). While in the (weak) selected model, this quantity corresponds to the parameter vector $\theta^*$ satisfying $\pi^*=\mathbf X_M\theta^*$, in the saturated model, it corresponds to the best linear
predictor in the population for design matrix $\mathbf X_M$ in the sense of the squared $L^2$-norm.

\subsection{Inference procedures with SIGLE}
\label{sec:intro-infeprocedures}
\label{sec:intro-conditional-MLE}

In this section, we show how the SIGLE test statistics naturally emerge by establishing a parallel between post selection inference and M-estimation with a misspecified model.

\paragraph{SIGLE statistic in the selected model.} 
In the post-selection paradigm, we work conditional to the selection event~$\{Y \in E_M\}$. The conditional distribution of the observations is a conditional exponential family with the same parameters and sufficient statistics but different support and normalizing constant:
\[\overline {\mathds P}_{\theta}(Y) \propto \mathds 1_{E_M}(Y)  \prod_{i=1}^N h_{v}(y_i)\exp
\Big\{
\frac{y_i\mathbf X_{i,M}\theta-\xi(\mathbf X_{i,M}\theta )}{v}  \Big\}, \,\]
where the symbol $\propto$ means ‘proportional to'. When $E_M = \mathcal Y^N$ ({\it i.e.,} when there is no conditioning), we will simply denote $\overline {\mathds P}_{\theta}$ by $\mathds P_{\theta}$. In the following we will denote by~$\overline {\mathds E}_{\theta}$ (resp. ${\mathds E}_{\theta}$) the expectation with respect to~$\overline {\mathds P}_{\theta}$ (resp. $\mathds P_{\theta}$). In the selected model, we want to conduct inference on~$\theta^*$ (from Eq.\eqref{e:selectedmodel2}) based on the conditional and unpenalized MLE computed on the selected model~$M$, namely
\begin{align}\label{eq:uncondi-MLE-XM}
 \widehat \theta &\in \arg \min_{\theta \in \Theta_M} \mathcal L_{N}(\theta,Z^M),   \qquad \mathcal L_{N}(\theta,Z^M)=\sum_{i=1}^n \left\{ \xi(\mathbf X_{i,M}\theta) -y_i\mathbf X_{i,M}\theta\right\} ,
\end{align}
where~$Z^M=(Y,\mathbf X_M)$ and where~$Y$ is distributed according to~$\overline {\mathds P}_{\theta^*}$. Eq.\eqref{eq:uncondi-MLE-XM} can be understood as a mean-field approximation of the true likelihood where we make the assumption that the $Y_i$'s are independent conditional to the selection event~$\{Y\in E_M\}$. This simplification might make our model misspecified in that
$\overline {\mathds P}_{\theta^*}$ might fall out of the framework of independent Bernoulli trials. The asymptotic properties of the MLE under a misspecified model are well known. First we expect $\widehat \theta$ to be asymptotically consistent for a parameter vector $\overline \theta(\theta^*)$ which minimizes the conditional expected negative log-likelihood defined by
\begin{align} \label{def:thetabar}
 \overline \theta(\theta^*) &\in \arg \min_{\theta \in \Theta_M} \overline {\mathds E}_{\theta^*} \left[\mathcal L_{N}(\theta,Z^M)\right]\,. 
\end{align}
In the following, when there is no ambiguity we will simply denote~$\overline \theta(\theta^*)$ by~$\overline \theta$. 
The density $\mathds P_{\overline \theta}$ can be understood as the projection of the true underlying distribution $\overline {\mathds P}_{\theta^*}$ on the model using the Kullback-Leibler divergence. Second, we expect $\sqrt N (\widehat \theta - \overline \theta)$ to be asymptotically normal with zero mean and covariance matrix $\overline V:= \lim_{N\to \infty} N \overline V_N(\theta^*)$ (provided that the limit exists) where 
\begin{equation} \label{eq:heurisitc-CLT}  \overline V_N(\theta^*):= H_N(\overline \theta)^{-1} \left[  L_N(\overline \theta,Z^M) L_N(\overline \theta,Z^M)^{\top}\right] H_N(\overline \theta)^{-1},\end{equation}
where
\[L_N(\theta,Z^M): =  \frac{\partial \mathcal L_N}{\partial \theta}(\theta,Z^M) = \mathbf X_M^{\top}\big(\sigma(\mathbf X_M \theta)-Y\big),\]
is the score function and where~$H_N(\theta) := \frac{\partial^2 \mathcal L_N}{\partial \theta^2}(\theta,Z^M)$ is the Hessian of the log-likelihood. This result (provided in~\cite[Theorem 3.2]{white82}) holds under some regularity conditions such as the continuous differentiability of the score function and a domination assumption on the Hessian of the log-likelihood. Denoting 
\[ \overline L_N(\theta,\mathbf X_M)= \overline{\mathds E}_{\theta^*}\left[\frac{\partial \mathcal L_{N}}{\partial \theta}(\theta,Z^M)\right],\]
it holds that the conditional unpenalized MLE $\widehat\theta$ and the minimizer $\overline{\theta}$ of the conditional risk satisfy the first order condition

\begin{tabular}{rcrclclc}
$L_N(\widehat \theta,Z^M)=0 $&  i.e. & $\quad \mathbf X_M^{\top}(Y-\pi^{\widehat \theta})=0$ &  $\Leftrightarrow$& $ \Xi(\widehat \theta)=\mathbf X_M^{\top}Y$ & $\Leftrightarrow $ &$\widehat \theta = \Psi(\mathbf X_M^{\top}Y)$, & \refstepcounter{equation}\thetag{\theequation}\label{eq:gradbar0}\\
$\text{and }\overline L_N(\overline \theta,\mathbf X_M)=0$ & i.e. & $\quad \mathbf X_M^{\top}(\overline \pi^{\theta^*}-\pi^{\overline \theta})=0$ & $\Leftrightarrow$ & $ \Xi(\overline \theta) = \mathbf X_M^{\top} \overline \pi ^{\theta^*}$ & $\Leftrightarrow$ & $ \overline \theta = \Psi(\mathbf X_M^{\top}\overline \pi^{\theta^*})$, & \refstepcounter{equation}\thetag{\theequation}\label{eq:gradbar}
\end{tabular}
where~$\pi^{\theta}= \mathds E_{\theta}[Y]=\sigma(\mathbf X_M\theta)$ and~$\overline \pi^{\theta} = \overline {\mathds E}_{\theta}[Y]$. This leads to 
\[\overline V_N( \theta^*) = H_N(\overline \theta)^{-1}\overline G_N^c( \theta^*) H_N(\overline \theta)^{-1},\]
where
\[H_N(\overline \theta)=\mathbf X_M^{\top} \mathrm{Diag}(\sigma'(\mathbf X_M\overline \theta))\mathbf X_M \quad \text{and}\quad \overline G_N^c( \theta^*)=\mathbf X_M ^{\top}\overline{\mathds E}_{\theta^*}\big[(Y-\pi^{\overline \theta})(Y-\pi^{\overline \theta})^{\top}\big]\mathbf X_M.\]

Let us state explicitly that the previous asymptotic considerations hold under specific assumptions that are not satisfied in our setting. Nevertheless, building a bridge between the standard theory of the MLE under model misspecification and our framework of PSI can help us choose a relevant covariance structure to design the SIGLE test statistic. In the rest of this paper, we will consider the following proxy for the covariance matrix of the score~$\overline G_N^c(\theta^*)$:
\[\overline G_N( \theta^*)=\mathbf X_M ^{\top}\mathrm{Diag}\big(\overline \pi^{ \theta^*}\odot (1-\overline \pi^{ \theta^*})\big)\mathbf X_M.\]
$\overline G_N(\theta^*)$ is obtained from $\overline G^c_N(\theta^*)$ by using~$\mathbf X_M^{\top}\overline \pi^{\theta^*}=\mathbf X_M^{\top} \pi^{\overline \theta}$ (cf. Eq.\eqref{eq:gradbar}) and by keeping only the diagonal elements of the covariance matrix~$\overline{\mathds E}_{\theta^*}\big[(Y-\overline \pi^{ \theta^*})(Y-\overline \pi^{ \theta^*})^{\top}\big]$ while setting to zero the off-diagonal entries. Therefore, in the selected model SIGLE relies on the following test statistic
\begin{equation}\label{eq:sigle-stat-1.4}\|[\overline G_N(\theta^*)]^{-1/2}H_N(\overline \theta)(\widehat \theta - \overline \theta)\|_2^2.\end{equation} 
The choice to work with~$\overline G_N(\theta^*)$ rather than~$\overline G_N^c(\theta^*)$ is motivated by several reasons: 
\begin{enumerate}
 \item Working with~$\overline G_N(\theta^*)$ makes the theoretical analysis simpler although the post-selection inference methods proposed in this paper remain valid in the case one uses $\overline G_N^c(\theta^*)$.
 \item Extensive numerical experiments have shown that the power of hypothesis tests using SIGLE with either~$\overline G_N(\theta^*)$ or~$\overline G_N^c(\theta^*)$ is very similar (and some of them are presented in Section~\ref{sec:selective-inference}).
 \item Only~$N$ coefficients need to be estimated to approximate~$\overline G_N(\theta^*)$ as opposed to the~$N^2$ coefficients required to estimate~$\overline G_N^c(\theta^*)$. As a consequence, working with~$\overline G_N(\theta^*)$ might allow to reduce the variance of our estimate of the SIGLE statistic and thus to get closer to the power that would give SIGLE using the unknown quantities~$\overline \pi^{\theta^*}$, $\overline {\theta}(\theta^*)$.
 \end{enumerate}

\medskip
\paragraph{SIGLE statistic in the saturated model.} 
Let us start by introducing some notations. By assuming that $\xi$ is strictly convex, one can compute $\mathbf X\vartheta^*$ from $\pi^*$, allowing us to denote equivalently ${\mathds P}_{\pi^*}\equiv \mathds P_{\vartheta^*}$ with an abuse of notation. Given some set of selected variables~$M\subseteq[d]$ with~$s:=|M|$ and some~$\vartheta^*\in\mathds R^d$, we denote by $\overline{\mathds P}_{\pi^*}$ the distribution of~$Y$ given~$E_M$, namely \[\overline{\mathds P}_{\pi^*}(Y) \propto \mathds 1_{Y\in E_M} \mathds P_{\pi^*}(Y),\]
$\pi^* = \sigma(\mathbf X \vartheta^*)$ and where $\propto$ means equal up to a normalization constant.  In the selected model with $\theta^*\in \Theta_M$ satisfying Eq.\eqref{e:selectedmodel2}, we will also denote $ {\mathds P}_{\pi^*}\equiv  {\mathds P}_{\theta^*}$.  
\medskip

In the saturated, we have already explained that we focus on the statistic~$\mathbf X_M^{\top}Y$. Recalling the definition of~$\Xi$ (cf. Eq.\eqref{eq:Xi}) and using Eq.\eqref{eq:gradbar0}, we get that $\mathbf X_M^{\top}Y=\Xi(\widehat \theta)$. Therefore, one can apply the delta method to convert the heuristic CLT obtained for~$\widehat \theta$ (cf. Eq.\eqref{eq:heurisitc-CLT}) into a similar asymptotic result for~$\mathbf X_M^{\top}Y$. The delta method suggests that $\Xi(\widehat \theta)=\mathbf X_M^{\top}Y$ should be asymptotically normal with mean~$\lim_{N\to \infty} \Xi(\overline \theta)=\mathbf X_M^{\top} \overline \pi^{\theta^*}$ (using Eq.~\eqref{eq:gradbar}) and covariance matrix
\[\lim_{N \to\infty} \nabla \Xi(\overline \theta)^{\top}\overline V_N(\theta^*)\nabla \Xi(\overline \theta) = \lim_{N\to \infty} \overline G_N(\theta^*),\]
where we used that~$\nabla \Xi(\overline \theta)=H_N(\overline \theta)$. A careful reader would note that it makes no sense to refer to~$\theta^*$ in the saturated model. To overcome this issue, one can realize that~$\theta^*$ only appears in the asymptotic description of~$\mathbf X_M^{\top}Y$ through~$\overline \pi^{\theta^*}=\overline{\mathds E}_{\theta^*}[Y]=\overline {\mathds E}_{\pi^*}[Y]$. Therefore, denoting~$\overline \pi^{\theta^*}$ by
\[\overline \pi^{\pi^*}:=\overline {\mathds E}_{\pi^*}[Y],\] the previous discussion suggests that~$\mathbf X_M^{\top}(Y-\overline \pi^{\pi^*})$ should be asymptotically normal with mean~$0$ and covariance matrix~$\lim_{N\to \infty} \overline G_N(\pi^*)$ where
\[\overline G_N(\pi^*):=\mathbf X_M^{\top} \mathrm{Diag}\big(\overline \pi^{\pi^*}\odot (1-\overline \pi^{\pi^*})\big)\mathbf X_M.\]
Therefore, SIGLE in the saturated model relies on the following test statistic
\begin{equation}\label{eq:sigle-stat-1.4-sat}\| [\overline G_N(\pi^*)]^{-1/2}\mathbf X_M^{\top}(Y-\overline \pi^{\pi^*})\|_2^2.\end{equation}

\bigskip
\paragraph{Discussion.} The presentation of the SIGLE statistics in this section naturally gives rise to the following questions $\big(\mathcal Q_k\big)_{k\in[4]}$:

\begin{itemize}
\item {\bf $\mathcal Q_1$: How to use the SIGLE statistics~\eqref{eq:sigle-stat-1.4} and~\eqref{eq:sigle-stat-1.4-sat} in practice?}\par
In most cases, the distribution of the observations conditional to the selection event is unknown and computing~\eqref{eq:sigle-stat-1.4} or~\eqref{eq:sigle-stat-1.4-sat} requires to use sampling methods.

We consider hypothesis tests with pointwise nulls as presented in Table~\ref{table:hypo-testing}. Assuming that we are able to sample state according the $\overline {\mathds P}_{\pi_0^*}$ (resp. $\overline {\mathds P}_{\theta_0^*}$), we can compute estimates $\widetilde G_N(\pi^*_0), \widetilde \pi^{\pi^*_0}$ (resp. $\widetilde V_N(\theta^*_0), \widetilde \theta(\theta^*_0)$) of the unknown quantities $\overline G_N(\pi^*_0)$, $\overline \pi^{\pi^*_0}$ (resp. $\overline V_N(\theta^*_0)$, $\overline \theta(\theta^*_0)$) by sampling from the conditional null distribution $\overline {\mathds P}_{\pi^*_0}$ (resp. $\overline {\mathds P}_{\theta_0^*}$ in the selected model). 

\begin{table}[!ht]
\centering
\begin{tabular}{C{1.6cm}|C{3cm}|C{5cm}}
& Null and alternative & Test statistic \\\hline
Saturated model & $\mathds H_0: \, \{\pi^*=\pi^*_0\}$, $\mathds H_1: \, \{\pi^*\neq\pi^*_0\} $ & $\|\widetilde G_N(\pi^*_0)^{-1/2} (\mathbf X_M^{\top}Y-\mathbf X_M^{\top}\widetilde \pi^{\pi^*_0}) \|_2^2$\\\hline
Selected model&$\mathds H_0: \, \{\theta^*=\theta^*_0\},$ $\mathds H_1: \, \{\theta^*\neq\theta^*_0\} $  & $\|\widetilde V_N(\theta^*_0)^{1/2} (\Psi(\mathbf X_M^{\top}Y)-\widetilde \theta( \theta^*_0))\|_2^2$
\end{tabular}
\caption{Test statistics of SIGLE.}
\label{table:hypo-testing}
\end{table}

\medskip
\begin{changebar}
\noindent \underline{In the case of logistic regression}, we rely on a gradient alignment viewpoint of the selection event to provide in Section~\ref{sec:sampling} an algorithm which allows us to sample from $\overline {\mathds P}_{\pi^*}$ given any $\pi^*$. In Section~\ref{sec:selective-inference}, we present our hypothesis tests in both the saturated and the selected models.
\end{changebar}

\item {\bf $\mathcal Q_2$: What are the asymptotic properties of the SIGLE statistics \eqref{eq:sigle-stat-1.4} and \eqref{eq:sigle-stat-1.4-sat}?}\par
The way the SIGLE statistics have been motivated in this section naturally opens the question of their asymptotic properties. More precisely, can we find conditions ensuring that~$[\overline G_N(\pi^*)]^{-1/2}\mathbf X_M^{\top}(Y-\overline \pi^{\pi^*})$ (resp.$[\overline G_N(\theta^*)]^{-1/2}H_N(\overline \theta)(\widehat \theta - \overline \theta)$ in the selected model) is asymptotically normal? Asymptotic considerations have already been used in the literature to design post-selection inference methods in GLMs such as in~\cite{taylorGLM}. Such approaches often rely on non-rigorous computations conducted under (very) restrictive assumptions. 

\begin{changebar}
\underline{In the case of logistic regression}, we prove conditional central limit theorems (CLTs) for the SIGLE statistics in both the selected and the saturated model. As far as we know, we are the first to provide such results in the field of PSI. Our conditional CLTs hold under conditions that are similar to the ones usually considered in the literature when studying the asymptotic properties of the MLE in high dimensions (cf.\cite{bunea}). \par Furthermore, we provide an extensive comparison between our methods and the one from~\cite{taylorGLM} on both the numerical side (cf. Section~\ref{sec:experiments}) and the theoretical side (cf. Section~\ref{sec:MLEvsBIAS}).
\end{changebar}

\item {\bf $\mathcal Q_3$: Other statistics might have been considered. How the SIGLE statistics from~\eqref{eq:sigle-stat-1.4} and~\eqref{eq:sigle-stat-1.4-sat} perform compared to other approaches?}\par
At the end of Section~\ref{sec:selective-inference}, we show with numerical experiments that SIGLE statistics lead to more powerful testing procedures compared to methods based on other reasonable choices for the test statistics. In Section~\ref{sec:experiments}, we compare our method with state of the art approaches for PSI in logistic regression. 

\item {\bf $\mathcal Q_4$: How the methods of this paper can be interpreted when the model is misspecified from the start?}\par
So far, we have considered the case where the observed data~$y_i \in \mathcal Y$ has indeed by generated from the GLM presented in Section~\ref{sec:1.1}. Can we extend the methods presented in this paper when we remove this assumption? In Section~\ref{apdx:advanced_discussion_misspe}, we consider that the~$y_i$'s are i.i.d. and distributed according to an arbitrary probability distribution~$\mathds P$.
\end{itemize}

\begin{changebar}

\end{changebar}

\subsection{Related works}

In the Gaussian linear model with a known variance, the distribution of the linear transformation $\eta^{\top}Y$ (with $\eta^{\top}=e_k^{\top}\mathbf X_M^+$) is a truncated Gaussian conditionally on $E_M^{S_M}$ and $\mathrm{Proj}_{\eta}^{\perp}(Y)$. This explicit formulation of the conditional distribution allows to conduct exact post-selection inference procedures~\citep[cf.][Section 4]{fithian2014optimal}. However, when the noise is assumed to be Gaussian with an unknown variance, one needs to also condition on $\|Y\|^2$ which leaves insufficient information about $\theta^*_k$ to carry out a meaningful test in the saturated model \citep[cf.][Section 4.2]{fithian2014optimal}.

Outside of the Gaussian linear model, there is little hope to obtain a useful exact characterization of the conditional distribution of some transformation of~$\mathbf X_M^{\top}Y$. 
In the following, we sketch a brief review of this literature, see references therein for further works on this subject.
\begin{itemize}
\item Linear model but non-Gaussian errors.\\  
 Let us mention for example~\cite{tian17asympto,tib18} where the authors consider the linear model but relaxed the Gaussian distribution assumption for the error terms. They prove that the response variable is asymptotically Gaussian so that applying the well-oiled machinery from~\cite{sun16} gives asymptotically valid post-selection inference methods.
\item GLM with Gaussian errors.\\
\cite{xiang2020} consider generalized linear models with Gaussian noise and can then immediately apply the polyhedral lemma to the appropriate transformation of the response.
\end{itemize}
We classify existing works with Table~\ref{tab:review}.
\begin{table}[ht!]
\centering
\begin{tabular}{c||c|c}
Noise & Linear Model & GLM\\\hline\hline
Gaussian  & \cite{sun16}  & \cite{xiang2020}\\\hline
\multirow{2}{*}{Non-Gaussian} &\cite{tian17asympto} and & SIGLE (this paper) and\\
 & \cite{tib18} & \cite{taylorGLM}
\end{tabular}
\caption{Positioning of SIGLE (this paper) among some pioneering works on PSI in GLMs.}
\label{tab:review}
\end{table}

\pagebreak[3]

One important challenge that remains so far only partially answered is the case of GLMs without Gaussian noise, such as in logistic regression. In~\cite{fithian2014optimal}, the authors derive powerful unbiased selective tests and confidence intervals among all selective level-$\alpha$ tests for inference in exponential family models after arbitrary selection procedures. Nevertheless, their approach is not well-suited to account for discrete response variable as it is the case in logistic regression. In Section~6.3 of the former paper, the authors rather encourage the reader to make use of the procedure proposed by~\cite{taylorGLM} in such context. Both this paper and~\cite{taylorGLM} are tackling the problem of post selection inference in the logistic model.

\subsection{Contributions and organization of the paper}

\paragraph{SIGLE for an arbitrary GLM (Sec.\ref{sec:intro}).}
\begin{enumerate}
\item We provide a new formulation of the selection event in GLMs shedding light on the $\mathcal C^m$-diffeomorphism~$\Psi$ that carries the geometric information of the problem (cf. Theorem~\ref{thm:ES}). $\Psi$ allows us to define rigorously the notions of selected/saturated models for arbitrary GLM (cf. Sec.\ref{sec:sec:selec-satur}).
\item  We provide a new perspective on post-selection inference in the selected model for GLMs through the conditional MLE approach of which~$\Psi$ is a key ingredient (cf. Sec.\ref{sec:intro-infeprocedures}). 
\item We introduce the SIGLE statistics in both the saturated and the selected model. Computing these statistics and calibrating the SIGLE hypothesis testing require to be able to sample from the distribution of the observations conditional to the selection event (cf. Sec.\ref{sec:intro-infeprocedures}).
\end{enumerate} 

\paragraph{SIGLE for the Sparse Logistic Regression (SLR) (from Sec.\ref{sec:selective-inference}).}
\begin{enumerate}
\item[4.] We describe in details the way to use SIGLE in practice in both the selected in the saturated model (cf. Sec.\ref{sec:selective-inference}).
\item[5.] In the context of the SLR, we introduce a new sampling method allowing to compute the SIGLE statistics and to calibrate our methods (cf. Sec.\ref{sec:sampling}).
\item[6.]  We provide an extensive comparison between this paper and the heuristic from~\cite{taylorGLM} which is currently considered the best to use in the context of SLR \citep[cf.][Section 6.3]{fithian2014optimal}, as far as we know. The methods are compared both on the numerical side (cf. Sec.\ref{sec:experiments}) and the theoretical side (cf. Sec.\ref{sec:MLEvsBIAS}).
\item[7.] Going back to the motivation behind the choice of the SIGLE statistics, we study the asymptotic properties of the conditional MLE. We provide conditions under which conditional CLTs hold (cf. Sec.\ref{sec:CLT}). 
\end{enumerate}

\paragraph{Outline.}  In this paper, we focus specifically on the SLR. We start by describing the SIGLE hypothesis testing methods in this context in Section~\ref{sec:selective-inference}. In Section~\ref{sec:sampling}, we rely on a {\it gradient-alignment} viewpoint on the selection event to design a simulated annealing algorithm which is proved--for an appropriate cooling scheme--to provide iterates whose distribution is asymptotically uniform on the selection event. In Section~\ref{sec:experiments}, we present the results of our simulations. We conclude in Section~\ref{sec:asymptotic} by providing two conditional central limit theorems.

\paragraph{Notations.}For any set of indexes $M \subseteq[d]:=\{1,\dots,d\}$ and any vector~$v$, we denote by~$v_M$ the subvector of~$v$ keeping only the coefficients indexed by~$M$, namely~$v_M = (v_k)_{k\in M}$. Analogously, $v_{-M}$ will refer to the subvector~$(v_k)_{k\notin M}$. $|M|$ denotes the cardinality of the finite set~$M$. For any~$x\in \mathds R^d$, $\|x\|_{\infty}:=\sup_{i\in[d]}|x_i|$ and for any $p\in[1,\infty)$, $\|x\|_p^p:=\sum_{i\in[d]}x_i^p$. For any~$A\in \mathds R^{d\times p}$, we define the Frobenius norm of~$A$ as~$\|A\|_F:=(\sum_{i\in[d],j\in[p]}A_{i,j}^2)^{1/2}$ and the operator norm of~$A$ as~$\|A\|:=\sup_{x\in \mathds R^p, \|x\|_2=1}\|Ax\|_2$. We further denote by~$A^+$ the pseudo-inverse of~$A$. Considering that~$A$ is a symmetric matrix, $\lambda_{\min}(A)$ and~$\lambda_{\max}(A)$ will refer respectively to the minimal and the maximal eigenvalue of~$A$. $\odot$ denotes the Hadamard product namely for any~$A,B\in\mathds R ^{d\times p}$, $A\odot B:=(A_{i,j}B_{i,j})_{i\in[d],j\in[p]}$. By convention, when a function with real valued arguments is applied to a vector, one need to apply the function entrywise. $\mathrm{Id}_d\in\mathds R^{d\times d}$ will refer to the identity matrix and~$\mathcal N(\mu,\Sigma)$ will denote the multivariate normal distribution with mean~$\mu\in\mathds R^d$ and covariance matrix~$\Sigma$. For any $x \in \mathds R^d$, $R>0$ and for $p\in[1,\infty]$, we define $\mathds B_{p}(x,R)=\{z \in \mathds R^d \, |\, \|z\|_p\leq R\}$. \\
Let us finally recall that given some set of selected variables $M\subseteq[d]$ with $s:=|M|$ and some $\vartheta^*\in\mathds R^d$, we denote by $\overline{\mathds P}_{\pi^*}$ the distribution of $Y$ conditional on $E_M$, namely \[\overline{\mathds P}_{\pi^*}(Y) \propto \mathds 1_{Y\in E_M} \mathds P_{\pi^*}(Y),\]
$\pi^* = \sigma(\mathbf X \vartheta^*)$ and where $\propto$ means equal up to a normalization constant. By assuming that $\xi$ is strictly convex, one can compute $\mathbf X\vartheta^*$ from $\pi^*$, allowing us to denote equivalently ${\mathds P}_{\pi^*}\equiv \mathds P_{\vartheta^*}$ with an abuse of notation. In the selected model with $\theta^*\in \Theta_M$ satisfying Eq.\eqref{e:selectedmodel2}, we will also denote $ {\mathds P}_{\pi^*}\equiv  {\mathds P}_{\theta^*}$.

\section{Comprehensive description of SIGLE for SLR}
\label{sec:selective-inference}

From this section, we consider the case of the logistic regression where we recall that $Y=(y_i)_{i\in[N]}$ and for all $i\in[N]$, $y_i\sim \mathrm{Ber}(\pi^*_i)$ with $\pi^*=\sigma(\mathbf X \vartheta^*)$. As already explained in the introduction, the SIGLE statistics are motivated by the conditional CLTs provided in details in Section~\ref{sec:asymptotic}. In this section, we describe our methods.

\paragraph{SIGLE in the saturated model.}
Given some~$\pi^*_0 \in \mathds R^N$, we consider the hypothesis test with null and alternative hypotheses defined by
\begin{equation}\label{eq:test}\mathds H_0 \; :\; \{\pi^*=\pi^*_0\}\quad \text{and}\quad \mathds H_1 \; :\; \{\pi^* \neq \pi^*_0\}.\end{equation}
The statistics given by the CLT from Theorem~\ref{prop:CLT} (cf.  Section~\ref{sec:asymptotic}) naturally leads us to introduce the ellipsoid $W_N$ given by \[W_N=\left\{ Y \in \{0,1\}^N \; | \; \left\|[\overline G_N(\pi^*_0)]^{-1/2} \mathbf X_M^{\top}\left(Y-\overline {\pi}^{\pi^*_0}\right)\right\|^2_2 \geq w_{N,1-\alpha} \right\},\]
where
\begin{itemize}
\item $w_{N,1-\alpha} $ is the quantile of order~$1-\alpha$ of the SIGLE statistic \[\left\|[\overline G_N(\pi^*_0)]^{-1/2} \mathbf X_M^{\top}\left(Y-\overline {\pi}^{\pi^*_0}\right)\right\|^2_2,\]
\item $\overline G_N (\pi^*):=\mathbf X_M^{\top} \mathrm{Diag}((\overline \sigma^{\pi^*})^2) \mathbf X_M$ with $(\overline \sigma^{\pi^*})^2:= \overline \pi^{\pi^*} \odot (1-\overline \pi^{\pi^*})$.
\end{itemize}

If~$\overline{\mathds P}_{\pi_0^*}$ was nice enough, we could hope to easily compute $i)$~$\overline \pi^{\pi^*_0}$ and then~$\overline G_N(\pi_0^*)$ and $ii)$ $w_{N,1-\alpha}$. Contrary to the linear model where the conditional distribution is known to be a truncated Gaussian, we do not have such a characterization of $\overline{\mathds P}_{\pi_0^*}$ in SLR. As a consequence, we propose in the paper two different ways to sample state in the selection event and to estimate the parameters~$\overline \pi^{\pi^*_0}$ and $w_{N,1-\alpha}$ in order to approximate the rejection region~$W_N$. Both methods are presented in Section~\ref{sec:sampling}. The first sampling approach is a simple rejection sampling method. This method is particularly appropriate when the number of features~$d$ is small. When~$d$ is getting large, another sampling method is needed and we introduce in this paper the SEI-SLR algorithm. From Proposition~\ref{prop:uniform-distribution} (cf. Section~\ref{sec:sampling}), we know that under an appropriate cooling scheme, the asymptotic distribution of the states visited by the SEI-SLR algorithm (cf. Algorithm~\ref{algo1}) is the uniform distribution on the selection event. We deduce that under the null, we are able to estimate~$\overline \pi^{\pi^*_0}$ and thus $\overline G_N(\pi^*_0)$. Algorithm~\ref{algo:sigle_sat} describes the testing procedure when we sample states in the selection event using the SEI-SLR algorithm. \par
When the states~$(Y^{(t)})_{t\geq 1}$ in step 2. of Algorithm~\ref{algo:sigle_sat} are sampled using the rejection method instead of the SEI-SLR algorithm, one only needs to change the way~$\widetilde \pi^{\pi^*_0}$ and~$\zeta_{N,T}$ are computed by using
\[\widetilde \pi^{\pi_0^*} =\frac1T \sum_{t=1}^T Y^{(t)} \quad \text{and} \quad \zeta_{N,T} =\frac1T \sum_{t=1}^T \mathds 1_{Y^{(t)} \in \widetilde W_N}. \]

\begin{algorithm}
\begin{algorithmic}[1]
 	\STATE {\bf Input:}~$\mathbf X \in \mathds R^{N\times d}$, $Y\in\mathds R^N$, $\lambda>0$, $\alpha \in (0,1)$.\\
$(\mathbf X,Y,\lambda)$ characterizes the selection event $E_M$ (cf. Eq.\eqref{def:EM0}). 
  \STATE Sample states $(Y^{(t)})_{t\geq 1}$ uniformly distribution on~$E_M$ using the SEI-SLR algorithm (cf. Algorithm~\ref{algo1}).\\
  \STATE Compute:\\
 	$\quad$ - $\displaystyle \widetilde \pi^{\pi^*_0} =  \frac{\sum_{t=1}^T \mathds P_{\pi^*_0}(Y^{(t)}) Y^{(t)}}{\sum_{t=1}^T \mathds P_{\pi^*_0}(Y^{(t)})},$\\
  $\quad$ - $\widetilde G_N = \mathbf X_M^{\top} \mathrm{Diag}\left( \widetilde \pi^{\pi^*_0}\odot (1-\widetilde \pi^{\pi^*_0})\right) \mathbf X_M,$
  \\
  $\quad$ - $\widetilde w_{N,1-\alpha}$ which is the quantile of order $1-\alpha$ of the sequence $\left( \, \left\|\widetilde G_N^{-1/2}\mathbf X_M^{\top}\left( Y^{(t)}- \widetilde \pi^{\pi_0^*}\right)\right\|^2_2\, \right)_{t\geq 1}$.
  \STATE Define $\widetilde W_N:= \left\{  Y \in \{0,1\}^N \; | \; \left\|\widetilde G_N^{-1/2}\mathbf X_M^{\top}\left( Y- \widetilde \pi^{\pi_0^*}\right)\right\|^2_2 > \widetilde w_{N,1-\alpha} \right\}.$
   \STATE Reject the null hypothesis $\mathds H_0$ when \[ \zeta_{N,T}:= \frac{\sum_{t=1}^T \mathds P_{\pi^*_0}(Y^{(t)}) \mathds 1_{Y^{(t)}\in  \widetilde W_N}}{\sum_{t=1}^T \mathds P_{\pi^*_0}(Y^{(t)})} > \alpha.\]
\end{algorithmic}
\caption{SIGLE in the saturated model.}
\label{algo:sigle_sat}
\end{algorithm}

\paragraph{SIGLE in the selected model.}
Given some~$\theta^*_0 \in \mathds R^s$, we consider the hypothesis test with null and alternative hypotheses defined by
\begin{equation}
\label{eq:test-selected}
\mathds H_0 \; :\; \{\theta^*=\theta^*_0\}\quad \text{and}\quad \mathds H_1 \; :\; \{\theta^* \neq \theta^*_0\}\,.
\end{equation}
The statistic given by the CLT from Theorem~\ref{thm:MLEasymptotic} (cf.  Section~\ref{sec:asymptotic}) naturally leads us to introduce the ellipsoid $W_N$ given by 

\hspace*{-0.6cm}
\setlength\tabcolsep{1.5pt}
\begin{tabular}{rll}
\multirow{2}{*}{$W_N:=\Bigg\{Y \in \{0,1\}^N \; \Bigg| \;$} & $\diamond \, \, \mathbf X_M^{\top} Y \in \mathrm{Im}(\Xi)$&\multirow{2}{2mm}{$\Bigg\}$,}\\
&$\diamond \, \, \left\|[\overline G_N(\theta^*_0)]^{-1/2}H_N(\overline \theta(\theta_0^*))\left( \Psi(\mathbf X_M^{\top}Y)-\overline \theta(\theta^*_0)\right)\right\|^2_2 >  w_{N,1-\alpha}$  &
\end{tabular}\\
where
\begin{itemize}
\item $w_{N,1-\alpha} $ is the quantile of order~$1-\alpha$ of the SIGLE statistic \[ \left\|[\overline G_N(\theta^*_0)]^{-1/2}H_N(\overline \theta(\theta_0^*))\left( \Psi(\mathbf X_M^{\top}Y)-\overline \theta(\theta^*_0)\right)\right\|^2_2,\]
\item $\displaystyle H_N(\theta) :=\mathbf X_M^{\top} \mathrm{Diag}(\sigma'(\mathbf X_M  \theta)) \mathbf X_M=\mathbf X_M^{\top} \mathrm{Diag}((\sigma^{\theta})^2) \mathbf X_M$ is the Fisher information matrix with~$(\sigma^{\theta})^2:=  \pi^{\theta}\odot (1- \pi^{\theta})$ and $\pi^{\theta}=\mathds E_{\theta}[Y]$,
\item $\displaystyle \overline G_N (\theta^*):=\mathbf X_M^{\top} \mathrm{Diag}((\overline \sigma^{\theta^*})^2) \mathbf X_M$ is the natural counterpart of the Fisher information matrix~$H_N(\theta^*)$ when we work under the conditional distribution~$\overline {\mathds P}_{\theta^*}$ with~$(\overline \sigma^{\theta^*})^2:= \overline \pi^{\theta^*}\odot (1-\overline \pi^{\theta^*})$, $\overline \pi^{\theta^*}=\overline{\mathds E}_{\theta^*}[Y].$
\end{itemize}

We rely - as in the saturated model - on the SEI-SLR algorithm or the rejection sampling method (cf. Section~\ref{sec:sampling}) to estimate the parameters~$\overline \pi^{\theta^*_0}$ and $w_{N,1-\alpha}$ in order to approximate the rejection region~$W_N$. The SIGLE procedure in the selected model in presented in Algorithm~\ref{algo:sigle_sat} when the SEI-SLR algorithm is used.

When the states~$(Y^{(t)})_{t\geq 1}$ in step 2. of Algorithm~\ref{algo:sigle_sel} are sampled using the rejection method instead of the SEI-SLR algorithm, one only needs to change the way~$\widetilde \pi^{\theta^*_0}$ and~$\zeta_{N,T}$ are computed by using
\[\widetilde \pi^{\theta_0^*} =\frac1T \sum_{t=1}^T Y^{(t)} \quad \text{and} \quad \zeta_{N,T} =\frac1T \sum_{t=1}^T \mathds 1_{Y^{(t)} \in \widetilde W_N}. \]

\begin{algorithm}
\begin{algorithmic}[1]
 	\STATE {\bf Input:}~$\mathbf X \in \mathds R^{N\times d}$, $Y\in\mathds R^N$, $\lambda>0$, $\alpha \in (0,1)$.\\
$(\mathbf X,Y,\lambda)$ characterizes the selection event $E_M$ (cf. Eq.\eqref{def:EM0}). 
  \STATE Sample states $(Y^{(t)})_{t\geq 1}$ uniformly distribution on~$E_M$ using the SEI-SLR algorithm (cf. Algorithm~\ref{algo1}).\\
  \STATE Compute:\\
 	$\quad$ - $\displaystyle \widetilde \pi^{\theta^*_0} =  \frac{\sum_{t=1}^T \mathds P_{\theta^*_0}(Y^{(t)}) Y^{(t)}}{\sum_{t=1}^T \mathds P_{\theta^*_0}(Y^{(t)})} ,$\\
  $\quad$ - $ \widetilde \theta = \Psi(\mathbf X_M^{\top} \widetilde \pi^{\theta^*_0}),$\\
  $\quad$ - $\widetilde G_N = \mathbf X_M^{\top} \mathrm{Diag}\left( \widetilde \pi^{\theta^*_0}\odot (1-\widetilde \pi^{\theta^*_0})\right) \mathbf X_M,$
  \\
  $\quad$ - $\widetilde w_{N,1-\alpha}$ which is the quantile of order $1-\alpha$ of the sequence $\left( \,  \left\|\widetilde G_N^{-1/2}H_N(\widetilde \theta)\left( \Psi(\mathbf X_M^{\top}Y^{(t)})-\widetilde \theta\right)\right\|^2_2\, \right)_{t\geq 1}$.
  \STATE Define $\widetilde W_N:= \left\{  Y \in \{0,1\}^N \; | \;  \left\|\widetilde G_N^{-1/2}H_N(\widetilde \theta)\left( \Psi(\mathbf X_M^{\top}Y)-\widetilde \theta\right)\right\|^2_2 > \widetilde w_{N,1-\alpha} \right\}.$
   \STATE Reject the null hypothesis $\mathds H_0$ when \[ \zeta_{N,T}:= \frac{\sum_{t=1}^T \mathds P_{\theta^*_0}(Y^{(t)}) \mathds 1_{Y^{(t)}\in  \widetilde W_N}}{\sum_{t=1}^T \mathds P_{\theta^*_0}(Y^{(t)})} > \alpha.\]
\end{algorithmic}
\caption{SIGLE in the selected model.}
\label{algo:sigle_sel}
\end{algorithm}
The careful reader can notice that Algorithm~\ref{algo:sigle_sel} requires to compute efficiently~$\Psi(\mathbf X_M^{\top} \pi)$ for any~$\pi \in [0,1]^N$. In the specific case where $\pi=Y\in\{0,1\}^N$, we know that~$\Psi(\mathbf X_M^{\top}Y$ is the conditional MLE (cf. Eq.~\eqref{eq:gradbar0}) and thus can be computed using the usual Iterative Reweighted Least Squares algorithm. For an arbitrary~$\pi\in[0,1]^N$ (such as in step 3. of Algorithm~\ref{algo:sigle_sel} to compute~$\widetilde \theta$), we need to use another approach. In Section~\ref{apdx:inverting-PSI} of the Appendix, we describe in details our gradient descent-based method to compute~$\Psi(\mathbf X_M^{\top} \pi)$ which proved to be extremely accurate in our numerical experiments.

\paragraph{Discussion regarding the choice of the SIGLE statistic.}

As explained in Section~\ref{sec:intro-infeprocedures}, the SIGLE statistics can be motivated by making a connection between PSI and asymptotic properties of the MLE with model misspecification. Let us present a numerical experiment providing an additional support for the choice of the SIGLE statistics. We consider the hypothesis test in the saturated model presented in Table~\ref{table:hypo-testing} with $\pi^*_0=\frac12 \mathds 1_N.$
\medskip

We consider a design matrix $\mathbf X \in \mathds R^{100\times 10}$ where the entries are i.i.d. and sampled from a standard normal distribution. We use a regularization parameter $\lambda =5$. We work with the following three test statistics:
\begin{itemize}
\item the SIGLE statistic: $\left\|[\overline G_N(\pi^*_0)]^{-1/2} \mathbf X_M^{\top}\left(Y-\overline {\pi}^{\pi^*_0}\right)\right\|^2_2$,
\item the SIGLE correlated statistic: $\left\|[\overline G_N^c(\pi^*_0)]^{-1/2} \mathbf X_M^{\top}\left(Y-\overline {\pi}^{\pi^*_0}\right)\right\|^2_2$ where 
\[\overline G_N^c(\pi^*_0)=\mathbf X_M^{\top} \overline {\mathds E}_{\pi^*_0}\big[ (Y-\overline \pi^{\pi^*_0}) (Y-\overline \pi^{\pi^*_0})^{\top}\big] \mathbf X_M,\]
\item the logistic unconditional Fisher statistic: $\left\|[H_N(\pi^*_0)]^{-1/2} \mathbf X_M^{\top}\left(Y-\pi^*_0\right)\right\|^2_2$.
\end{itemize}

We calibrate each testing procedure by sampling under the null distribution. Figure~\ref{fig:power-different-stats}.$(a)$ presents the cumulative distribution function of the p-values obtained considering the alternative $\pi^* = \sigma(\mathbf X \vartheta^*)$ with $\vartheta^* = 0.2\times \mathds 1_d$ for the different tests. We see that the SIGLE statistic leads to more powerful tests compared to the logistic unconditional Fisher statistic. Moreover, the SIGLE statistic and the SIGLE correlated statistic give similar result as already explained in Section~\ref{sec:intro-infeprocedures}.

Figure~\ref{fig:power-different-stats}.$(b)$ shows the pdf of the SIGLE statistic under the null and the pdf of the closer $\chi^2$ distribution, in the sense that we chose the degree of freedom for the $\chi^2$ distribution that gives the smallest $L^2$ error between the $\chi^2$ quantiles and the SIGLE's quantiles. It appears that this optimal degree of freedom is $14$. Figure~\ref{fig:power-different-stats}.$(b)$ makes clear that the SIGLE statistic is not distributed as a $\chi^2$ random variable contrary to what our conditional CLT from Section~\ref{sec:asymptotic} is suggesting. The obvious reason is that our conditional CLT from Section~\ref{sec:asymptotic} holds only under restrictive conditions that are nonetheless standard in the literature (cf. \cite{bunea}). This is one reason motivating the calibration of the SIGLE procedures by sampling under the null. Let us highlight that this is not restrictive in the sense that the computation of the SIGLE statistics require anyway to sample under the null in order to estimate both $\overline {\pi}^{\pi^*_0}$ and $\overline G_N(\pi^*_0)$. 

\medskip
We conducted the same analysis in the selected model working with the following three test statistics:
\begin{itemize}
\item the SIGLE statistic: $\left\|[\overline G_N(\theta^*_0)]^{-1/2} H_N(\theta^*_0)\left( \Psi(\mathbf X_M^{\top}Y)-\overline {\theta}(\theta^*_0)\right)\right\|^2_2$,
\item the SIGLE correlated statistic: $\left\|[\overline G_N^c(\theta^*_0)]^{-1/2} H_N(\theta^*_0)\left( \Psi(\mathbf X_M^{\top}Y)-\overline {\theta}(\theta^*_0)\right)\right\|^2_2$, where 
\[\overline G_N^c(\theta^*_0)=\mathbf X_M^{\top} \overline {\mathds E}_{\theta^*_0}\big[ (Y-\overline \pi^{\theta^*_0}) (Y-\overline \pi^{\theta^*_0})^{\top}\big] \mathbf X_M,\]
\item the logistic unconditional Fisher statistic: $\left\|[H_N(\theta^*_0)]^{-1/2} \left(\Psi(\mathbf X_M^{\top}Y)-\theta^*_0\right)\right\|^2_2$.
\end{itemize}
The results are presented in Figures~\ref{fig:power-different-stats}.$(c)$ and~$(d)$ with similar conclusions.

\begin{figure}[ht!]
\vskip 0.2in
\begin{center}
 \centering
    \begin{subfigure}[b]{0.49\textwidth}
        \centering
        \includegraphics[width=\textwidth]{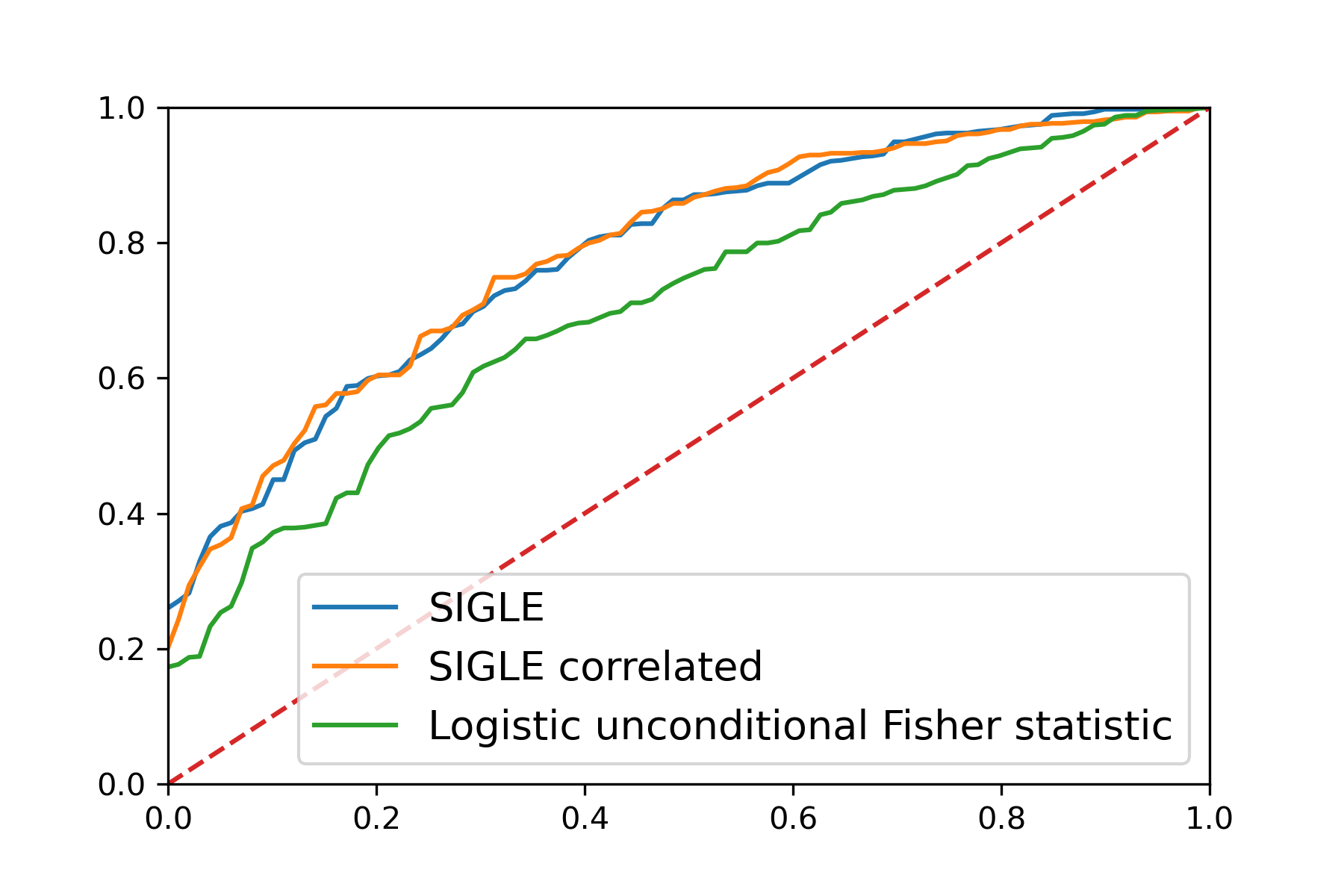}
        \caption[]%
        {{\small CDF of p-values for the alternative $\pi^* = \sigma(\mathbf X \vartheta^*)$ with $\vartheta^* = 0.2 \times \mathds 1_d$.}}    
    \end{subfigure}
    \hfill
     \centering
    \begin{subfigure}[b]{0.49\textwidth}
        \centering
    \includegraphics[width=\textwidth]{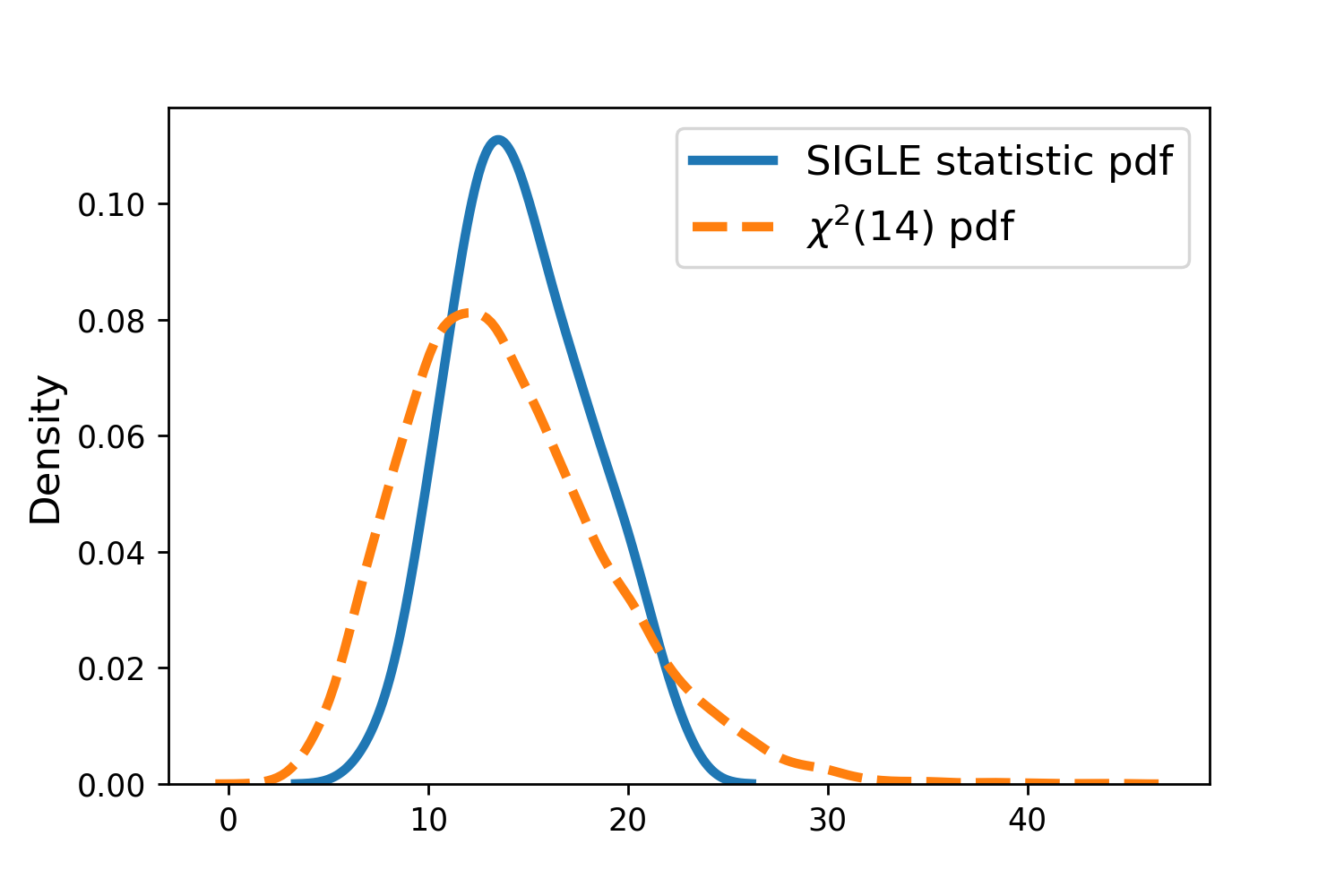}
        \caption[]%
        {{\small Pdfs of the SIGLE statistic under the null and of the $\chi^2(14)$ distribution.}}    
    \end{subfigure}

    \begin{subfigure}[b]{0.49\textwidth}
        \centering
        \includegraphics[width=\textwidth]{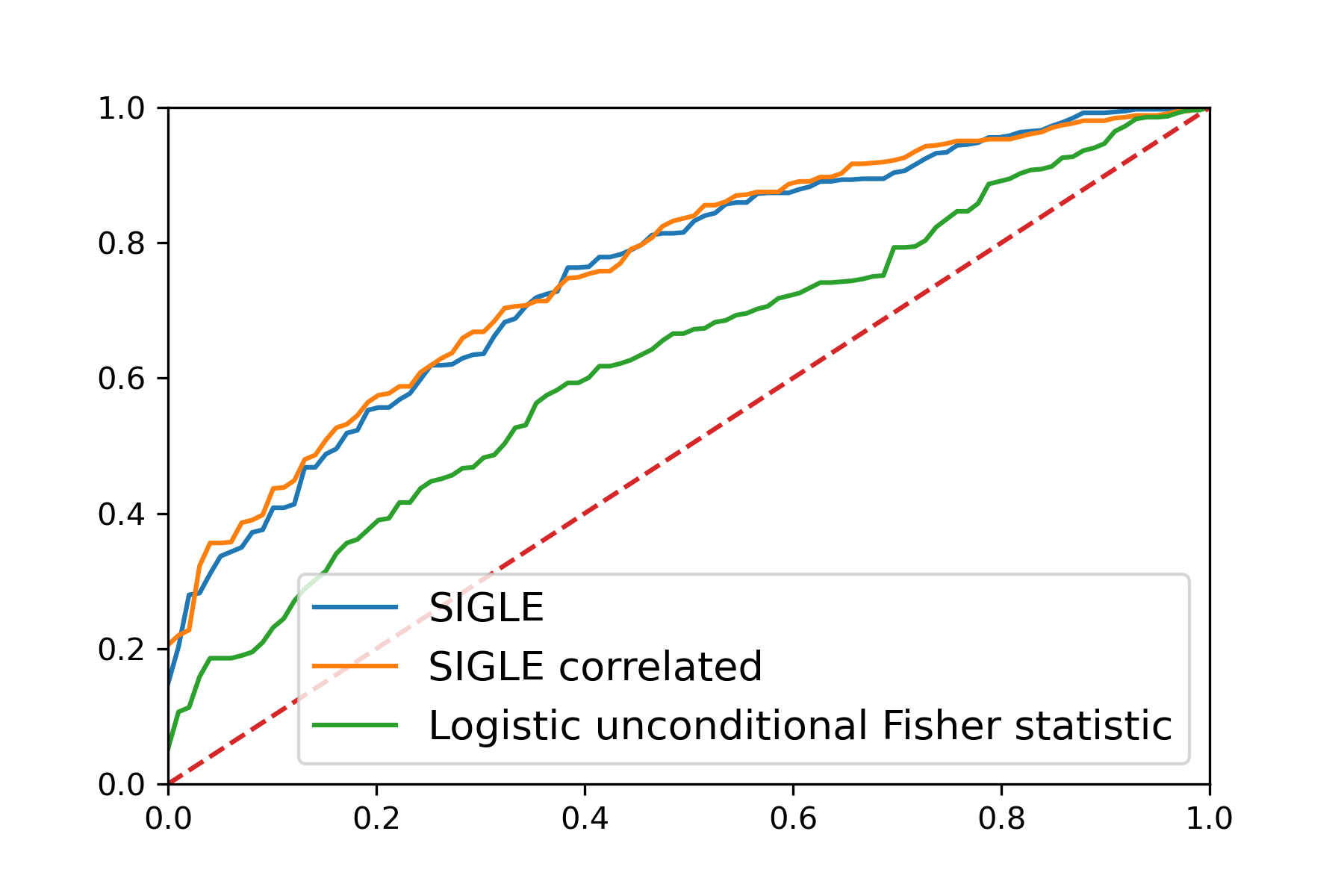}
        \caption[]%
        {{\small CDF of p-values for the alternative $\pi^* = \sigma(\mathbf X \vartheta^*)$ with $\vartheta^* = 0.2 \times \mathds 1_d$.}}    
    \end{subfigure}
    \hfill
     \centering
    \begin{subfigure}[b]{0.49\textwidth}
        \centering
    \includegraphics[width=\textwidth]{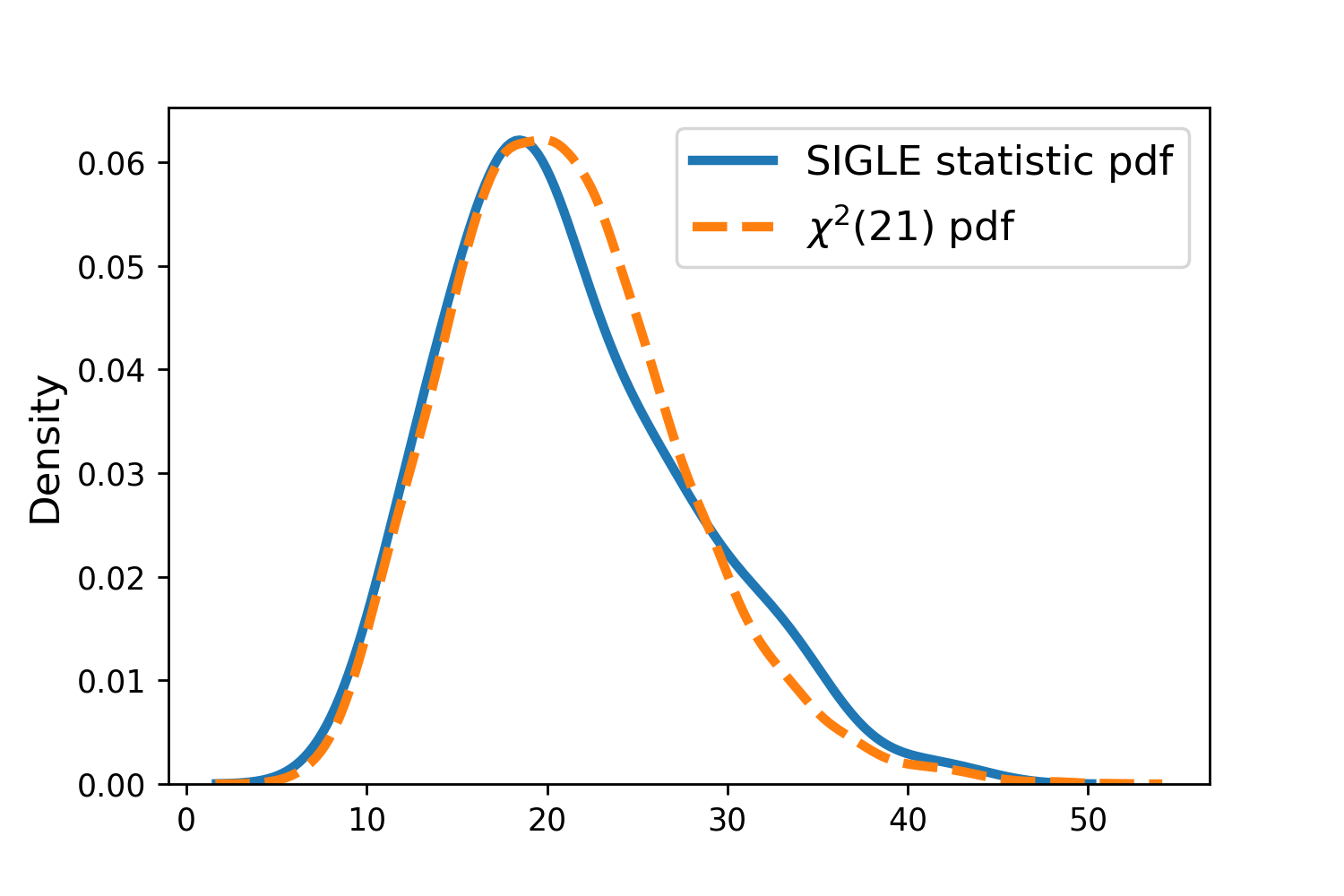}
        \caption[]%
        {{\small Pdfs of the SIGLE statistic under the null and of the $\chi^2(21)$ distribution.}}    
    \end{subfigure}
\caption{Figure $(a)$ (resp. $(c)$) shows the CDF of p-values for the alternative $\pi^* = \sigma(\mathbf X \vartheta^*)$ with $\vartheta^* = 0.2 \times \mathds 1_d$ for the test using the SIGLE statistic, the SIGLE correlated statistic and the logistic unconditional Fisher statistic in the saturated (resp. selected) model. Figure $(b)$ (resp. $(d)$) presents the probability density function (pdf) of the SIGLE statistic under the null and the one of a $\chi^2$ distribution with $14$ (resp. $21$) degrees of freedom in the saturated (resp. selected) model.}
\label{fig:power-different-stats}\end{center}
\vskip -0.2in
\end{figure}

\section{Sampling from the conditional distribution}
\label{sec:sampling}
 
Let us recall that we focus on the case of the logistic regression.
We propose two different approaches to compute quantities of the form~$\overline {\mathds E}_{\vartheta^*}[h(Y)]$ for some map~$h:\{0,1\}^N\to \mathds R.$\\
The first one is a simple rejection sampling method that can be used when carrying simple hypothesis testing as presented in Table~\ref{table:hypo-testing}. In this situation, one can sample states from $\mathds P_{\pi^*_0}$ in the saturated model (resp. $\mathds P_{\sigma(\mathbf X_M\theta^*_0)}$ in the selected model) while keeping only the ones leading to the selected support $M$. By construction, the distribution of the saved states is precisely $\overline{\mathds P}_{\pi^*_0}$ (resp. $\mathds P_{\sigma(\mathbf X_M\theta^*_0)}$). This method is particularly appropriate when the number of features~$d$ is small since the number of possible selected support for the lasso solution is exponential in~$d$. When~$d$ is getting large or when we want to derive a confidence region, another sampling method is needed.\\
In this section, we present an algorithm based on a simulated annealing approach that is proved to sample states~$Y^{(t)}$ uniformly distributed on the selection event~$E_M$ for any $M\subseteq[d]$ with cardinality~$s$ in the asymptotic regime as $t\to\infty$. Contrary to the rejection sampling method, this simulated-annealing based algorithm can be used to compute expectations of the form~$\overline {\mathds E}_{\vartheta^*}[h(Y)]$ regardless of the inference procedure conducted or when~$d$ is large. Nevertheless, let us point out that this approach requires an appropriate tuning of some parameters, and the convergence guarantees are only asymptotic. An extensive discussion of our sampling strategies is provided in Section~\ref{sec:sampling-strategies}.


\subsection{SEI-SLR: sampling the selection event}
\label{subsec:simu-annealing}

From Proposition~\ref{prop:signs} and the KKT conditions from~\eqref{eq:gKKTconditions}, we know that the selection event~$E_M$ can be written as \begin{equation}\label{def2:EM0}E_{M} = \left\{Y\in \{0,1\}^N \; |\; \mathds 1_{\|\widehat S_{-M}(Y)\|_{\infty}-1<0}, \; \mathds 1_{1=\min_{k \in M}\{|\widehat S_k(Y)|\}}\right\}.\end{equation}

\noindent \begin{minipage}{0.6\textwidth}
Based on the expression of~$E_{M}$ given in Eq.\eqref{def2:EM0}, we introduce the function
\[
b_{\delta}(x) = 1-\sqrt{\left(\frac{x}{\delta}\right)\wedge 1}\,,
\]
for some~$\delta>0$ and we define the energy
\[ \mathcal E(Y):= \max \left\{ p_1(Y) \; , \; p_2(Y)\right\},\]
\end{minipage} \hfill
\begin{minipage}{0.37\textwidth}

  \begin{tikzpicture}
    \draw[thick, gray, ->] (-0.25, 0) -- (3.25, 0)
      node[anchor=south west] {$x$};
      \draw[thick] (0, 0.2) -- (0, -0.2)
        node[anchor=north] {\(0\)};
      \draw[thick] (3, 0.2) -- (3, -0.2)
        node[anchor=north] {\(1\)};
    \draw[thick, gray, ->] (0, -0.25) -- (0, 3.25)
      node[anchor=west] {$b_{\delta}(x)$};
      \draw[thick] (-0.1, 3) -- (0.1, 3);
      \draw[blue] (0.8, 0) -- (3, 0);
    \draw[domain=0:0.8,smooth,variable=\x,blue] plot ({\x},{3-3*sqrt{\x/0.88}});
        \node at (-0.28,3) {$1$};
      \draw[thick] (0.8, 0.2) -- (0.8, -0.2)
        node[anchor=north] {\(\delta\)};
  \end{tikzpicture}      
\end{minipage}
where \[p_1(Y) := b_{\delta}\left(1-\|\widehat S_{-M}(Y)\|_{\infty}\right) \quad \text{and} \quad p_2(Y) :=  \frac{1}{| M|}\sum_{k\in M} (1-|\widehat S_k(Y)|).\]
The energy $\mathcal E$ measures how close some vector~$Y \in \{0,1\}^N$ is to~$E_{M}.$
With Lemma~\ref{lemma:energySA}, we make this claim rigorous by proving that for $\delta>0$ small enough, the selection event~$E_M$ corresponds to the set of vectors $Y\in \{0,1\}^N$ satisfying $\mathcal E(Y)=0$.




\begin{Lemma}\label{lemma:energySA} For any $M\subseteq [d]$, there exists $\delta
_c:=\delta_c(M,\mathbf X,\lambda)>0$ such that for all $\delta\in(0,\delta_c)$, the selection event $E_{M}=  \{Y \in \{0,1\}^N \; | \; \widehat M(Y) = M\}$ is equal to the set 
\[\left\{ Y \in \{0,1\}^N \; | \; p_1(Y)=0 \quad \text{and}\quad  p_2(Y)=0 \right\}.\]
\end{Lemma}

\begin{proof}
Let us consider some $\delta \in (0,\delta_c)$ where 
\[\delta_c:=  \min_{Y  \in E_M} \{ 1-\|\widehat S_{-M}(Y)\|_{\infty}\}.\]
Note that Eq.\eqref{def2:EM0} ensures that for any~$Y \in E_M$,~$\|\widehat S_{-M}(Y)\|_{\infty}<1$. This implies that $\delta_c>0$ since the set $E_M$ is finite.

It is obvious that for any $Y\in \{0,1\}^N$, the fact that $p_2(Y)=0$ is equivalent to $\min_{k \in M}|\widehat S_{k}(Y)|=1$. Moreover, thanks to our choice for the constant~$\delta$, it also holds that $p_1(Y)=0$ is equivalent to~$\|\widehat S_{-M}(Y)\|_{\infty}<1$. The characterization of the selection event~$E_M$ given by Eq.\eqref{def2:EM0} allows to conclude the proof.
\end{proof}

Lemma~\ref{lemma:energySA} states that-provided~$\delta$ is small enough--the selection event~$E_M$ corresponds to the set of global minimizers of the energy~$\mathcal E:\{0,1\}^N\to \mathds R_+$. This characterization allows us to formulate a simulating annealing~(SA) procedure in order to estimate~$E_M$. Let us briefly recall that simulated annealing algorithms are used to estimate the set of global minimizers of a given function. At each time step, the algorithm considers some neighbour of the current state and probabilistically decides between moving to the proposed neighbour or staying at its current location. While a transition to a state inducing a lower energy compared to the current one is always performed, the probability of transition towards a neighbour that leads to increase the energy is decreasing over time. The precise expression of the probability of transition is driven by a chosen {\it cooling schedule}~$(\mathrm T_t)_t$ where~$\mathrm T_t$ are called {\it temperatures} and vanish as~$t\to \infty$. Intuitively, in the first iterations of the algorithm the temperature is high and we are likely to accept most of the transitions proposed by the SA. In that way, we give our algorithm the chance to escape from local minimum. As time goes along, the temperature decreases and we expect to end up at a global minima of the function of interest. 

We refer to~\cite[Chapter 12]{bremaud2013markov} for further details on SA. Our method is described in Algorithm~\ref{algo1} and in the next section, we provide theoretical guarantees. In Algorithm~\ref{algo1}, $P:\{0,1\}^N\times \{0,1\}^N \to [0,1]$ is the Markov transition kernel such that for any~$Y\in \{0,1\}^N$, $P(Y,\cdot)$ is the probability measure on~$\{0,1\}^N$ corresponding to the uniform distribution on the vectors in~$\{0,1\}^N$ that differs from~$Y$ in exactly one coordinate.

\begin{algorithm}[H]
\textbf{Data: } $\mathbf X$, $Y$, $\lambda$, $K_0$, $T$
\begin{algorithmic}[1]
 \STATE Compute $\hat \vartheta^{\lambda} \in \underset{\vartheta \in \mathds R^d}{\arg \min} \{ \mathcal L_N(\vartheta,(Y,\mathbf X)) + \lambda \|\vartheta\|_1\}$\;
 \STATE Set $M=\{ k \in [d] \; |\; \hat \vartheta^{\lambda}_k \neq 0\}$\;
 \STATE $Y^{(0)}\leftarrow Y$\;
 \FOR{$t=1$ to $T$ }
\STATE $Y^{\mathrm{c}} \sim P(Y^{(t-1)},\cdot)$\;
 \STATE$\hat \vartheta^{\lambda,\mathrm c} \in \underset{\vartheta \in \mathds R^d}{\arg \min} \{ \mathcal L_N(\vartheta, (Y^{\mathrm{c}},\mathbf X)) + \lambda \|\vartheta\|_1\}$\;
\STATE $\widehat S(Y^{\mathrm{c}}) =\frac{1}{\lambda}\mathbf X^{\top}(Y^{\mathrm{c}}-\sigma(\mathbf X \hat \vartheta^{\lambda,\mathrm{c}}))$\;
  \STATE$\Delta \mathcal  E = \mathcal E(Y^{\mathrm{c}})-\mathcal E(Y^{(t-1)})$\;
  \STATE$U \sim \mathcal U([0,1])$\;
 \STATE  $\mathrm T_t \leftarrow \frac{K_0}{\log(t+1)}$\;
  \IF{$\exp\left(-\frac{\Delta \mathcal E}{\mathrm T_t}\right)\geq U$}
  \STATE $Y^{(t)} \leftarrow Y^{\mathrm{c}}$\;
   \ENDIF
 \ENDFOR
 \end{algorithmic}
 \caption{SEI-SLR: Selection Event Identification for SLR}
 \label{algo1}
\end{algorithm}

\subsection{Proof of convergence of the algorithm}


To provide theoretical guarantees on our methods in the upcoming sections, we need to understand what is the distribution of~$Y^{(t)}$ as~$t\to \infty$. This is the purpose of Proposition~\ref{prop:uniform-distribution} which shows that the SEI-SLR algorithm generates states uniformly distributed on~$E_M$ in the asymptotic~$t\to \infty$.
\begin{proposition}
\cite[Example 12.2.12]{bremaud2013markov} \label{prop:uniform-distribution}\\
For a cooling schedule satisfying $\mathrm T_t \geq 2^{N+1}/\log(t+1)$, the limiting distribution of the random vectors $Y^{(t)}$ is the uniform distribution on the selection event~$E_{M}$.
\end{proposition}
Proposition~\ref{prop:uniform-distribution} has the important consequence that we are able to compute the distribution of the binary vector~$Y=(y_i)_{i\in[N]}$ where each~$y_i$ is a Bernoulli random variable with parameter~$\pi^*_i \in (0,1)$ conditional on the selection event. The formal presentation of this result is given by Proposition~\ref{prop:point-alternative} which will be the cornerstone of our inference procedures presented in Section~\ref{sec:asymptotic}.

\begin{proposition}\label{prop:point-alternative}
Let us consider~$M\subseteq[d]$ and some~$\vartheta^* \in \mathds R^d$. Consider a random vector~$Y$ with distribution~$\overline{\mathds P}_{\pi^*}$ where $\pi^*=\sigma(\mathbf X \vartheta^*)$. For a cooling schedule satisfying $\mathrm T_t \geq 2^{N+1}/\log(t+1)$, it holds for any function $h:\{0,1\}^N \to \mathds R$,
 \[\frac{\sum_{t=1}^T h(Y^{(t)}) \mathds P_{\pi^*}(Y^{(t)})}{\sum_{t=1}^T  \mathds P_{\pi^*}(Y^{(t)})} \underset{T \to \infty}{\to}  \overline{\mathds E}_{\pi^*}\left[h(Y)  \right] \quad \text{almost surely}.\]
\end{proposition}

\begin{proof}Let us consider some map $h:\{0,1\}^N \to \mathds R$. Then,
\begin{align*}
  \overline{\mathds E}_{\pi^*}\left[h(Y)  \right]&=\frac{\sum_{y \in E_M} h(y) \mathds P_{\pi^*}(y)}{\sum_{y \in E_M} \mathds P_{\pi^*}(y)}= \frac{\mathds E(h(U_{M}) \mathds P_{\pi^*}(Y=U_{M}))}{\mathds E(  \mathds P_{\pi^*}(Y=U_{M} ))},
\end{align*}
where $U_M$ is a random variable taking values in $\{0,1\}^N$ which is uniformly distributed over $E_M$. Then the conclusion directly follows from Proposition~\ref{prop:uniform-distribution}.
\end{proof}

\section{Numerical results}
\label{sec:experiments}

The code to reproduce our results is available at the following url: \url{https://github.com/quentin-duchemin/SIGLE}.

\subsection{Sampling the conditional distribution with SEI-SLR}
\label{sec:sampling-strategies}

As already discussed in the beginning of Section~\ref{sec:sampling}, we propose two different ways to sample points on the hypercube $\{0,1\}^N$ allowing us to compute conditional expectations of the form $\overline {\mathds E}_{\theta^*} [h(Y)]$ or $\overline {\mathds E}_{\pi^*} [h(Y)]$ where $h:\{0,1\}^N \to \mathds R$. 

The first method is a simple rejection sampling approach and is described in Algorithm~\ref{algo:RS}.
\begin{algorithm}
\begin{algorithmic}[1]
    \STATE \text{\bf Input:} $ T$, $\pi^*$, $\mathbf X$, $M$, $\lambda$
  \STATE $t\leftarrow 0$
  \WHILE{$t<T$}
  \STATE $Y \sim \mathds P_{\pi^*}$
  \IF{$Y \in E_M$}
  \STATE $t\leftarrow t+1$
  \STATE $Y^{(t)}\leftarrow Y$
  \ENDIF
  \ENDWHILE
  \RETURN $(Y^{(t)})_{t\in[T]}$
\end{algorithmic}
\caption{Rejection sampling.}
\label{algo:RS}
\end{algorithm}
The rejection sampling algorithm does not require any parameter tuning and allow to estimate any expectation $\overline {\mathds E}_{\pi^*}[h(Y)]$ by taking a simple average over the list of returned states namely $\sum_{t\in[T]} h(Y^{(t)})$. Nevertheless, a major drawback of the rejection sampling method is its large computing time when the number of features $d$ is getting "large" (typically when $d$ exceeds ten). Indeed, the number of possible supports for a lasso solution is equal to $2^d$ and increases exponentially fast with $d$.

In order to bypass this curse of dimensionality, we proposed in Section~\ref{subsec:simu-annealing} the SEI-SLR algorithm: a simulated annealing-based method that is proved to generate states that are asymptotically uniformly distributed on the selection event. The SEI-SLR algorithm solves the computational issue faced by the rejection sampling for large $p$ values.
Nevertheless, the convergence of SEI-SLR algorithm requires the use of well-chosen parameters namely:
\begin{itemize}
\item the parameter $\delta$ involved in the energy (cf. Section~\ref{subsec:simu-annealing}), 
\item the temperatures $(\mathrm T_t)_{t}$,
\item the time horizon of the algorithm.
\end{itemize}
Let us finally mention that estimating expectations of the form $\overline {\mathds E}_{\theta^*}[h(Y)]$ from the samples $(Y^{(t)})_t$ obtained with the SEI-SLR algorithm requires the computation of weighting factors that allow to go from the uniform distribution on the selection event~$E_M$ to the target conditional distribution $\overline {\mathds E}_{\theta^*}$. In Table~\ref{table:RSvsSA}, we sum-up the previous discussion in order to give a comprehensive comparison between the two methods. In the rest of this section, we illustrate the performance of the SEI-SLR algorithm

\begin{table}
\centering
{\small
\begin{tabular}{C{4cm}|C{4cm}C{4cm}}
& Rejection Sampling  & SEI-SLR\\\hline
Conditions for application & \cellcolor{red!20!white}Simple hypothesis (cf. Table~\ref{table:hypo-testing})  & \cellcolor{green!20!white}No condition\\
Need for hyperparameters tuning&\cellcolor{green!20!white} No & \cellcolor{red!20!white}Yes\\
Computational time & Efficient only for a small $d$ but $N$ can be chosen (very) large & Easier to use for relatively small $N$ but $d$ can be large\\
\rowcolor{gray!20!white}Distribution of the sequence of states generated $(Y^{(t)})_{t\in [T]}$ & $\overline {\mathds P}_{\theta^*_0}$ or $\overline {\mathds P}_{\pi^*_0}$ (cf. Table~\ref{table:hypo-testing}) & Uniform distribution on $E_M$\\
\rowcolor{gray!20!white}& $\Downarrow$	 & $\Downarrow$	\\
\rowcolor{gray!20!white}In simple hypothesis testing with $\mathds H_0 \; :\; "\theta^*=\theta^*_0"$, $\overline {\mathds E}_{\theta^*_0}[h(Y)]\approx \dots$ & $\frac{1}{T} \sum_{t\in[T]} h(Y^{(t)})$ & $ \frac{\sum_{t\in[T]} \overline {\mathds P}_{\theta_0^*}(Y^{(t)})h(Y^{(t)})}{ \sum_{r\in[T]} \overline {\mathds P}_{\theta_0^*}(Y^{(r)})}$ \\
\rowcolor{gray!20!white}& i.e. the estimate is obtained with a simple average on the sequence of generated states. &  i.e. we need to weight properly each visited state.
\end{tabular}
}
\caption{Comparison between the rejection sampling method and the SEI-SLR algorithm.}
\label{table:RSvsSA}
\end{table}

We consider a design matrix $\mathbf X \in \mathds R^{10\times 20}$ where all entries are i.i.d. and sampled from a standard normal distribution. We consider $\delta = 0.01$, $\lambda=1.5$ and we sample some vector~$Y_0\in \{0,1\}^N$ with i.i.d. entries with a Bernoulli distribution of parameter $1/2$. Note that the tuple~$(\mathbf X,Y_0,\lambda)$ determined the set of active variables~$M$ (cf. Eq.\eqref{e:generalized_lasso}). We run the SEI-SLR algorithm for $3\,000\,000$ time steps. By choosing this toy example with a small value for~$N$, we are able to compute exactly the selection event~$E_M$ by running over the~$2^{10}$ possible vectors~$Y \in \{ 0,1\}^N$.
In the following, we identify each vector~$Y \in \{0, 1\}^N$ with the number between~$0$ and~$2^N-1=1024$ that it represents in the base-2 numeral system. Using this identification, it holds on our example that~$E_M = \{3,35,222,801,988,1020\}$.
\medskip

\begin{figure}[ht!]
\begin{center}
 \centering
    \begin{subfigure}[b]{0.49\textwidth}
        \centering
        \includegraphics[width=\textwidth]{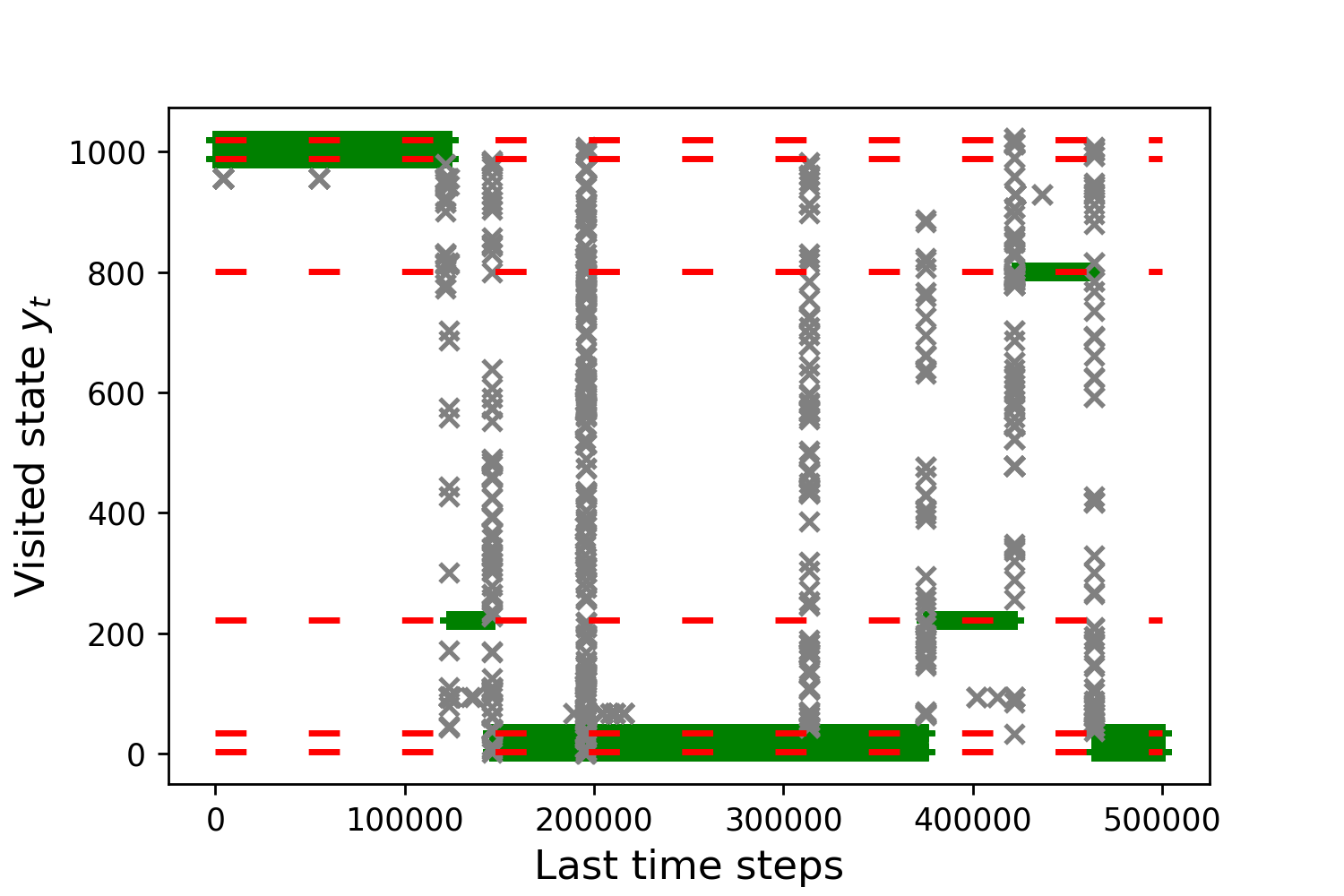}
        \caption[]%
        {{\small Last $500\,000$ visited states of the SEI-SLR algorithm. The dotted red lines represent the states in $E_M$.}}    
    \end{subfigure}
    \hfill
     \centering
    \begin{subfigure}[b]{0.49\textwidth}
        \centering
        \includegraphics[width=\textwidth]{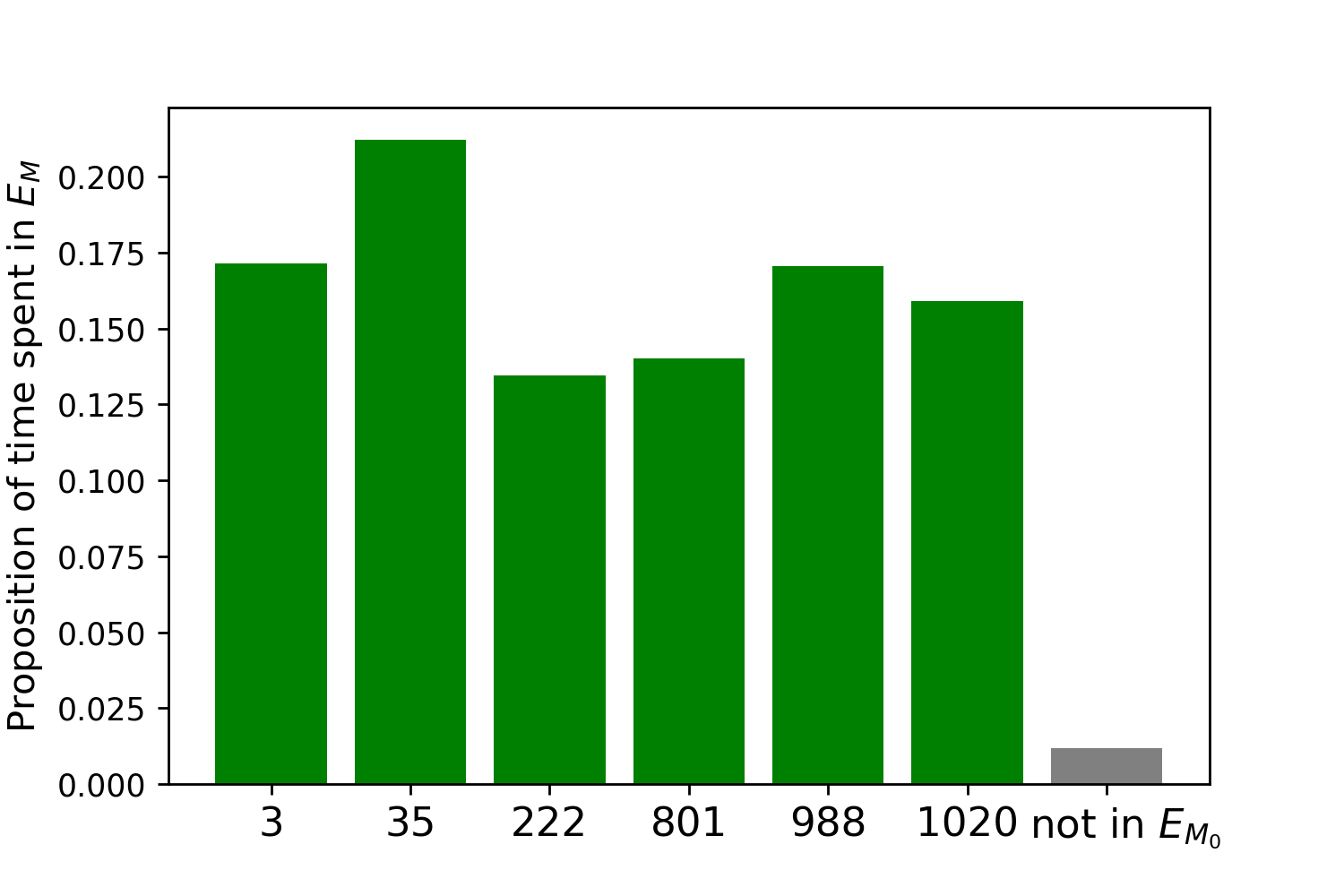}
        \caption[]%
        {{\small Time spent in each state of $E_{M}$ and outside of $E_{M}$.\\}}    
    \end{subfigure}
\caption{Visualization of the time spent in the selection event from the sequence of states provided by the SEI-SLR algorithm.}
\label{fig:time-spent}\end{center}
\vskip -0.2in
\end{figure}

Figure~\ref{fig:time-spent}.(a) shows
 the last~$500,000$ visited states for our simulated annealing path. On the vertical axis, we have the integers encoded by all possible vectors~$Y \in \{0,1\}^N$. The red dashed lines represent the states that belong to the selection event $E_M$. While crosses are showing the visited states on the last~$500,000$ time steps of the path, green crosses are emphasizing the ones that belong to the selection event. On this example, we see that the SEI-SLR algorithm covers properly the selection event without being stuck in one specific state of~$E_{M}$. The simulated annealing path is jumping from one state of~$E_{M}$ to another, ending up with an asymptotic distribution of the visited states that approximates the uniform distribution on~$E_{M}$ (see Figure~\ref{fig:time-spent}.(b)). Let us point that two neighboring states in space~$\{0,1\}^N$ will not necessarily be encoded by close integers.

Figure~\ref{fig:time-spent}.(a) suggests that the vectors encoded by the integers~$3$ and~$35$ are close in the space~$\{0,1\}^N$. Indeed, we see on Figure~\ref{fig:time-spent}.(a) that between indexes $180\,000$ and $350\,000$, our algorithm goes from one of these states to another passing through almost no state that does not belong to the selection event (this can be seen because there are only few gray crosses on this time window of the simulated annealing path). The same remark holds for the two states encoded by the integers~$988$ and~$1020$. However, we observe a large number of visited states that do not belong to~$E_M$ when we perform a transition between any other pair of states belonging to the selection event. We can therefore legitimately think that the selection event separates into four groups of fairly distant states. This is confirmed by Figure~\ref{fig:clusters} which presents the Hamming distances between the different vectors of~$E_M$ and reveals the existence of two clusters.

\begin{figure}[ht!]
\begin{center}
 \centering
    \begin{subfigure}[b]{0.67\textwidth}
        \centering
        \includegraphics[width=\textwidth]{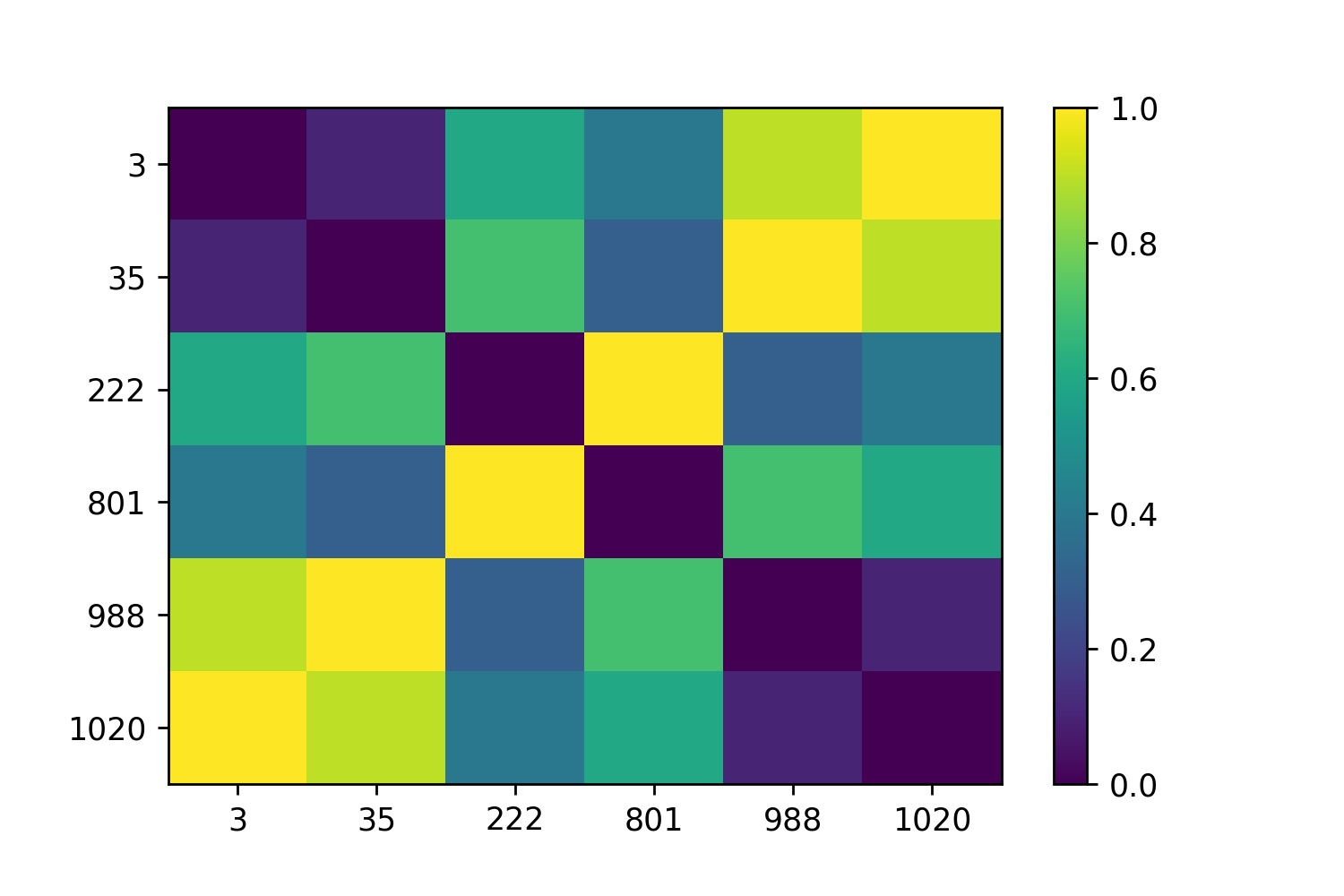}
    \end{subfigure}
\caption{Normalized (by $N$) Hamming distances between the different states of the selection event.}
\label{fig:clusters}
\end{center}
\vskip -0.2in
\end{figure}

With Figure~\ref{fig:other-path}, we show the results obtained from the SEI-SLR algorithm considering a similar experiment but taking $d=15$ (instead of $20$) and $\lambda=2$ (instead of $1.5$), which leads to a larger selection event.

\begin{figure}[ht!]
\vskip 0.2in
\begin{center}
 \centering
    \begin{subfigure}[b]{0.49\textwidth}
        \centering
        \includegraphics[width=\textwidth]{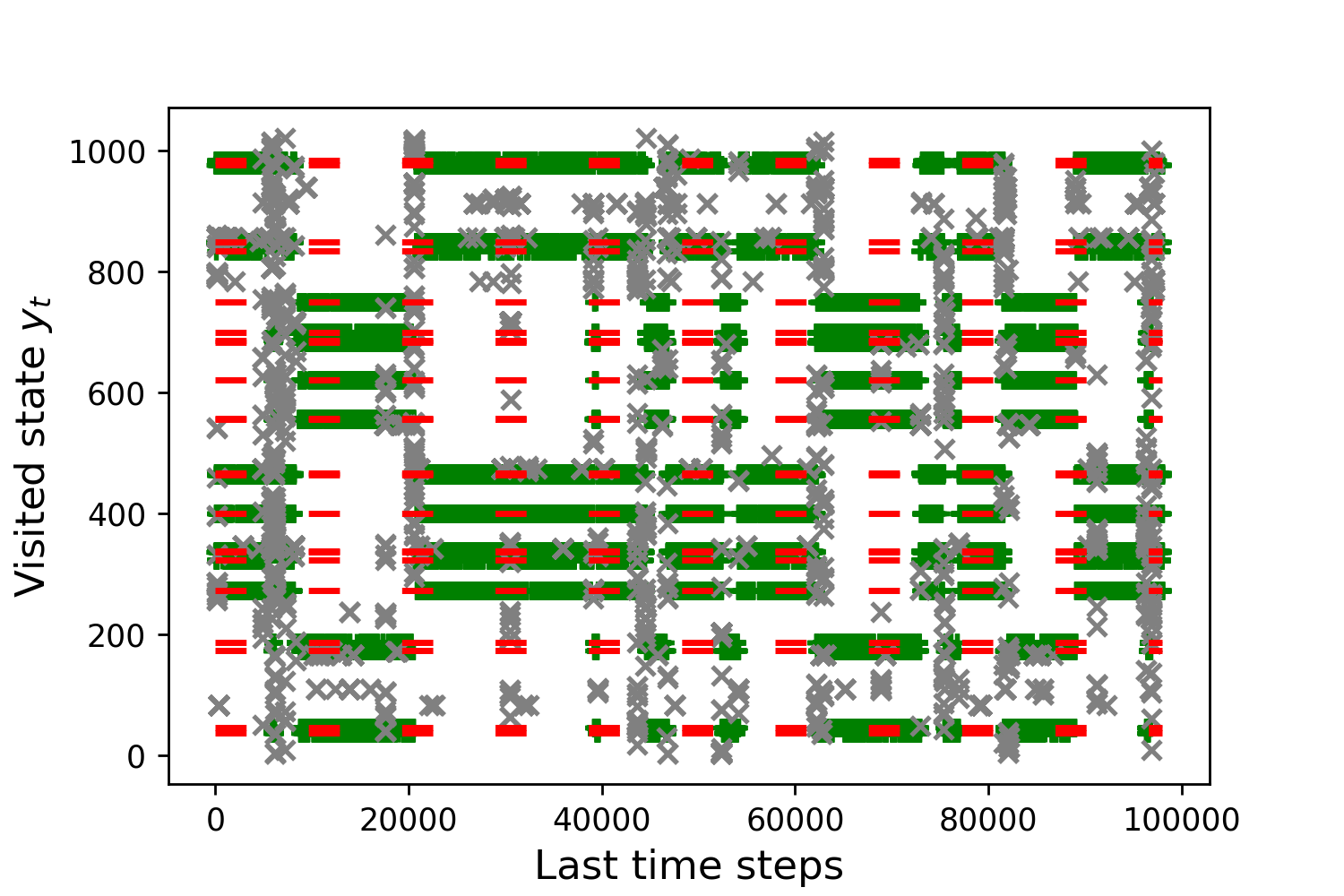}
        \caption[]%
        {{\small Last $100\,000$ visited states of the SEI-SLR algorithm. The dotted red lines represent the states in $E_M$.}}    
    \end{subfigure}
    \hfill
     \centering
    \begin{subfigure}[b]{0.49\textwidth}
        \centering
        \includegraphics[width=\textwidth]{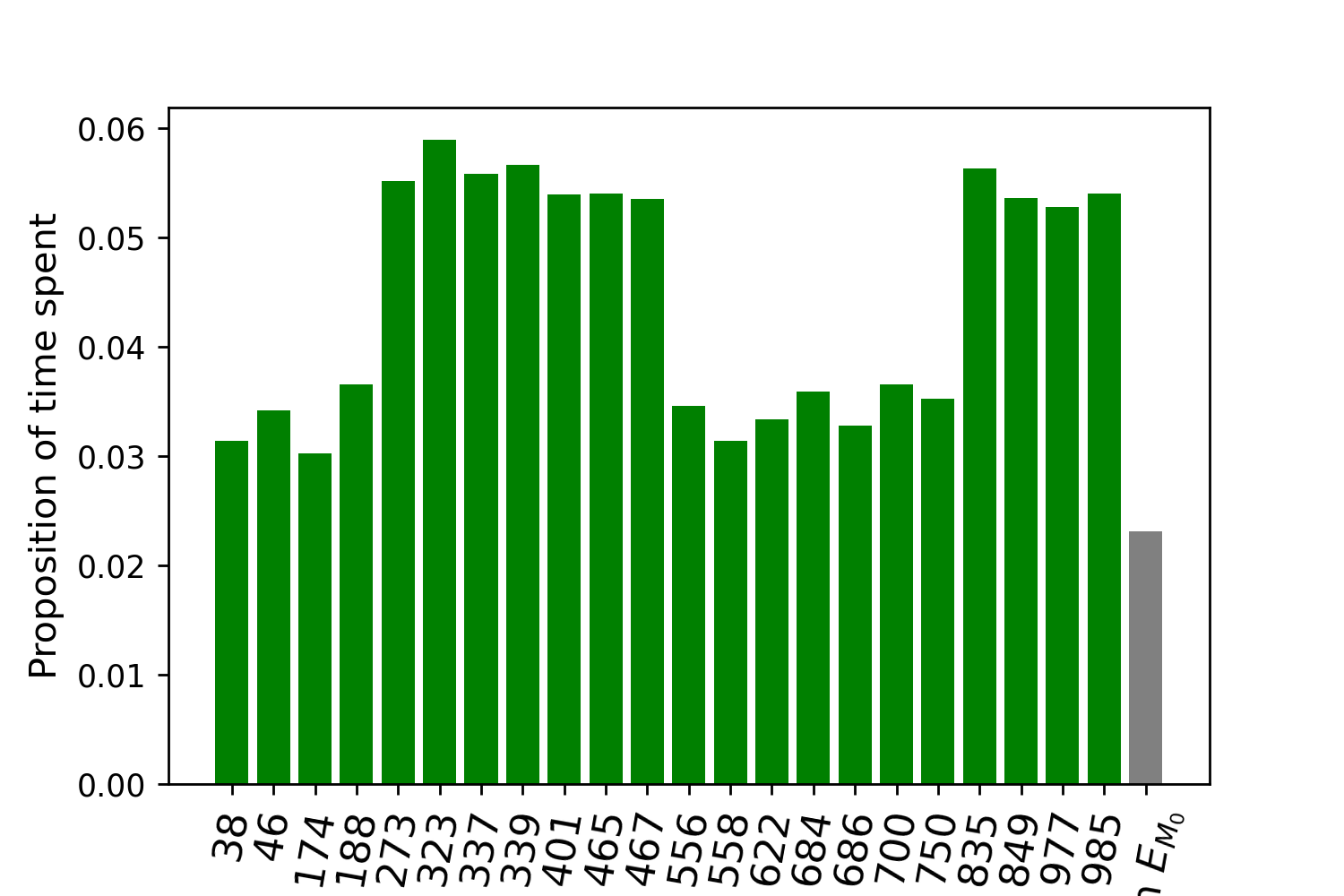}
        \caption[]%
        {{\small Time spent in each state of $E_{M}$ and outside of $E_{M}$.\\}}    
    \end{subfigure}
\caption{Visualization of the time spent in the selection event from the sequence of states provided by the SEI-SLR algorithm.}
\label{fig:other-path}\end{center}
\vskip -0.2in
\end{figure}

 \paragraph{Comparison with the linear model.} The previous theoretical and numerical results show that our approach allows to correctly identify the selection event~$E_M$. Nevertheless, this method suffers from the curse of dimensionality since the random walks in the simulated annealings need to cover a state space of~$2^N$ points. Let us mention that even in the linear model where the selection event~$E_M$ has the nice property to be a union of polyhedra, the method from \cite{sun16} to provide inference on a linear transformation of $Y$ can also cope with some computational issues. Indeed, the construction of confidence intervals conditionally on the event~$E_M$ requires the computation of~$2^s$ intervals (while the computation of each of them requires at least~$N^3$ operations) where~$s=|M|$ (see~\cite[Section 6]{sun16}). Roughly speaking, both our approach in the logistic model and the one from~\cite[Section 6]{sun16} in the linear model are limited in large dimensions. While in the linear case, computational efficiency of the known methods mainly depends on~$s=|M|$, the extra cost arising from the non-linearity of the logistic model is their dependence on~$N$.
 
 Let us finally mention that in the Gaussian linear model, one can bypass the limitation of computing the~$2^s$ intervals for each possible vector of dual signs on the equicorrelation set $M$ by conditioning further on the observed vector of signs $\widehat S_M(Y)=\mathrm{sign}(\widehat \theta^{\lambda})_M$. Stated otherwise, instead of conditioning on $E_M$, we condition on $E_{M}^{S_M}$ where $S_M=\widehat S_M(Y)$. This method reduces the computational burden but it will lead in general to less powerful inference procedures due to some information loss which can be quantified through the so-called 
leftover Fisher information. In Section~\ref{sec:conditional-signs}, we discuss with further details PSI when we condition additionally on the observed vector of signs.

 \subsection{Hypothesis Testing}
 
 \label{sec:hypo-testing}

 In this section, we propose to analyze the level and the power of the SIGLE procedure 
 considering the following simple hypothesis testing problem
 \[\mathds H_0: \, \{\theta^*=\theta^*_0\}, \qquad \mathds H_1: \, \{\theta^*\neq\theta^*_0\} .\]We compare the SIGLE method with the results obtained from a weak learner and from the heuristic method proposed by \cite{taylorGLM}.  
 
 \paragraph{Description of the settings of our experiments.}

We consider a design matrix $\mathbf X \in \mathds R^{N\times d}$ where the entries are i.i.d. and sampled from a standard normal distribution. We consider two different experiments (cf. Table~\ref{table:settings}).
For the Setting 1 under the null, the set of active variables $M$ is of size $4$. We sample states from $E_M$ using the rejection sampling method and approximately $8\%$ of the states sampled from $\mathds P_{\theta^*_0}$ fall in the selection event with this algorithm. For the Setting 2, we use the SEI-SLR algorithm to sample states in $E_M$. 

\begin{table}[h!]
\centering
\setlength{\tabcolsep}{5pt}
\begin{tabular}{c|cccccc}
   & $N$ & $d$ & $\mathbf X$ & $\lambda$ & $\theta^*_0 $ & Sampling method \\ \hline
Setting 1 & $100$ & $10$ & $\mathbf X_{i,j} \sim \mathcal N(0,1)$ & $5$  & $[0,\dots,0]$ & Rejection sampling  \\
Setting 2 & $20$ &  $15$ & $\mathbf X_{i,j} \sim \mathcal N(0,1)$ & $3$  & $[0,\dots,0]$ & SEI-SLR
\end{tabular}
\caption{Description of the experiments.}
\label{table:settings}
\end{table}

\subsubsection{Description of the benchmark methods}
 \label{sec:benchmark}
 
\paragraph{A weak learner.} Our weak learner is a two-sided test based on the statistic $\sum_{i=1}^n |\overline \pi_i^{\theta_0} - y_i|$ where $\overline \pi^{\theta_0}$ is the expectation of the vector of observations under the null conditional on the selection event (cf. Eq.\eqref{eq:gradbar}). Let us highlight that $\overline {\pi}^{\theta^*_0}$ is estimated by~$\widetilde \pi^{\theta^*_0}$ where
\begin{itemize}
\item $\widetilde \pi^{\theta^*_0}=\sum_{t=1}^T Y^{(t)}$ if the sequence $(Y^{(t)})_{t\in [T]}$ is generated from the rejection sampling method,
\item $\widetilde \pi^{\theta^*_0}=\frac{\sum_{t=1}^T Y^{(t)} \mathds P_{\theta^*_0}(Y^{(t)})}{\sum_{t=1}^T  \mathds P_{\theta^*_0}(Y^{(t)})}$ if we use the SEI-SLR algorithm to generate the states $(Y^{(t)})_{t\in [T]}$.
\end{itemize}
The method is calibrated empirically using the sequence~$(Y^{(t)})_{t\in [T]}$.
 
\paragraph{The PSI method from \cite{taylorGLM}.} The PSI method in the logistic model proposed by~\cite{taylorGLM} is described in details in Section~\ref{sec:debiasing}. Based on heuristic justifications, this approach has the advantage to provide an hypothesis testing method for any linear transformation of the debiased lasso solution $\underline \theta$ (i.e. of the form $\eta^{\top} \underline \theta$) that does not require a cumbersome sampling step. We propose to compare the SIGLE methods with the PSI procedure from \cite{taylorGLM} by considering different approaches:
\begin{itemize}
\item[{\bf TT-1}] We use the p-value obtained from a two-sided test based on the statistic~$\underline \theta_1$.
\item[{\bf TT-Bonferroni}] We use a Bonferroni method from the p-values computed from the set of two-sided composite tests with null hypotheses $\mathds H_0: \; "\theta^*_j = \big[\theta^*_0\big]_j"$ for $j \in [s]$ where $s=|M|$.  
\end{itemize}

\subsubsection{Calibration}
\label{sec:calibration-SIGLE}

\paragraph{SIGLE procedures.} 

To compute the SIGLE statistics, we need to estimate $\overline G_N(\pi^* _0)$ (and $\overline \theta(\theta^*_0) $ in the selected model). Since the conditional distribution~$\overline {\mathds P}_{\pi^*_0}$ (resp. $\overline {\mathds P}_{\theta^*_0}$) is not known, we sample states from these distributions to estimate these quantities. We use these states sampled in the selection event~$E_M$ in order to calibrate empirically the SIGLE procedures. In the literature, one often says that we {\it calibrate by sampling under the null}.

\paragraph{PSI methods from \cite{taylorGLM}.}
In \cite{taylorGLM}, the authors justify their approach with asymptotic considerations. Figure~\ref{fig:cali-null}.(a) shows that for a large value for $N$, the methods TT-1 and TT-Bonferroni are correctly calibrated since the CDF of the p-values are uniform under the null. On the contrary, for small value of $N$, the calibration of these procedures may be lost as shown with Figure~\ref{fig:cali-null}.(b).

\paragraph{The weak learner.} By construction, the p-values of the weak learner are stochastically larger than uniform under the null. The CDF of p-values are uniformly distributed in the Setting 1 with Figure~\ref{fig:cali-null}.(a). Note that in the Setting 2, the weak learner is irrelevant since the selection event~$E_M$ is such that $\overline \pi^{\theta^*_0}= \frac12 \mathbf 1_N$. This means that the test-statistic of the weak learner is constant.

\begin{figure}[ht!]
\vskip 0.2in
\begin{center}
 \centering
    \begin{subfigure}[b]{0.49\textwidth}
        \centering
        \includegraphics[width=\textwidth]{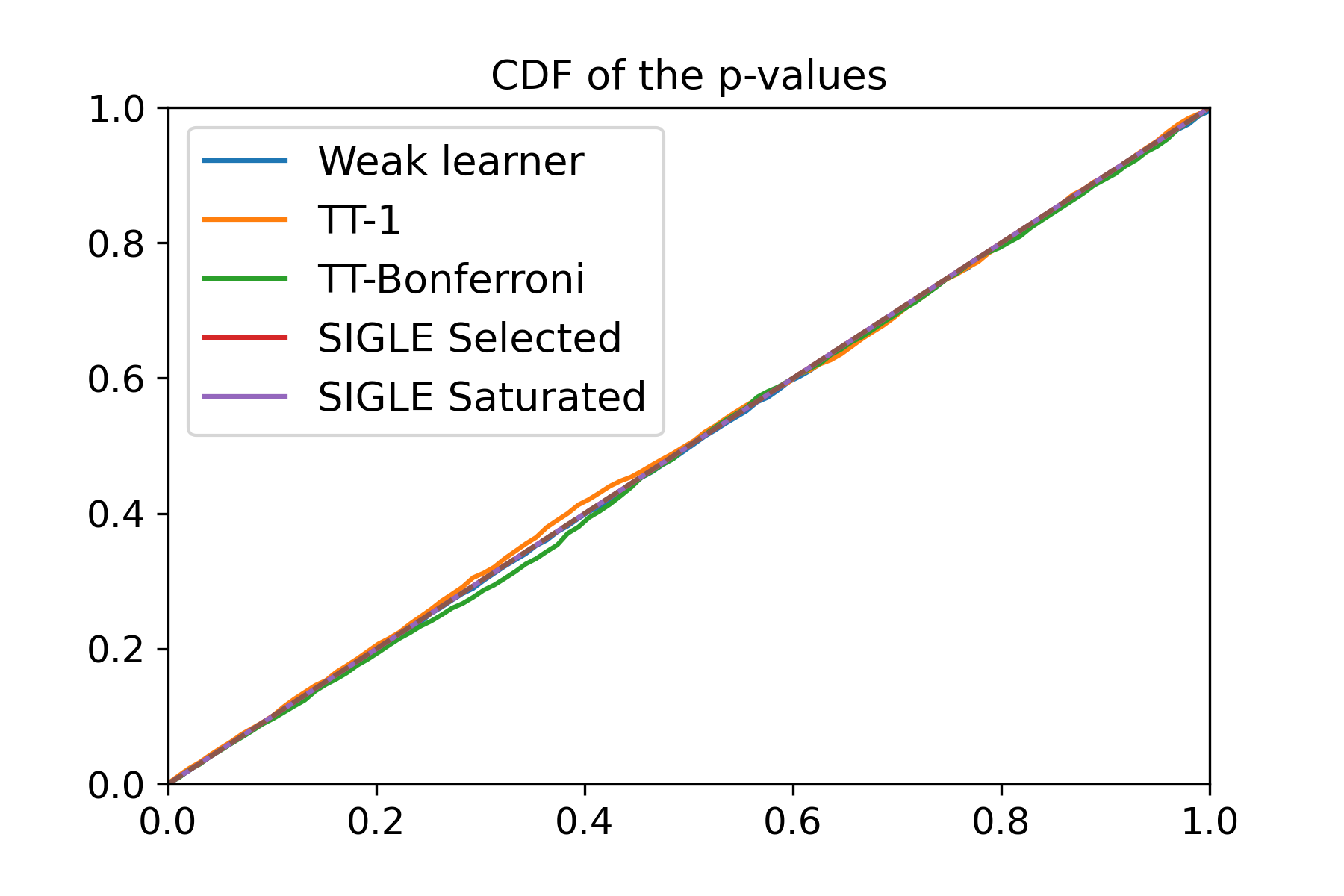}
        \caption[]%
        {{Setting 1.}}    
    \end{subfigure}
    \hfill
     \centering
    \begin{subfigure}[b]{0.49\textwidth}
        \centering
        \includegraphics[width=\textwidth]{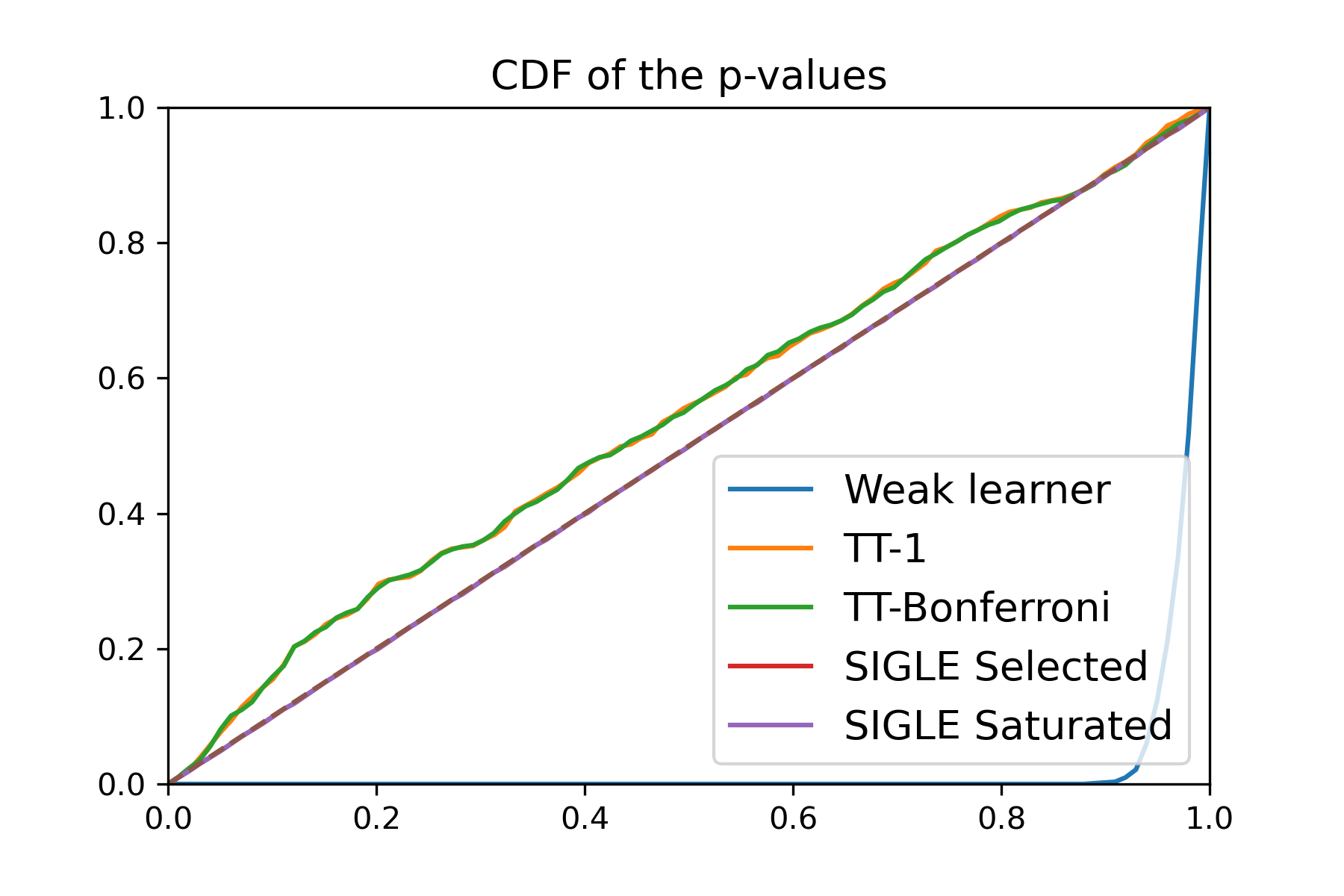}
        \caption[]%
        {{\small Setting 2.}}    
    \end{subfigure}
\caption{CDF of the p-values of the different testing procedures under the global null for the Settings given in Table~\ref{table:settings}. We calibrate empirically the SIGLE methods.}
\label{fig:cali-null}\end{center}
\vskip -0.2in
\end{figure}

\subsubsection{Power}

We consider two different types of alternatives:
\begin{itemize}
\item \underline{{\it Localized alternatives}.}\\
A {\it localized signal} is of the form $\vartheta^* = [\nu \, ,\, 0 \, ,\,,\, \dots \, , \, 0]$ for some $\nu >0$.
\item \underline{{\it Disseminated alternatives}.}\\
A {\it disseminated signal} is of the form $\vartheta^* = \nu \mathbf 1_d$ for some $\nu >0$. 
\end{itemize}
As explained in the previous sections, we calibrate the SIGLE methods empirically. Figure~\ref{fig:cali-null}.(a) shows that the p-values for the SIGLE methods are distributed uniformly under the null. Figure~\ref{fig:cali-null}.(a) also shows that the benchmark methods are correctly calibrated.\\ Figures~\ref{fig:power-localized-setting1} and~\ref{fig:power-disseminated-setting1} show that SIGLE is more powerful compared to the benchmark methods for the localized or disseminated alternative. In Figure~\ref{fig:power-disseminated-setting1}.(b), the power of the SIGLE methods is so high that we barely identify the curve of the CDF in the top-left corner. Figure~\ref{fig:power-nu-setting1} gives a complete visualization of the power of the different testing methods when tests have level $5\%$. We see that the methods of this paper are always improving upon the benchmark methods. The superiority of the SIGLE methods regarding power becomes even more significant when we consider disseminated alternatives. This is not surprising since the methods of this paper are intrinsically designed to tackle simple hypothesis testing problem.

Another interesting remark is that the procedure TT-1 is more powerful than the procedure TT-Bonferroni when considering localized alternatives as showed by Figure~\ref{fig:power-localized-setting1}. Again this result is not surprising: the TT-Bonferroni loses power by testing each coordinate of the parameter vector while TT-1 is focused on a single coordinate which is better suited to identify a localized signal. On the contrary, the TT-Bonferroni is more powerful when considering disseminated alternatives as observed with Figure~\ref{fig:power-disseminated-setting1} and Figure~\ref{fig:power-nu-setting1}.(b).

We conduct similar experiments in the Setting 2 given in Table~\ref{table:settings}. Figure~\ref{fig:cdf-setting2} shows that the TT-1 and TT-Bonferroni are still less powerful than the SIGLE methods. Moreover, Figure~\ref{fig:cdf-setting2}.(b) illustrates that in the high dimensional setting (i.e. when $d$ is larger than $N$), the size of the selection event can be small which leads to a non-smooth staircase function for the CDF of p-values. In the example of the Figure~\ref{fig:cdf-setting2}.(b), the selection event contains only~$22$ states.

\begin{figure}[ht!]
\vskip 0.2in
\begin{center}
 \centering
    \begin{subfigure}[b]{0.49\textwidth}
        \centering
        \includegraphics[width=\textwidth]{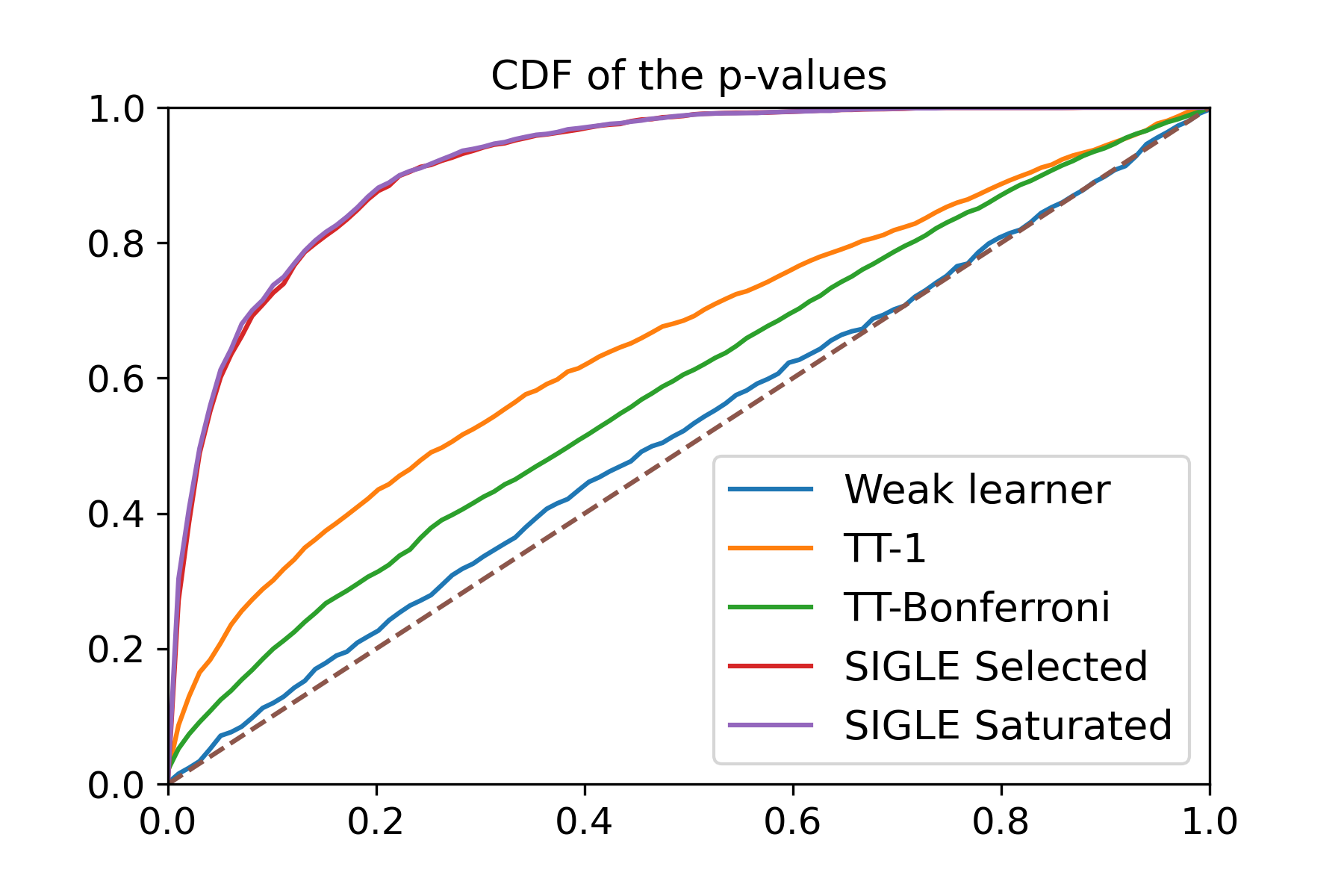}
        \caption[]%
        {{\small CDF of p-values for the alternative\\$\vartheta^*=[0.4 \, , \, 0 \, , \, \dots \, , \, 0]$.}}    
    \end{subfigure}
    \hfill
     \centering
    \begin{subfigure}[b]{0.49\textwidth}
        \centering
        \includegraphics[width=\textwidth]{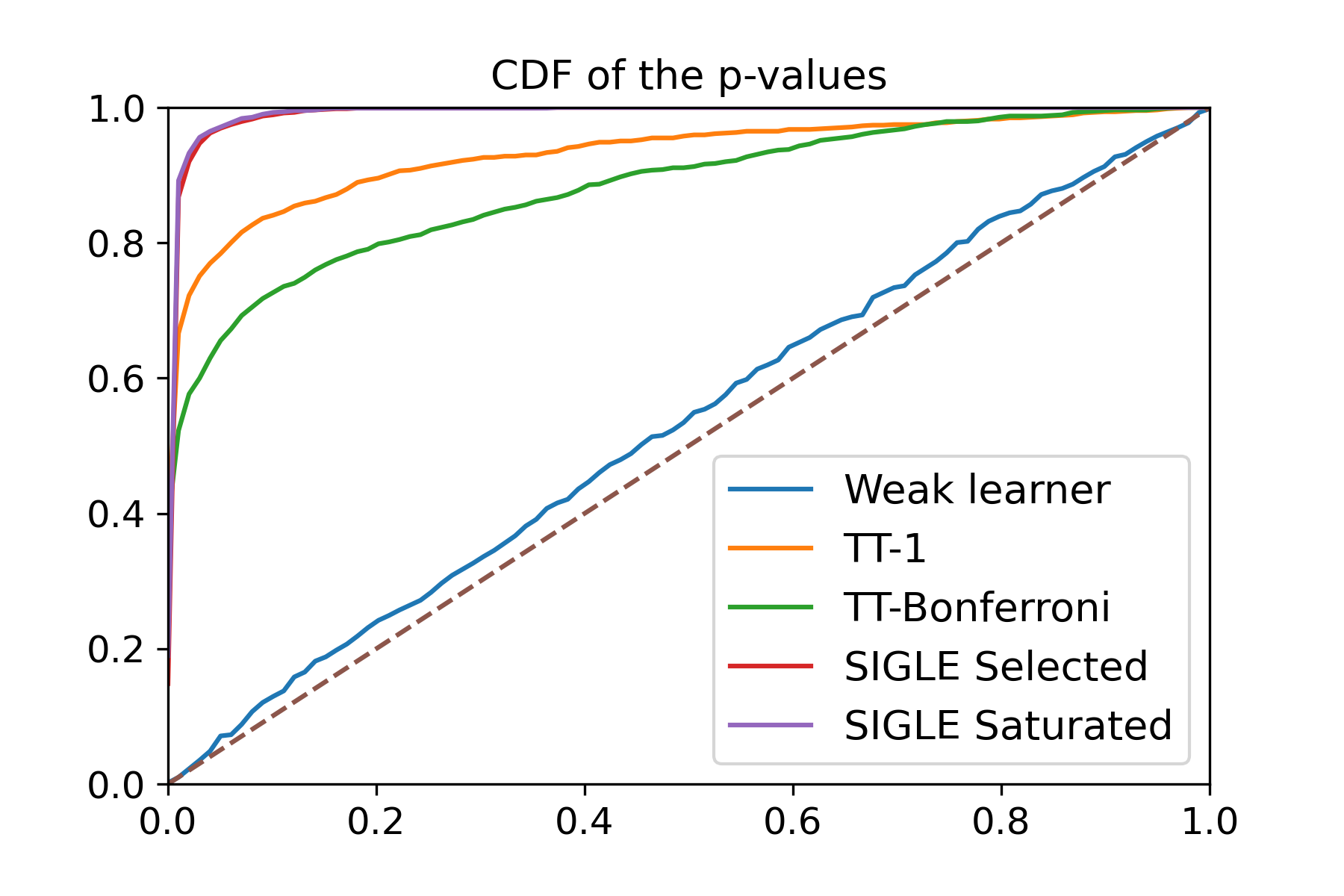}
        \caption[]%
        {{\small CDF of p-values for the alternative\\$\vartheta^*=[0.9 \, , \, 0 \, , \, \dots \, , \, 0]$.}}    
    \end{subfigure}
\caption{CDF of the p-values for the SIGLE procedures and the benchmark methods using the {\bf Setting 1} (cf. Table~\ref{table:settings}) for {\bf \color{red} localized alternatives}.}
\label{fig:power-localized-setting1}\end{center}
\vskip -0.2in
\end{figure}

\begin{figure}[ht!]
\vskip 0.2in
\begin{center}
 \centering
    \begin{subfigure}[b]{0.49\textwidth}
        \centering
        \includegraphics[width=\textwidth]{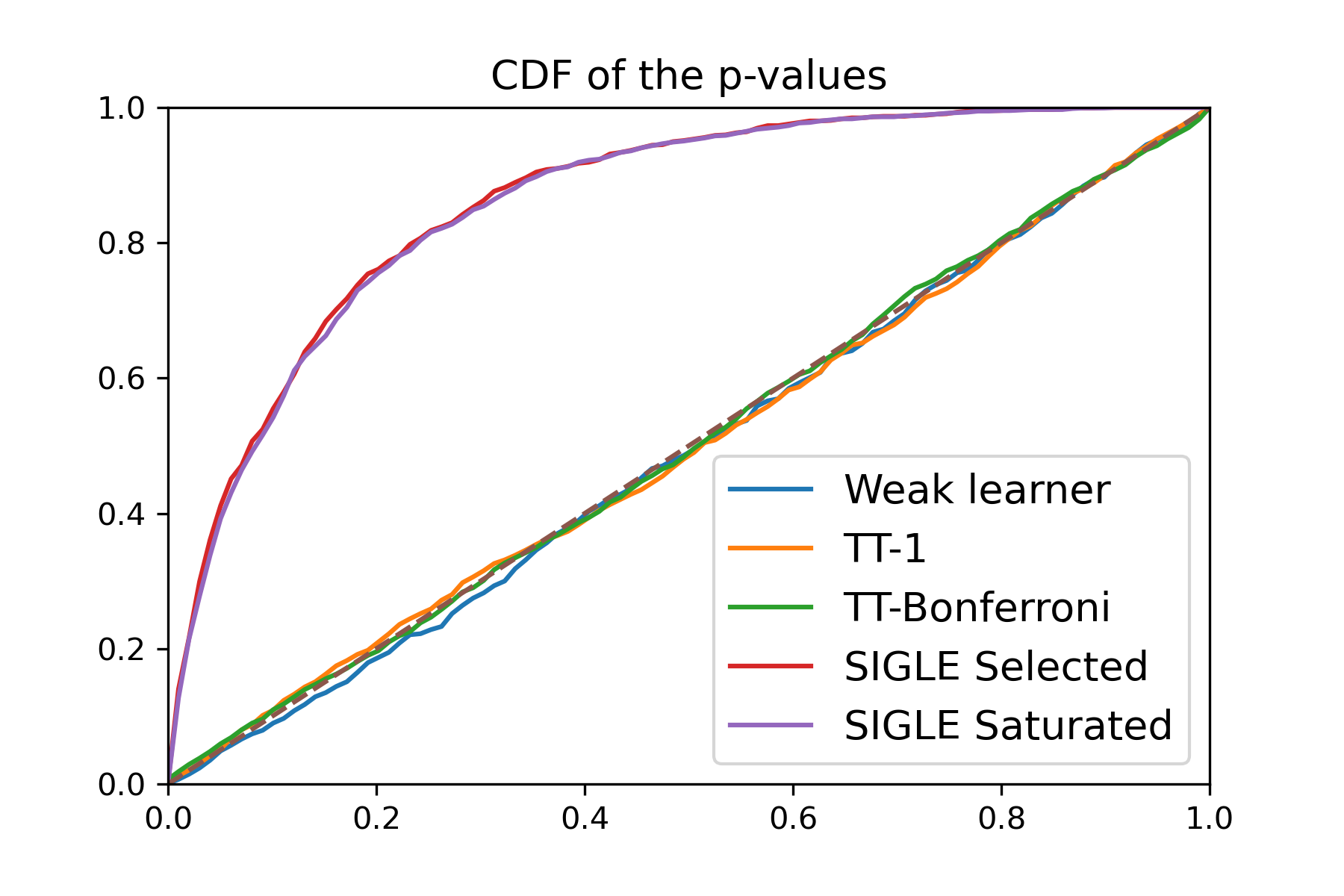}
        \caption[]%
        {{\small CDF of p-values for the alternative\\$\vartheta^*=0.04 \times\mathbf 1_{d}$.}}    
    \end{subfigure}
    \hfill
     \centering
    \begin{subfigure}[b]{0.49\textwidth}
        \centering
        \includegraphics[width=\textwidth]{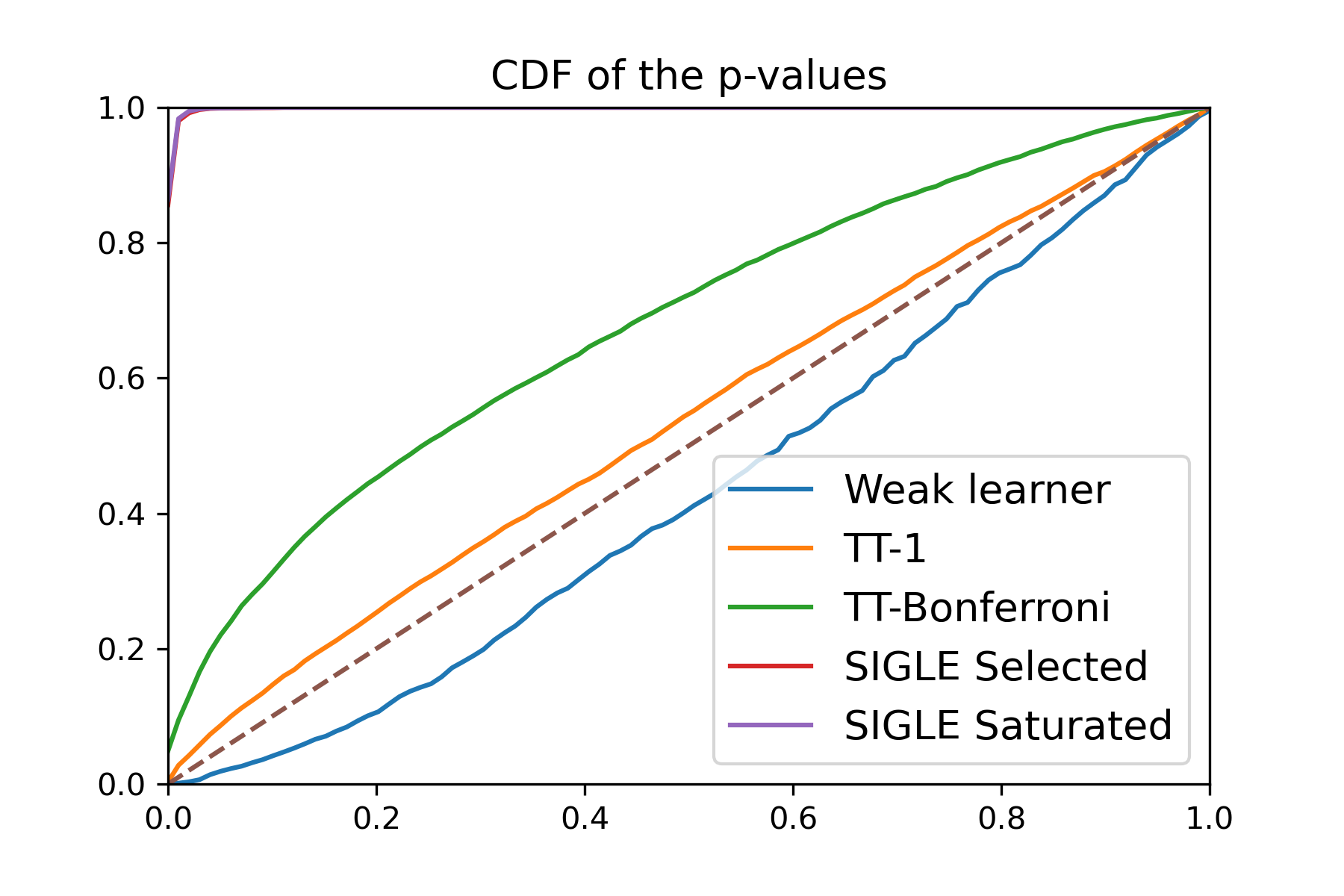}
        \caption[]%
        {{\small CDF of p-values for the alternative\\$\vartheta^*=0.3 \times\mathbf 1_{d}$.}}    
    \end{subfigure}
\caption{CDF of the p-values for the SIGLE procedures and the benchmark methods using the {\bf Setting 1} (cf. Table~\ref{table:settings}) for {\bf \color{red}disseminated alternatives}.}
\label{fig:power-disseminated-setting1}\end{center}
\vskip -0.2in
\end{figure}

\begin{figure}[ht!]
\vskip 0.2in
\begin{center}
 \centering
\begin{subfigure}[b]{0.49\textwidth}
\centering
\includegraphics[width=\textwidth]{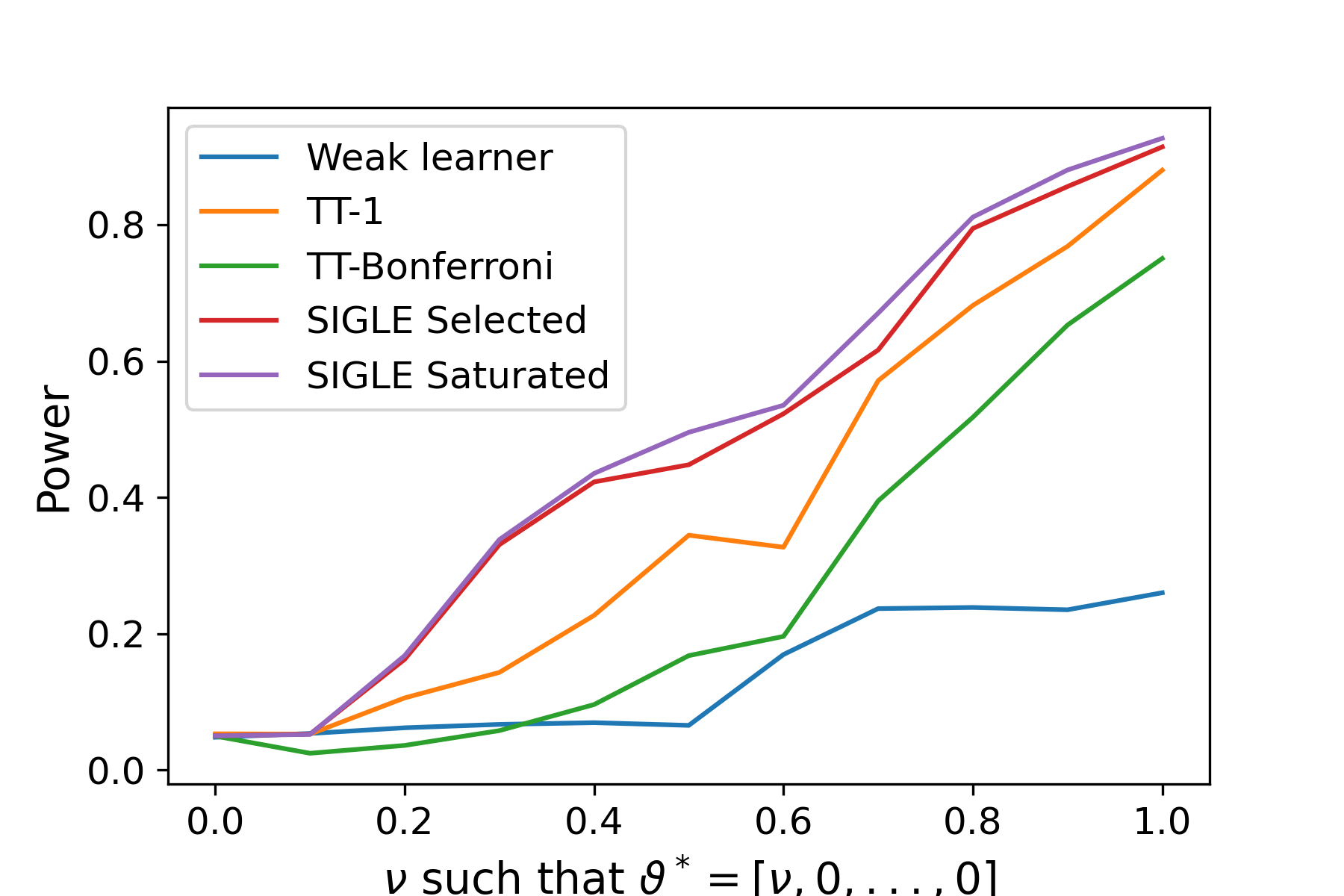}
\caption[]%
{{\small \color{red} \bf Localized alternatives.}} 
\end{subfigure}
    \hfill
     \centering
    \begin{subfigure}[b]{0.49\textwidth}
        \centering
        \includegraphics[width=\textwidth]{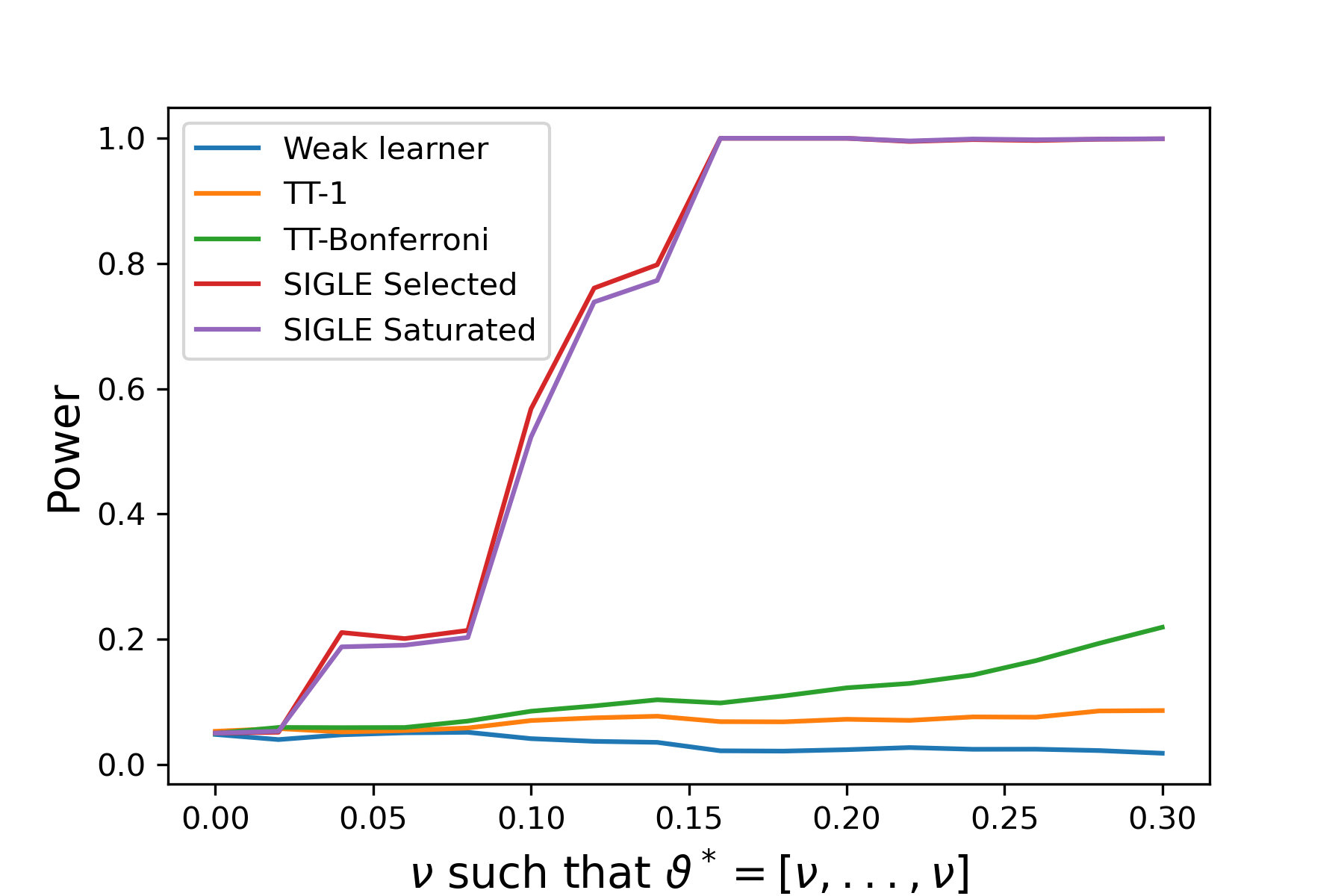}
        \caption[]%
     {{\small \color{red} \bf Disseminated alternatives.}}    
    \end{subfigure}
\caption{Comparison of the power of the SIGLE procedures and the benchmark methods using the {\bf Setting 1} (cf. Table~\ref{table:settings}) for tests with level $0.05$.}
\label{fig:power-nu-setting1}\end{center}
\vskip -0.2in
\end{figure}

\begin{figure}[ht!]
\vskip 0.2in
\begin{center}
 \centering
    \begin{subfigure}[b]{0.49\textwidth}
        \centering
        \includegraphics[width=\textwidth]{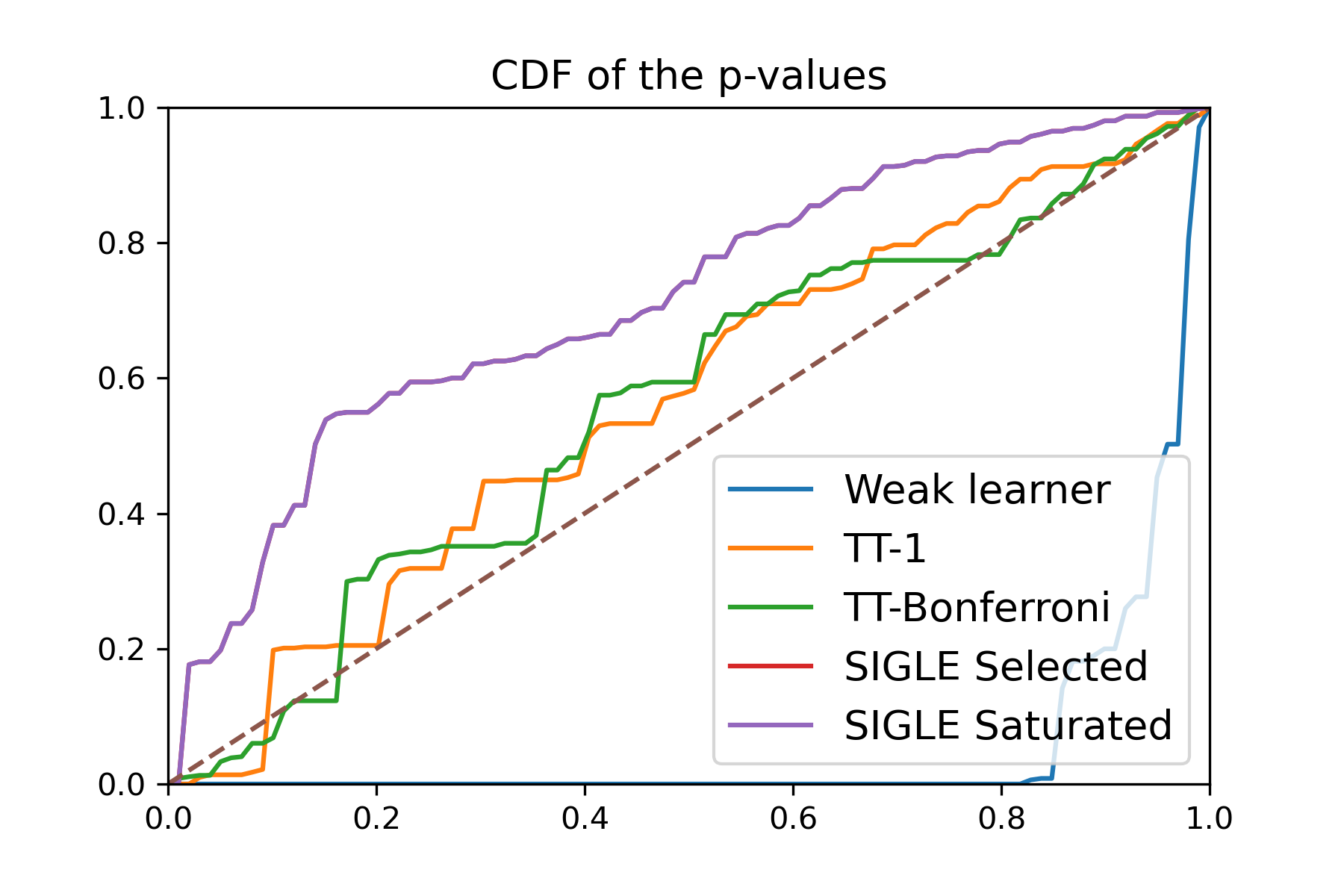}
        \caption[]%
        {{\small CDF of p-values for the {\bf \color{red} localized alternative} $\vartheta^*=[3 \, , \, 0 \, , \, \dots \, , \, 0]$.}}    
    \end{subfigure}
    \hfill
     \centering
    \begin{subfigure}[b]{0.49\textwidth}
        \centering
        \includegraphics[width=\textwidth]{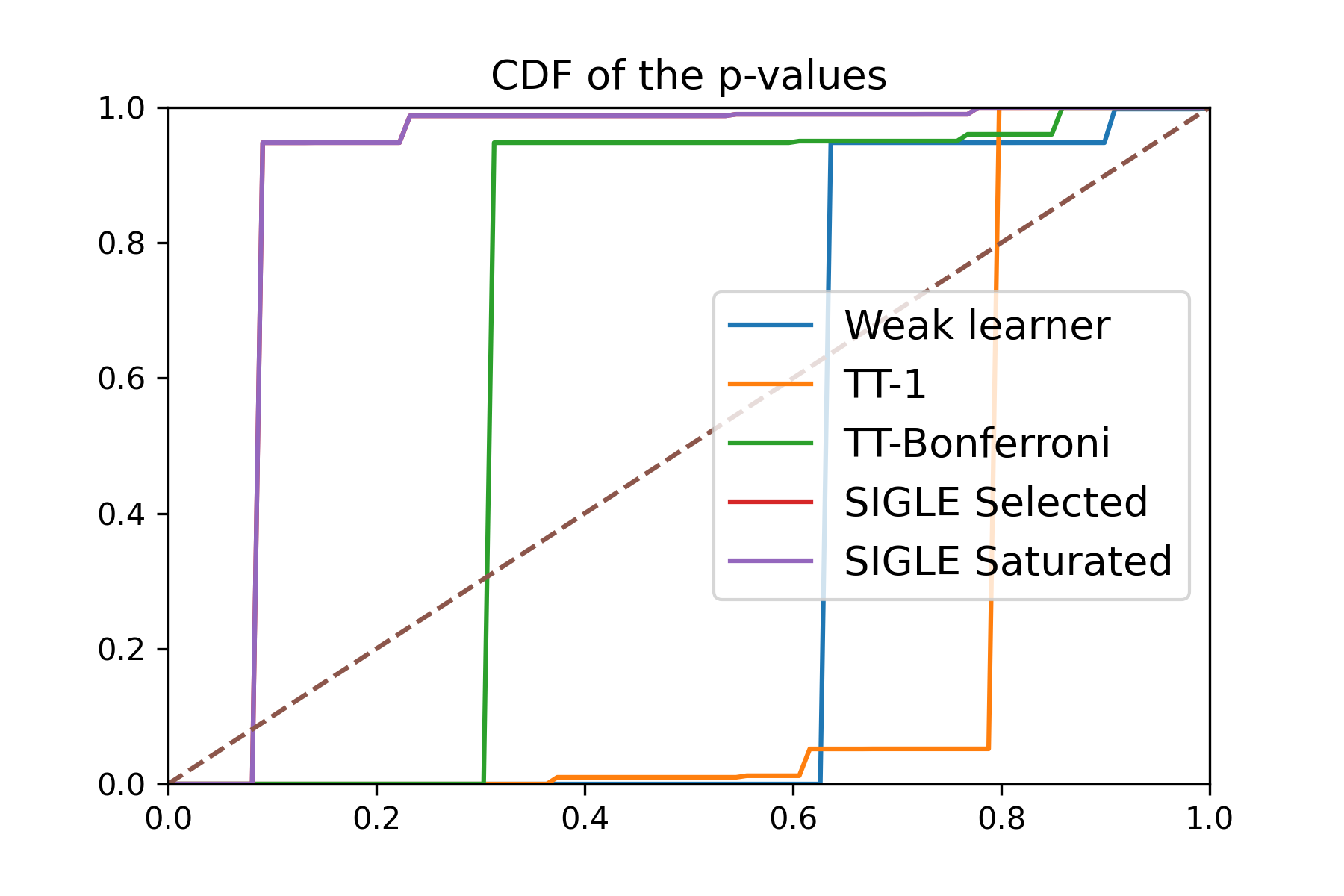}
        \caption[]%
        {{\small CDF of p-values for the {\bf \color{red} disseminated alternative}  $\vartheta^*=1.5 \times\mathbf 1_{d}$.}}    
    \end{subfigure}
\caption{CDF of the p-values for the SIGLE procedures and the benchmark methods using the {\bf Setting 2} (cf. Table~\ref{table:settings}).}
\label{fig:cdf-setting2}\end{center}
\vskip -0.2in
\end{figure}

\subsubsection{Computational time and implementation details}

{\bf Implementation of the SIGLE procedures.}
\begin{itemize}
\item \underline{SIGLE in the selected model.}\\
In the selected model, the SIGLE testing method requires to compute $\overline \theta (\theta^*_0)=\Psi(\mathbf X_M^{\top}\overline \pi^{\theta^*_0})$. Since we do not have a closed-form expression for $\Psi=\Xi^{-1}$, we first tried to learn this function by using a feed-forward neural network. We were not able to reach sufficient accuracy with this method and we proposed a gradient descent based approach to approximate $\overline \theta (\theta^*_0)$ from the estimate~$\widetilde \pi^{\theta^*_0}$ of~$\overline \pi^{\theta^*_0}$ (cf. Section~\ref{sec:benchmark}). This algorithm is fully described in Section~\ref{sec:sigle-graddescent}. Making use of a proper warm start, we found this method highly robust and accurate to compute~$\overline \theta (\theta^*_0)$.\\
\item \underline{SEI-SLR algorithm and speed of convergence.}\\
In the previous sections, we proved the correctness of the SEI-SLR algorithm: the states visited by the algorithm are asymptotically distributed according to the uniform measure on $E_M$. This is an asymptotic result and MCMC methods are known to converge slowly. In order to increase the speed of convergence of the SEI-SLR algorithm, we found very useful in practice to introduce a {\it repulsing force} in the markovian transition kernel. Denoting~$Y^{(t)}$ the visited state at time~$t$, we sample a candidate~$Y^c \sim P(Y^{(t)},\cdot)$ where we recall that~$P(Y^{(t)},\cdot)$ is the uniform distribution over the neighbours of~$Y^{(t)}$, i.e. the states of the hypercube~$\{0,1\}^N$ that differ from~$Y^{(t)}$ in exactly one coordinate. Instead of accepting the transition towards the candidate state~$Y^c$ if
\[1-\exp( - \frac{\Delta \mathcal E}{\mathrm T_t})\leq U_t,\]
where~$U_t \sim \mathcal U([0,1])$ and $\Delta \mathcal E:= \mathcal E(Y^c) - \mathcal E(Y^{(t)})$, we decide to set~$Y^{(t+1)}\leftarrow Y^c$ if and only if
\[\min\left( \, 1-\exp( - \frac{\Delta \mathcal E}{\mathrm T_t}) \, , \, 1-\mathcal E(Y^{(t)}) \, \right)\leq U_t.\]
The extra term $ 1-\mathcal E(Y^{(t)})$ in the acceptance rate acts like a {\it repulsion force}. If the current state $Y^{(t)}$ does not belong to the selection event, the energy $\mathcal E(Y^{(t)})$ is strictly positive. Nevertheless, if the neighbours of $Y^{(t)}$ have an energy which is larger than $\mathcal E(Y^{(t)})$, the algorithm may get stuck at $Y^{(t)}$ for some time before exploring other regions of the hypercube. Thanks to the extra term $ 1-\mathcal E(Y^{(t)})$, the acceptance rate is boosted whenever the current state is known to be outside of the selection event.
\end{itemize}

\subsubsection{Visualization of the SIGLE procedure in the selected model}
\label{ref:experiments}

Figure~\ref{fig:ellipse} provides a visualization of the SIGLE procedure in the selected model. We consider a design matrix $\mathbf X \in \mathds R^{100\times 5}$ with i.i.d. entries sampled according to a standard normal distribution. We consider the null hypothesis $\mathds H_0: " \theta^* = \mathbf 0"$. In Figure~\ref{fig:ellipse}.(a), we work under $\mathds H_0$ and we choose a regularization parameter~$\lambda=7$ in order to have a selected support of size 2 to be able to visualize in the plane the SIGLE method in the selected model. We calibrate our testing procedure empirically and we see on Figure~\ref{fig:ellipse}.(a) that $95\%$ of the states sampled using the rejection sampling method fall into the orange ellipse, meaning that our test has level $5\%$. On Figure~\ref{fig:ellipse}.(b), we consider a localized alternative by considering $\vartheta^*=[0.5 \, , \, 0 \, ,\, \dots \, , \, 0]$ and we choose $\lambda=8$ in order to have $|M|=2$. In this case, the number of states falling into the orange ellipse is less than $95\%$ which means that we reject the null hypothesis.


\begin{figure}[ht!]
\vskip 0.2in
\begin{center}
 \centering
    \begin{subfigure}[b]{0.49\textwidth}
        \centering
        \includegraphics[width=\textwidth]{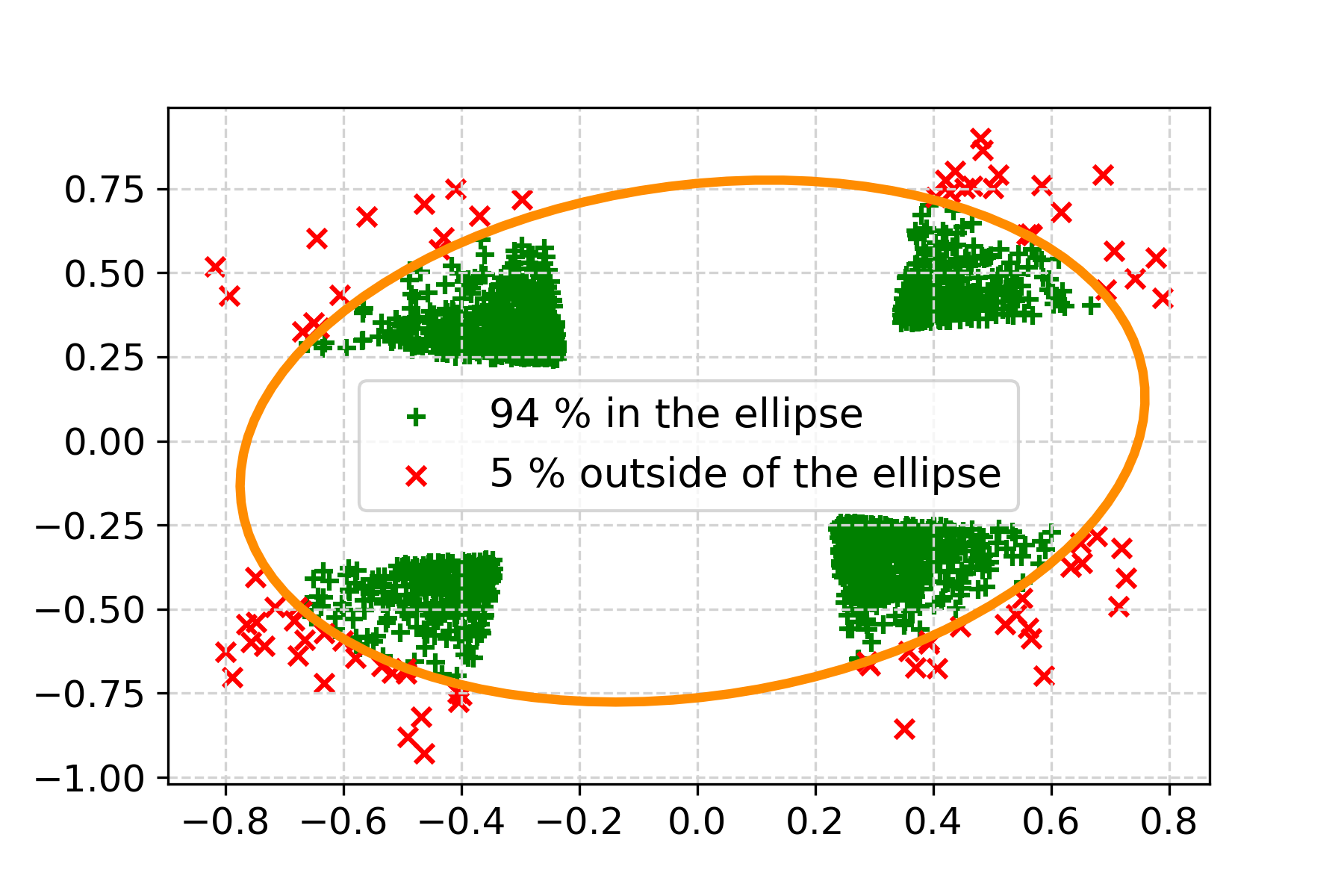}
        \caption[]%
        {{\small $\vartheta^*=\mathbf 0$, $\lambda=7$}}    
    \end{subfigure}
    \hfill
     \centering
    \begin{subfigure}[b]{0.49\textwidth}
        \centering
        \includegraphics[width=\textwidth]{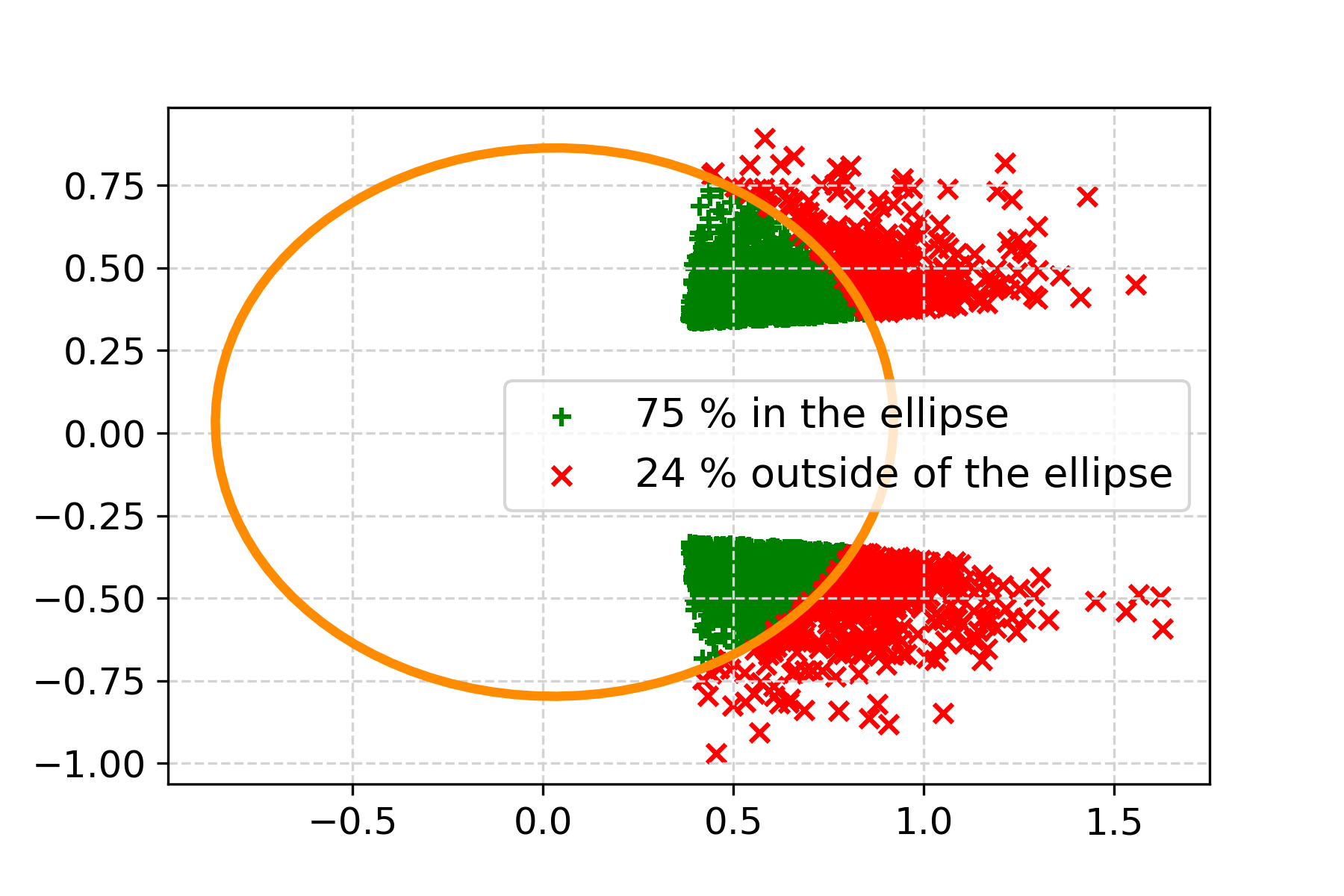}
        \caption[]%
        {{\small $\vartheta^*=[0.5,0, \dots,0]$, $\lambda=8$}}      
    \end{subfigure}
\caption{The orange ellipse represents the set of parameter $\theta \in \mathds R^s$ such that $\|\widetilde G_N^{-1/2}H_N( \widetilde \theta)(\theta-\widetilde \theta)\|_2^2=q_{1-\alpha}$ where $q_{1-\alpha}$ is the empirical quantile of order $1-\alpha$ of the test statistic of the SIGLE procedure in the selected model under the null $"\theta^*=0"$. For each $t$, we plot the MLE $\Psi(\mathbf X_M^{\top}Y^{(t)})$ with a green plus if the point falls into the orange ellipse and with a red cross otherwise. 
}
\label{fig:ellipse}
\end{center}
\vskip -0.2in
\end{figure}

\subsection{Discussion and final remarks}
\paragraph{Calibration.}

Despite the method proposed by \cite{taylorGLM} lacks theoretical guarantees, our experiments have shown that it is most of the time correctly calibrated. The calibration of SIGLE requires to sample under the null, which makes the method computationally more heavy.

\medskip

\paragraph{Power.}
Our experiments have shown that the {\it empirically calibrated} SIGLE procedures seem to be systematically more powerful compared to the approach from \cite{taylorGLM}. We would like to point out two main possible reasons explaining the lack of power of the PSI method from \cite{taylorGLM}.
\begin{itemize}
\item[$(i)$] We are tackling a simple hypothesis testing problem while the method proposed in~\cite{taylorGLM} is more naturally suited to address composite testing problems (typically testing the nullity of a specific coordinate of $\theta^*$). Note that the SIGLE methods cannot easily tackle single testing problems since the whole parameter $\pi_0^* $ (resp. $\theta^*_0$ in the selected model) is need to estimate $\overline G_N(\pi^*_0)$ (resp. $\overline G_N(\pi^*_0)$ and $\overline \theta(\theta^*_0)$). When deriving their PSI method, \cite{taylorGLM} face a similar issue and propose to use a plug-in approach by remplacing the unknown parameter $\theta^*$ by the lasso solution. A similar plug-in approximation for SIGLE could be investigated and this research direction is left for future work.

\item[$(ii)$] The method proposed by~\cite{taylorGLM} is motivated by non-rigorous computations that aim at characterizing the distribution of the debiased lasso solution $\underline \theta$ conditional on the selection event $E_M^{S_M}$ (we refer to Section~\ref{sec:debiasing} for details). It is well-known that conditioning on both the active variables and the vector of dual signs can lead to less powerful testing procedures. This statement can be made rigorous through the concept of {\it leftover Fisher information} (see Section~\ref{sec:conditional-signs} or \cite{fithian2014optimal} for details). As summarized in~\cite{fithian2014optimal}, "{\it on average, the price of conditioning on the [signs]~$S_M$ -- the price of selection -- is the information~$S_M$ carries about $\theta^*$}". Roughly speaking, even if the observed vector of dual signs is very surprising under the null, the method from~\cite{taylorGLM} will not reject the null hypothesis unless we are surprised anew by looking at $\underline \theta$. On the contrary, the SIGLE methods rely on the characterization of some test statistic conditional on $E_M$ (without conditioning on the signs).\\
In the Linear LASSO, the same situation arises and in~\cite{sun16}, the authors proved that one can rely on the work done conditional on $E_M^{S_M}$ in order to derive a more powerful testing method (at least on average) at the price of an additional computational cost. In this case, the linear transformation of the response vector is not distributed as a Gaussian truncated to an interval (as conditional on $E_M^{S_M}$) but is now a truncated Gaussian with a truncation set being a union of intervals. One important remark is that contrary to the Linear LASSO, the method from~\cite{taylorGLM} cannot be easily adapted to get more power by working directly on $E_M$. The reason is that the test statistic itself depends on the vector of dual signs $S_M$ (and not only the bounds of the truncation interval).
\end{itemize}


\begin{figure}[!h]
\centering
\begin{tikzpicture}[>=stealth,yscale=2,xscale=1.4]  
\node(?1) at (0,0)[rectangle,draw,text width=4cm,text centered] {What is the paradigm considered?};
\node(l?1) at (-3.5,0)[rectangle,draw,text width=2cm,text centered] {SIGLE};
\node(?2) at (0,-1)[rectangle,draw,text width=4cm,text centered] {What type of testing problem are you tackling?};
\node(l?2) at (3.5,-1)[rectangle,draw,text width=2cm,text centered] {TT};
\node(?3) at (0,-2.1)[rectangle,draw,text width=4cm,text centered] {At least one of the two parameters $d$ or $N$ is small?};
\node(l?3) at (-3.5,-2.1)[rectangle,draw,text width=2cm,text centered] {SIGLE};
\node(?4) at (0,-3.3)[rectangle,draw,text width=4cm,text centered] {Are you ready to sacrifice computational time to increase power?};
\node(r?4) at (-3.5,-3.3)[rectangle,draw,text width=2cm,text centered] {SIGLE};
\node(l?4) at (3.5,-3.3)[rectangle,draw,text width=2cm,text centered] {TT};
\node(tr?1) at (1,-0.5)[text width=7cm,text centered] {\it Selected model};
\node(tl?1) at (-2.1,0.03)[text width=2cm,text centered] {\it Saturated model};
\node(tl?2) at (2.1,-0.98)[text width=2cm,text centered] {\it Composite testing};
\node(tr?2) at (1.54,-1.5)[text width=7cm,text centered] {\it Simple hypothesis testing};
\node(tl?3) at (-2.1,-2)[text width=2cm,text centered] {\it Yes};
\node(tr?3) at (0.3,-2.68)[text width=7cm,text centered] {\it No};
\node(tl?4) at (2.1,-3.2)[text width=2cm,text centered] {\it No};
\node(tr?4) at (-2.1,-3.2)[text width=7cm,text centered] {\it Yes};
\draw[->] (?1) -- (?2);
\draw[->] (?1) -- (l?1);
\draw[->] (?2) -- (?3);
\draw[->] (?2) -- (l?2);
\draw[->] (?3) -- (l?3);
\draw[->] (?4) -- (l?4);
\draw[->] (?4) -- (r?4);
\draw[->] (?3) -- (?4);
\end{tikzpicture}
\caption{Choosing the PSI testing method that matches your setting: SIGLE (this paper) or TT (from~\cite{taylorGLM}).}
\label{fig:chooseYOURtest}
\end{figure}
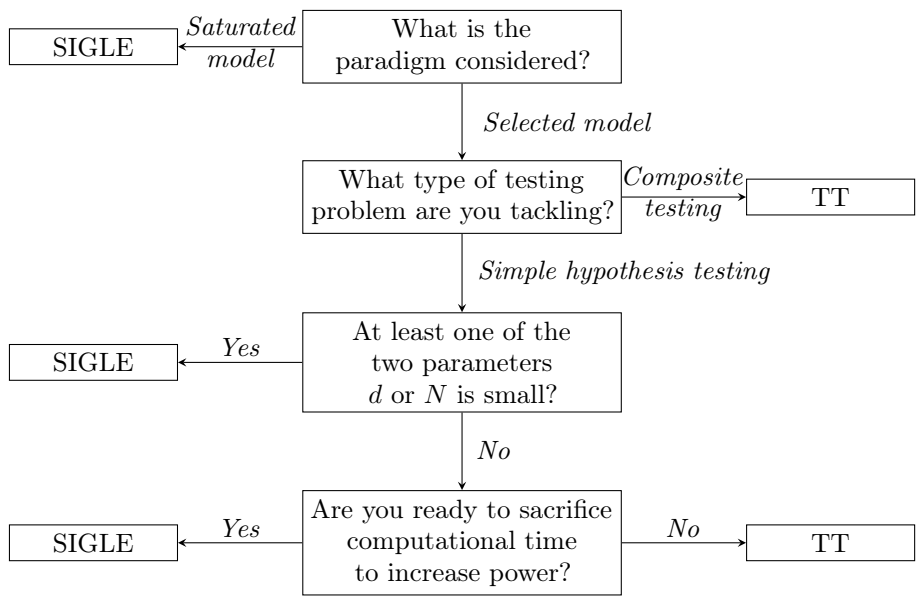

\paragraph{Conclusion.} The SIGLE procedures require to use either the rejection sampling method or the SEI-SLR algorithm to estimate the matrix $\overline G_N(\pi^*_0)$ (and the parameter $\overline \theta(\theta^*_0)$ in the selected model) and to estimate the parameter $w_{N,1-\alpha}$ needed to define the rejection region. This sampling step is the main computational burden of the SIGLE procedures. On the contrary, the approach of~\cite{taylorGLM} does not require such sampling stage and only requires to compute the bounds of the truncation interval of the distribution the $\eta^{\top}\underline \theta$ for some fixed vector~$\eta \in \mathds R^s$ under the null.\\
Figure~\ref{fig:chooseYOURtest} summarizes the main differences between the methods proposed in this paper and the one from~\cite{taylorGLM} and provides an easy way to select the best method for a given setting. This organizational chart stresses that when $d$ is small, the rejection sampling method allows to efficiently sample states from the conditional distribution $\overline {\mathds P}_{\theta^*}$ while when $N$ is small, the SEI-SLR algorithm allows to efficiently sample states uniformly distributed on $E_M$. In both cases, the SIGLE methods can be used with a small computational time and they should be preferred to get more powerful methods. 

\section{Conditional Central Limit Theorems for SLR}
\label{sec:asymptotic}
\label{sec:CLT}

\subsection{Preliminaries}
\label{assumption:design}

Before presenting our conditional CLTs, let us present the framework in which we state our asymptotic results. Let $(d_N)_{N\in\mathds N}$ be a non-decreasing sequence of positive integers converging to $d_{\infty} \in \mathds N \cup \{+\infty\}$ and let $s\in [d_1,d_{\infty}]\cap \mathds N$. For any $N$, we consider $[\vartheta^*]^{(N)} \in \mathds R^{d_N}$, $\lambda^{(N)}>0$, $M^{(N)}\subseteq [d_N]$ with cardinality $s$ and a design matrix $\mathbf X^{(N)}\in \mathds R^{N \times d_N}$. We recall the definitions of the selection event $E_M^{(N)}$ corresponding to the tuple $(\lambda^{(N)}, M^{(N)}, \mathbf X^{(N)})$ and of the conditional probability distribution $\overline {\mathds P}_{\pi^*}^{(N)}$ given in Section \ref{sec:intro-infeprocedures}.
We assume that it holds
\begin{itemize}
\item $\displaystyle K:=\sup_{N\in \mathds N}\max_{i\in [N],j\in M^{(N)}} |\mathbf X_{i,j}^{(N)} |  < \infty$,
\item there exist constants $C,c>0$ (independent of $N$) such that for any $N \in \mathds N$, 
\[cN\leq\lambda_{\min}(\big[\mathbf X_{M^{(N)}}^{(N)}\big]^{\top}\mathbf X_{M^{(N)}}^{(N)})\leq \lambda_{\max}(\big[\mathbf X_{M^{(N)}}^{(N)}\big]^{\top}\mathbf X_{M^{(N)}}^{(N)})\leq C N.\]
\end{itemize}
{\bf Remark.} Note that the latter assumption holds in particular if the matrices $\big({\mathbf X^{(N)}}/{\sqrt N}\big)_{N\geq1}$ satisfy (uniformly) the so-called $s$-Restricted Isometry Property~(RIP) condition \citep[cf.][Definition 7.10]{wainwright2019high}. Let us recall that a matrix $A \in \mathds R^{N\times p}$ satisfies the $s$-RIP condition if there exists a constant~$\delta_s \in (0,1)$ such that for any $N\times s$ submatrix~$A_s$ of~$A$, it holds
\[1-\delta_s \leq \lambda_{\min}(A_s^{\top}A_s) \leq \lambda_{\max}(A_s^{\top} A_s) \leq 1+\delta_s.\]
In Section~\ref{sec:CLT-saturated}, we start by presenting our first CLT for $\big[\mathbf X_M^{(N)}\big]^{\top}Y$ where $Y$ is distributed according to $\overline {\mathds P}_{\pi^*}^{(N)}$. 
Thereafter, we prove in Section~\ref{sec:CLT-selected} a CLT for the conditional unpenalized MLE $\widehat \theta$ working with the design $\mathbf X_M^{(N)}$ (see Eq.\eqref{eq:uncondi-MLE-XM}). 
\medskip

\noindent \begin{minipage}{.6\textwidth}
The proofs of our conditional CLTs make use of \cite[Thm.1]{bardet2008dependent} and rely on triangular arrays $\vec \xi:=\left((\xi_{i,N})_{i\in[N]},\, N \in \mathds N \right)$ where~$\xi_{i,N}$ is a random vector in~$\mathds R^s$ and is a function of the deterministic quantities~$\lambda^{(N)}$, $\mathbf X^{(N)}$, $M^{(N)}$ and of the random variable~$Y$ with probability distribution
$\overline{\mathds P}^{(N)}_{\pi^*}$. Most dependent CLTs have been proven for causal time series (typically satisfying some mixing condition) and are not well-suited to our case since conditioning on the selection event introduces a complex dependence structure. 
\end{minipage}\hfill
\begin{minipage}{.37\textwidth}
\begin{tabular}{cccccc}
$\xi_{1,1}$&&&&&\\
$\xi_{1,2}$ & $\xi_{2,2}$ &&&&\\
$\xi_{1,3} $& $\xi_{2,3}$ & $\xi_{3,3}$&&&\\
$\dots$ & $\dots$ & $\dots$ & $\dots$ &  &\\
$\xi_{1,N} $& $\xi_{2,N}$ & $\xi_{3,N}$&$\dots$&$\xi_{N,N}$&\\
$\dots$ & $\dots$ & $\dots$ & $\dots$ &  $\dots$ &$\dots$
\end{tabular}
\end{minipage}
\smallskip

\noindent The dependent Lindeberg CLT from \cite[Thm.1]{bardet2008dependent} gives us the opportunity to find conditions involving mainly the covariance matrix of $Y$ under which our conditional CLTs hold. More precisely, we provide conditions ensuring that the lines of the $\mathds R ^s$-valued process indexed by a triangular system $\vec \xi$ satisfy some Lindeberg's condition. Let us stress that we discuss the assumptions of the theorems presented in Sections~\ref{sec:CLT-saturated} and \ref{sec:CLT-selected} in Section~\ref{sec:discussion-CLT}.
\bigskip

\noindent To alleviate this notational burden, we will not specify the dependence on $N$ in the remainder of the paper, meaning that we will simply refer to $\mathbf X^{(N)}$, $M^{(N)}$, $d_N$, $[\vartheta^*]^{(N)}, \overline {\mathds P}^{(N)}_{\pi^*}, \dots$ as $\mathbf X$, $M$, $d$, $\vartheta^*, \overline {\mathds P}_{\pi ^*}, \dots$. Nevertheless, let us stress again that the integer $s$ is fixed and does not depend on $N$ in this paper.

\subsection{A conditional CLT for the saturated model}
\label{sec:CLT-saturated}

We aim at providing a simple hypothesis testing procedure and a confidence interval for the parameter~$\mathbf X^{\top}_M\pi^*$ conditionally on the selection event~$E_M$. To do so, we prove in this section a CLT for~$\mathbf X_M^{\top} Y$ when~$Y$ is a random variable on~$\{0,1\}^N$ following the multivariate Bernoulli distribution with parameter~$\pi^*\in [0,1]^N$ conditionally on the event~$\{Y \in E_M\}$. Let us first recall the notation for the distribution of~$Y$ conditional on~$E_M$ in the saturated model \[\overline {\mathds P}_{\pi^*}(Y) \propto \mathds 1_{E_M}(Y)  \mathds P_{ \pi^*}(Y), \,\] where the symbol $\propto$ means ‘proportional to'. In the following, we will denote by~$\overline{\mathds E}_{\pi^*}$ the expectation with respect to~$\overline {\mathds P}_{\pi^*}$. With Theorem~\ref{prop:CLT}, we give a conditional CLT that holds under some conditions that involve in particular the covariance matrix of the response~$Y$ under the distribution~$\overline{\mathds P}_{\pi^*}$, namely
\[
\overline \Gamma^{\pi^*}: = \overline {\mathds E}_{\pi^*}\left[ (Y-\overline \pi^{\pi^*})(Y-\overline \pi^{\pi^*})^{\top}\right]\in [-1,1]^{N\times N},\]
where~$\overline \pi ^{\pi^*} = \overline {\mathds E}_{\pi^*}[Y]$.

\begin{thm} \label{prop:CLT}
We keep the notations and assumptions from Section~\ref{assumption:design}. We denote~$\pi^* = \sigma(\mathbf X \vartheta^*)$ and~$Y$ the random vector taking values in~$\{0,1\}^N$ and distributed according to~$\overline {\mathds P}_{\pi^*}$.   
Assume further that
\begin{enumerate}
\item $\displaystyle \sum_{i=1}^N  \sqrt{  \|(\mathbf X_{[i-1],M})^{\top}\overline \Gamma^{\pi^*}_{[i-1],[i-1]} \mathbf X_{[i-1],M} \|_F\left(1- 2 \overline  \pi_i^{\pi^*}\right)^2} \underset{N\to +\infty}{=} o (  N),$
\item there exists~$\overline \sigma^2_{\min}>0$ such that $\overline \pi^{\pi^*}_i (1-\overline \pi^{\pi^*}_i )\geq \overline \sigma^2_{\min}$ for all $i \in [N]$.
\end{enumerate}
Then it holds
\[u^{\top}[\overline G_N(\pi^*)]^{-1/2} \mathbf X_M^{\top}( Y - \overline \pi^{\pi^*}) \overset{(d)}{\underset{N\to + \infty}{\longrightarrow}} \mathcal N(0,1),\]
where~$u$ is a unit~$s$-vector and where $\overline G_N (\pi^*):=\mathbf X_M^{\top} \mathrm{Diag}((\overline \sigma^{\pi^*})^2) \mathbf X_M$ with $(\overline \sigma^{\pi^*})^2:= \overline \pi^{\pi^*} \odot (1-\overline \pi^{\pi^*})$.
\end{thm}

\subsection{A conditional CLT for the selected model}
\label{sec:CLT-selected}

We now work under the condition that there exists $\theta^*\in\mathds R ^s$ such that $\mathbf X_M \theta^*=\mathbf X \vartheta^*$. 
Given some~$Y \in \{0,1\}^N$ and provided that~$\mathbf X_M^{\top} Y \in \mathrm{Im}(\Xi)$,~$\Psi(\mathbf X_M^{\top}Y)$ is the MLE $\widehat \theta$ of the unpenalized logistic model. \citet[Theorem~1]{candesMLE} ensures that the MLE exists asymptotically almost surely when $Y$ is distributed as $\mathds P_{\theta^*}$. When the distribution of $Y$ is $\overline {\mathds P}_{\theta^*}$, we prove in Section~\ref{proof:thmMLE} a weaker counterpart of this result showing that for $N$ large enough, the MLE exists with high probability.

We aim at providing a simple hypothesis testing procedure and a confidence interval for the parameter~$\theta^*$ conditionally on the selection event. To do so, we first prove a CLT for the MLE~$\widehat \theta$ when~$Y$ is distributed according to $\overline {\mathds P}_{\theta^*}$ ({\it i.e.,} $Y$ is a random variable on~$\{0,1\}^N$ following the multivariate Bernoulli distribution with parameter~$\sigma(\mathbf X_M \theta^*)$ conditioned on the event~$\{Y \in E_M\}$). The unconditional MLE~$\widehat \theta$ (using only the features indexed by $M$) is known to be consistent and asymptotically efficient meaning that when~$Y$ is distributed according to~$\mathds P_{\theta^*}$,
\begin{equation}\label{eq:CLT-uncondi} u^{\top} [  H_N(\theta^*)]^{1/2}(\widehat \theta- \theta^*) \underset{N\to + \infty}{\overset{(d)}{\longrightarrow}} \mathcal N(0,1),\end{equation}
where~$u$ is a unit $s$-vector and where
\[H_N(\theta) :=\mathbf X_M^{\top} \mathrm{Diag}(\sigma'(\mathbf X_M  \theta)) \mathbf X_M=\mathbf X_M^{\top} \mathrm{Diag}((\sigma^{\theta})^2) \mathbf X_M,\]
is the Fisher information matrix with~$(\sigma^{\theta})^2:=  \pi^{\theta}\odot (1- \pi^{\theta})$ and $\pi^{\theta}=\mathds E_{\theta}[Y]$. 

In the following, we will consider the natural counterpart of the Fisher information matrix~$H_N(\theta^*)$ when we work under the conditional distribution~$\overline {\mathds P}_{\theta^*}$,
\[\overline G_N (\theta^*):=\mathbf X_M^{\top} \mathrm{Diag}((\overline \sigma^{\theta^*})^2) \mathbf X_M,\quad  (\overline \sigma^{\theta^*})^2:= \overline \pi^{\theta^*}\odot (1-\overline \pi^{\theta^*}),\, \overline \pi^{\theta^*}=\overline{\mathds E}_{\theta^*}[Y].\]
Theorem~\ref{thm:MLEasymptotic} proves that the MLE~$\widehat \theta$ under the conditional distribution~$\overline {\mathds P}_{\theta^*}$ also satisfies a CLT analogous to Eq.\eqref{eq:CLT-uncondi} by replacing respectively~$\theta^*$ and~$H_N(\theta^*)^{1/2}$ by~$\overline \theta(\theta^*)$ (cf. Eq.\eqref{def:thetabar}) and~$[\overline  G_N(\theta^*)]^{-1/2}H_N(\overline \theta(\theta^*))$. This conditional CLT holds under some conditions that involve in particular the covariance matrix of the response~$Y$ under the distribution~$\overline{\mathds P}_{\theta^*}$, namely
\[
\overline \Gamma^{\theta^*} = \overline {\mathds E}_{\theta^*}\left[ (Y-\overline \pi^{\theta^*})(Y-\overline \pi^{\theta^*})^{\top}\right]\in [-1,1]^{N\times N}.\]

\begin{thm}  \label{thm:MLEasymptotic} We keep the notations and assumptions from Section~\ref{assumption:design}. Let us consider~$\theta^* \in \mathds R^s$ and let us denote by~$Y$ the random vector taking values in~$\{0,1\}^N$ and distributed according to~$\overline {\mathds P}_{\theta^*}$.   
Assume further that
\begin{enumerate}
\item $\displaystyle \sum_{i=1}^N  \sqrt{  \|(\mathbf X_{[i-1],M})^{\top}\overline \Gamma^{\theta^*}_{[i-1],[i-1]} \mathbf X_{[i-1],M} \|_F\left(1- 2 \overline  \pi_i^{\theta^*}\right)^2} \underset{N\to +\infty}{=} o (  N),$
\item there exists~$\overline \sigma^2_{\min}>0$ such that for any $N$ and for any $i \in [N]$, \[\overline \pi^{\theta^*}_i (1-\overline \pi^{\theta^*}_i ) \wedge \sigma'(\mathbf X_{i,M} \overline \theta(\theta^*))\geq \overline \sigma^2_{\min}.\]
\item there exists some $\mathfrak K>0$ such that for any $N\in \mathds N$, \[\mathrm{Tr}\left[\overline G_N^{-1}\mathbf X_M^{\top} \overline {\Gamma}^{\theta^*}\mathbf X_M  \right] < \mathfrak K .\]
\end{enumerate}
Then, \[u^{\top} [\overline  G_N(\theta^*)]^{-1/2}H_N(\overline \theta(\theta^*))\big(\widehat \theta- \overline \theta(\theta^*)\big) \underset{N\to + \infty}{\overset{(d)}{\longrightarrow}} \mathcal N(0,1),\]
where~$u$ is a unit~$s$-vector and where we recall that~$\widehat \theta = \Psi(\mathbf X_M^{\top}Y)$ is the MLE.
\end{thm}

The proof of Theorem~\ref{thm:MLEasymptotic} can be found with full details in Section~\ref{proof:thmMLE} and we only provide here the main arguments. First we use Theorem~\ref{prop:CLT} that shows that the distribution of~$[\overline G_N(\theta^*)]^{-1/2}L_{N}(\overline \theta,Z^M)$ is asymptotically Gaussian using a Lindeberg Central Limit Theorem for dependent random variables from~\cite{bardet2008dependent}. Then, we show that for $N$ large enough, the following holds with high probability: the MLE~$\widehat \theta$ exists and is contained within an ellipsoid centered at~$\overline \theta$ with vanishing volume. This kind of result has already been studied in~\cite{Liang2012MaximumLE} but the proof provided by Liang and Du is wrong (Eq.(3.7) is in particular not true). As far as we know, we are the first to provide a correction of this proof in Section~\ref{proof:thmMLE}. Let us also stress that working with the conditional distribution $\overline {\mathds P}_{\theta^*}$ brings extra-technicalities that need to be handled carefully.\\
Using this consistency of~$\widehat \theta$ together with the smoothness of the map~$\theta \mapsto L_{N}( \theta,Z^M)$, one can convert the previously established result for 
\[
[\overline G_N(\theta^*)]^{-1/2}L_{N}(\overline \theta,Z^M)=[\overline G_N(\theta^*)]^{-1/2}(L_{N}(\overline \theta,Z^M)-L_{N}(\widehat \theta,Z^M))
\,,
\]
into a CLT for~$\widehat \theta$.

\subsection{Discussion}

\label{sec:discussion-CLT}
In this section, we discuss {\it informally} the assumptions of both Theorems~\ref{prop:CLT} and~\ref{thm:MLEasymptotic}. The conditions of Theorems~\ref{prop:CLT} and~\ref{thm:MLEasymptotic} can be seen at first glance as arcane or restrictive. Without pretending that those conditions are easy to check in practice, looking at these requirements through the lens of the usual asymptotic alternative where~$\vartheta^*$ itself depends on~$N$ gives a different perspective. Such assumption on~$\vartheta^*$ has been considered for example in \cite{bunea} or \cite[Section 3.1]{taylorGLM}. Following this line of work, we consider that $\vartheta^*=\alpha_N^{-1}\beta^*$ where each entry of~$\beta^*$ is independent of~$N$ and~$(\alpha_N)_N$ is a sequence of increasing positive numbers such that~$\alpha_N \underset{N\to \infty}{\to}+\infty$. We further assume~$\beta^*$ is $s^*$-sparse with support~$M^*$ (and with $s^*$ independent of $N$). Let us analyze the conditions of our theorems in this framework by considering that $E_M = \{0,1\}^N$ (i.e. there is no conditioning). Then, condition 3 of Theorem~\ref{thm:MLEasymptotic} holds automatically since in this case $\mathbf X_M^{\top}\overline \Gamma^{\theta^*}\mathbf X_M = H_N(\theta^*)$ and $\overline G_N^{-1} = [H_N(\theta^*)]^{-1}$, meaning that $\mathfrak K=s$ works. The condition 2 of Theorems~\ref{prop:CLT} and~\ref{thm:MLEasymptotic} holds also automatically since $\alpha_N\underset{N\to \infty}{\to}+\infty$, while the condition 1 is satisfied as soon as $\alpha_N \underset{N\to \infty}{=} \omega(N^{1/2})$. \\
\indent The quantity $\alpha_N$ is quantifying the dependence arising from conditioning on the selection event: the weaker the dependence between the entries of the random response $Y\sim \overline {\mathds P}_{\pi^*}$, the smaller $\alpha_N$ can be chosen while preserving the asymptotic normal distribution. Note that in the papers \cite{bunea} and \cite[Section 3.1]{taylorGLM}, the authors typically consider the case where $\alpha_N \underset{N\to \infty}{\sim}  N^{1/2}$, corresponding to the regime at which the validity of our CLTs may be questioned based on the simple analysis previously conducted. Nevertheless, we stress that stronger assumptions on the design could allow to bypass this apparent limitation. A promising line of investigation is the following: taking a closer at the proofs of Theorems~\ref{prop:CLT} and~\ref{thm:MLEasymptotic}, one can notice that the condition 1 can actually be weakened by
\[\min_{\nu \in \mathfrak{S}_N}\sum_{i=1}^N  \sqrt{  \|(\mathbf X_{\nu([i-1]),M})^{\top}\overline \Gamma^{\pi^*}_{\nu([i-1]),\nu([i-1])} \mathbf X_{\nu([i-1]),M} \|_F\left(1- 2 \overline  \pi_{\nu(i)}^{\pi^*}\right)^2} \underset{N\to +\infty}{=} o (  N),\]
where $\mathfrak{S}_N$ is the set of permutations of $[N]$.

\newpage

\bibliography{sample} 

\clearpage

\appendix

{\bf Guidelines for the Appendix.}
\begin{itemize}
\item {\bf Section~\ref{sec:MLEvsBIAS}: Regularization bias and conditional MLE.}\par
In this first section of the Appendix, we shed light on the difference between SIGLE and the work of~\cite{taylorGLM}. Both methods have already been compared on the practical side in Section~\ref{sec:experiments}. In Section~\ref{sec:MLEvsBIAS}, we take a step back to understand the different paradigms considered in these two approaches. We describe the strengths and drawbacks of both methods, highlighting the fact that the method of~\cite{taylorGLM} rely on non rigorous computations while SIGLE can be proved (see Section~\ref{apdx:hypotest_CLT}) to be asymptotically valid under the set of assumptions presented in Section~\ref{sec:discussion-CLT}.
\item {\bf Section~\ref{apdx:hypotest_CLT}: Theoretical guarantees for SIGLE in SLR.}\par
In this section, we show how the conditional CLTs of Section~\ref{sec:CLT} can be used to prove that the SIGLE methods are asymptotically correctly calibrated when the restrictive conditions of Section~\ref{sec:discussion-CLT} are satisfied.
\item {\bf Section~\ref{apdxsec:cr}: Confidence region.}\par
Following the spirit of the previous section, we make use of the conditional CLTs presented in Section~\ref{sec:CLT} to show how one can get confidence region using SIGLE.
\item {\bf Section~\ref{apdx:advanced_discussion}: Side notes about SIGLE.}\par
In this section, we put in the limelight more advanced questions related to the methods proposed in this paper. We start by proposing a reinterpretation of the methods presented in this paper when we consider that the model is misspecified in the sense that the observations~$y_i$'s have not been initially generated from the GLM presented in Section~\ref{sec:1.1}. In a second and last part, we focus on the diffeomorphism $\Psi$ which is a key ingredient involved in SIGLE. We provide a new perspective on~$\Psi$ relying on tools from convex analysis before explaining how we compute in practice quantities of the form~$\Psi(\rho)$ that are involved in the algorithms presented in this paper.
\item {\bf Section~\ref{apdx:proofs}: Proofs.}\par
We provide all the proofs of the theoretical results presented in this paper.
\item {\bf Section~\ref{sec:conditional-signs}: Inference conditional on the signs.}\par
We start by a gentle introduction to the Leftover Fisher information. Introduced in~\cite{fithian2014optimal}, this concept allows to show that conditioning on both the selected support and the signs of the dual variable (i.e.~$E_M^{S_M}$ with the notations of Section~\ref{sec:intro}) lead in general to wider (and thus worse) confidence intervals. Our goal is to use this preliminary to discuss with more details the method proposed by~\cite{taylorGLM}. In particular, we explain that the former approach is doomed to work conditional to~$E_M^{S_M}$ since the usual trick used in the linear model to condition only on~$E_M$ does not apply for an arbitrary GLM. 
\end{itemize}

\section{Regularization bias and conditional MLE}

\label{sec:MLEvsBIAS} 
In this section, we wish to emphasize the different nature of our approach and that of~\cite{taylorGLM} which we consider as the more relevant point of comparison, to the best of our knowledge. While we rely on a conditional MLE viewpoint, the former paper consider a debiasing approach.
\begin{itemize}
\item {\it The debiasing approach}\\
$\ell_1$-penalization induced a soft-thresholding bias and one can first try to modify the solution of the penalized GLM~$\hat \vartheta^{\lambda}$ to approximate the unconditional MLE of the GLM using only the features in the selected support~$M$ by some vector~$\underline \theta$. Provided that we work with a {\it correctly specified model}~$M$--{\it i.e.,} one that contains the true support~$\{j\in[d] \, |\, \vartheta^*_j\neq0\}$--standard results ensure that the unconditional MLE is asymptotically normal, asymptotically efficient and centered at~$\vartheta^*_M$. If one can show that the selection event only involve polyhedral constraints on a linear transformation $\eta^{\top}\underline \theta$ of the debiased vector~$\underline \theta$, the conditional distribution of~ $\eta^{\top}\underline \theta$ would be a truncated Gaussian. This is the approach from~\cite{taylorGLM} that we detail in Section~\ref{sec:debiasing}.
\item {\it SIGLE : the conditional MLE viewpoint}\\
In this paper we follow a different route: one can grasp the nettle by studying directly the properties of the unpenalized conditional MLE. 
\end{itemize}

\subsection{Selective inference through debiasing}
\label{sec:debiasing}

The idea behind the method proposed by~\cite{taylorGLM} is that we need two key elements to deploy the approach from~\cite{sun16} proposed in the linear model with Gaussian errors: 
\begin{itemize}
    \item A statistic~$T(Y)$ converging in distribution to a Gaussian distribution with a mean involving the parameter of interest;
    \item A selection event that can be written as a union of polyhedra with respect to~$\eta^{\top}T(Y)$ for some vector $\eta$.
\end{itemize}

\noindent
In practice, a solution of the generalized linear Lasso (cf. Eq.\eqref{e:generalized_lasso}) can be approximated using the Iteratively Reweighted Least Squares (IRLS).
 Defining 
\begin{align*}
W(\vartheta) &= \nabla^2_{\eta}\mathcal L_N(\eta)\big|_{\eta=\mathbf X\vartheta} = \mathrm{Diag}(\sigma'(\mathbf X \vartheta)),\\
\text{and } z(\vartheta) &= \mathbf X\vartheta - [W(\vartheta)]^{-1}\nabla_{\eta}\mathcal L_N(\eta) \big|_{\eta=\mathbf X\vartheta}  = \mathbf X\vartheta +  [W(\vartheta)]^{-1}(Y-\sigma(\mathbf X \vartheta)),
\end{align*}
the IRLS algorithm works as follows.
\begin{algorithm}
\begin{algorithmic}[1]
 	\STATE Initialize~$\vartheta_c=0$.
 	\STATE Compute~$W(\vartheta_c)$ and~$z(\vartheta_c)$.
 	\STATE Update the current value of the parameters with \vspace{-0.3cm}\[\vartheta_c \leftarrow {\arg \min}_{\vartheta} \frac12(z(\vartheta_c)-\mathbf X \vartheta)^{\top} W(\vartheta_c)(z(\vartheta_c)-\mathbf X \vartheta)+\lambda \|\vartheta\|_1.\]\vspace{-0.54cm}
 	\STATE Repeat steps 2. and 3. until convergence.
\end{algorithmic}
\end{algorithm}

\noindent If the IRLS has converged, we end up with a solution~$\hat  \vartheta^{\lambda}$ of Eq.\eqref{e:generalized_lasso} and, for $M=\{j\in[d]\, |\, \hat \vartheta^{\lambda}_j \neq0\}$, the active block of stationary conditions (Eq. \eqref{def:ES} $(i)$) can be written as
\[ \mathbf X^{\top}_M W\left\{z - \mathbf X_M \hat \vartheta^{\lambda}_M\right\} = \lambda S_M,\]
where~$W=W( \hat \vartheta^{\lambda})$, ~$z=z( \hat \vartheta^{\lambda})$ and $S_M = \mathrm{sign}(\hat\theta^{\lambda}_M)$. 
The solution~$\hat \vartheta^{\lambda}_M$ should be understood as a biased version of the unpenalized MLE~$\widehat \theta$ obtained by working on the support~$M$, namely
\[\widehat \theta \in \arg \min_{\theta \in \Theta_M} \sum_{i=1}^N
\xi(\langle \mathbf X_{i,M},\theta\rangle)
-\langle y_i\mathbf X_{i,M},\theta\rangle .\] 
If we work with a {\it correctly specified model}~$M$--{\it i.e.,} one that contains the true support~$\{j\in[d] \, |\, \vartheta^*_j\neq0\}$--then it follows from standard results that the MLE~$\widehat \theta$ is a consistent and asymptotically efficient estimator of~$\vartheta^*_M$ (see e.g.~\cite[Theorem 5.39]{van2000asymptotic}). A natural idea consists in debiasing the vector of parameters~$\vartheta^{\lambda}_M$ in order to get back to the parameter~$\widehat \theta$ and to use its nice asymptotic properties for inference. We thus consider
\[\underline \theta = \vartheta^{\lambda}_M + \lambda \left( \mathbf X_M^{\top} W \mathbf X_M\right)^{-1}S_M,\]
so that~$\underline \theta$ satisfies \begin{equation}\label{eq:almost-KKT1}\mathbf X_M^{\top}W\left\{ z- \mathbf X_M \underline \theta\right\}=0.\end{equation}
If one replaces~$W$ and~$z$ in Eq.\eqref{eq:almost-KKT1} by~$W(\underline \vartheta)$ and~$z(\underline \vartheta)$ (with the obvious notation that~$\underline \vartheta_M=\underline \theta$ and~$\underline \vartheta_{-M}=0$), Eq.\eqref{eq:almost-KKT1} corresponds to the stationarity condition of the unpenalized MLE for the generalized linear regression using only the features in~$M$.

Hence, \cite{taylorGLM} propose to treat the debiased parameters~$\underline \theta$ has asymptotically normal centered at~$\vartheta^*_M$ with covariance matrix~$\left( \mathbf X_M^{\top} W(\vartheta^*) \mathbf X_M\right)^{-1}$. Since~$\vartheta^*$ is unknown, they use a plug-in estimate and replace~$W(\vartheta^*)$ by ~$W(\hat \vartheta^{\lambda})$ in the Fisher information matrix. By considering that~$\vartheta^*=N^{-1/2}\beta^*$ where each entry of $\beta^*$ is independent of~$N$, they claim that the selection event~$E_M^{S_M}$ can be asymptotically approximated by
\[\mathrm{Diag}(S_M)\left(\underline \theta - \lambda\left( \mathbf X_M^{\top} W \mathbf X_M\right)^{-1}S_M\right) \geq 0.\]
 Hence, to derive post-selection inference procedure, they apply the polyhedral lemma to the limiting distribution of $N^{1/2}\underline \theta$, with~$M$ and~$S_M$ fixed.

\subsection{Discussion}

\paragraph{Duality between SIGLE and debiasing approaches.}
Oversimplifying the situation, our approach could be understood as the dual counterpart of the one from~\cite{taylorGLM} in the sense that the former paper is first focused on getting an (unconditional) CLT and deal with the selection event in a second phase. On the contrary, we are first focused on the conditional distribution ({\it i.e.}, we want to be able to sample from the conditional distribution) while the asymptotic (conditional) distribution considerations come thereafter. Figure~\ref{fig;dual_view_PSI} provides a visualization of these two different perspectives that can be used for PSI.

\begin{figure}[!ht]
\centering
\scalebox{0.9}{\begin{tikzpicture}[thick,scale=2]

\draw[color=blue,-stealth] (-2,3.5)--node[above]{$N\to +\infty$}(-1,3.5)node[xshift=0.8cm]{\shortstack{Asymptotic\\normality}} ;
\draw[color=auburn,-stealth] (-2,3.5)--node[left]{$\cdot \; | \; E_M$}(-2,2.5)node[yshift=-0.45cm]{\shortstack{Conditioning\\on $E_M$}} ;

\draw[color=blue,-stealth] (-1,0)--node[above]{We prove a {\color{red}conditional CLT}}node[below]{\shortstack{ $\overline V_N(\theta^*)^{1/2} (\widehat \theta-\overline \theta( \theta^*)) \underset{N\to \infty}{\overset{(d)}{\longrightarrow}} \mathcal N(0,\mathrm{Id}_s)$ }}(1.6,0) ;
\draw[color=auburn,-stealth] (-1,3)--node[midway,fill=white]{\shortstack{We consider $\widehat \theta$ the\\ {\color{red}conditional MLE} on $M$:\\
$ \widehat \theta \in \arg \min_{\theta \in \mathds R^s} \mathcal L_{N}^M(\theta,Y)$}}(-1,0);

\draw[dotted] (-1,3) --node[pos=.6,below,rotate=-40,]{This paper} node[pos=.6,above,rotate=-40]{Taylor \& Tibshirani} (1.5,1) ;

\draw[color=blue,-stealth] (-1,3)--node[above]{\shortstack{{\color{red}Debiasing} the generalized linear lasso solution}} node[below]{$\underline \theta^{\lambda} \underset{N\to \infty}{\overset{(d)}{\longrightarrow}}\mathcal N(\theta^*,\mathcal I(\theta^*)^{-1})$ }(2.5,3) ;
\draw[color=auburn,-stealth] (2.5,3)--node[midway,fill=white]{\shortstack{Asymptotic description  \\of $E_M$ as ${\color{red}A \underline \theta^{\lambda} \leq b}$}}(2.5,1);
\node[rectangle,draw] (r) at (2.5,0.65) {\shortstack{Inference using\\ the polyhedral Lemma}};

\node[rectangle,draw] (r) at (2.5,0) {\shortstack{Inference using the \\ SEI-SLR algorithm}};

\end{tikzpicture}}
\caption{Duality between SIGLE and debiasing approaches.}
\label{fig;dual_view_PSI}
\end{figure}
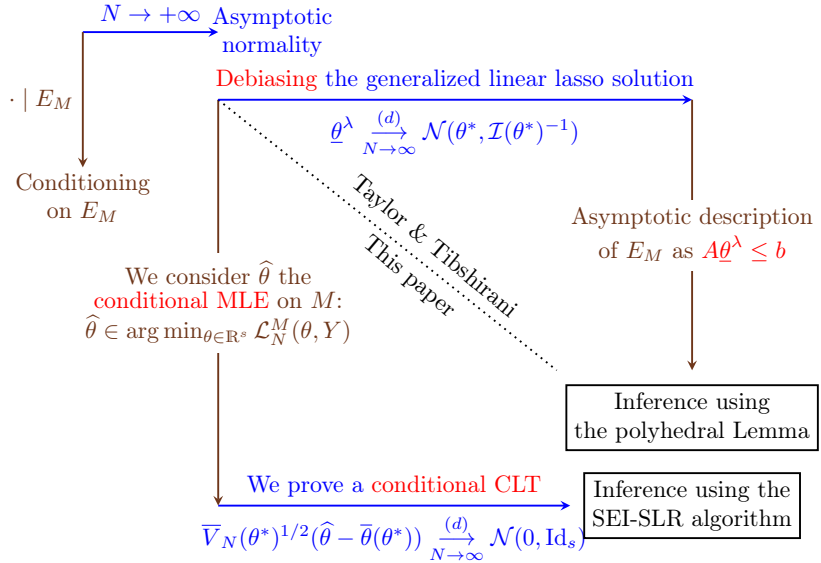




\paragraph{Comprehensive comparison between SIGLE and \cite{taylorGLM}.}

In~\cite{taylorGLM}, the authors consider only the more restrictive framework of the selected model where $\mathbf X \vartheta^* = \mathbf X_M \theta^*$ for some $\theta^* \in \mathds R^s$. Their method allows to conduct PSI inference on any linear transformation of $\theta^*$ (including in particular the local coordinates $\theta^*_j$ for $j\in [s]$), and can be efficiently used in practice. The authors do not provide a formal proof of their claim but rather motivate their approach with asymptotic arguments where they consider in particular that $\vartheta^*=N^{-1/2}\beta^*$ where each entry of $\beta^*$ is independent of $N$.

On the other hand, this paper presents {\it simple hypothesis} PSI methods in both the saturated and the selected models, in the sense that statistical inference is conducted on the vector-valued parameter of interest. Our methods are computationally more expensive than the one from~\cite{taylorGLM}, but they are proved (see Section~\ref{apdx:hypotest_CLT}) to be asymptotically valid under some set of assumptions that we discuss in details in Section~\ref{sec:discussion-CLT}. Table~\ref{table:comparison-taylor} sums up this comparison.

\renewcommand{\arraystretch}{1.3}
\begin{table}[!ht]
\centering
\hspace*{-1cm}
\begin{tabular}{|C{3.8cm}|C{4.4cm}|C{4.4cm}|}\cline{2-3}
 \multicolumn{1}{c|}{} & \cite{taylorGLM} & SIGLE (this paper)\\ \hline
Selected model & {\color{green}\cmark} & {\color{green}\cmark}  \\\hline
Saturated model & {\color{red}\xmark} & {\color{green}\cmark}  \\\hline
Hypotheses tested in the selected model & Composite: $\theta^*_j=[\theta_0^*]_j$ for some $j$ & Simple: $\theta^*=\theta^*_0$ \\\hline
Formal proof & {\color{red}\xmark} & {\color{green}\cmark}  \\\hline
Assumption on $\vartheta^*=\alpha_N^{-1}\beta^*$ with entries of $\beta^*$ independent of $N$ &For the theoretical sketches supporting their result, they consider $\alpha_N=N^{1/2}.$ & Require $\alpha_N = \omega(N^{1/2})$.
\\\hline
Low computational cost 
&{\color{green}\cmark}  &  {\color{red}\xmark} \\\hline
\end{tabular}
\caption[Comparison between SIGLE and~\cite{taylorGLM}.]{Comparison between SIGLE and~\cite{taylorGLM}.}
\label{table:comparison-taylor}
\end{table}

Note that our paper should be understood as an extension of the work from~\cite{Meir2017TractablePM} to the SLR. Indeed, the authors of the former paper propose a method to compute the conditional MLE after model selection in the linear model. 
They show empirically that the proposed confidence intervals are close to the desired level but they are not able to provide theoretical justification of their approach.

In Section~\ref{apdx:hypotest_CLT}, we show our the conditional CLTs provided in Section~\ref{sec:CLT} can be used to derive theoretical guarantees for the SIGLE procedures under the restrictive assumptions given in Theorems~\ref{prop:CLT} and~\ref{thm:MLEasymptotic}.

\section{Theoretical guarantees for SIGLE in SLR}
\label{apdx:hypotest_CLT}

In this section, we make use of the conditional CLTs presented in Section~\ref{sec:CLT} to prove that the SIGLE methods are asymptotically correctly calibrated when the assumptions of Section~\ref{sec:CLT} are satisfied. Let us stress that this section is not of practical interest for two main reasons. First, the condition under which the theoretical guarantees presented in this section hold are restrictive and correspond to the ones usually considered in the literature when analyzing the asymptotic properties of the MLE in high dimensions. Second, making use of the conditional CLTs of Section~\ref{sec:CLT} does not allow us to bypass the computational burden of sampling from the conditional distribution. Indeed, to get the SIGLE statistics, one still need to compute quantities such that $\overline G_N(\theta^*_0)$ (resp. $\overline G_N(\pi^*_0)$) or $\overline \theta(\theta^*_0)$ (resp. $\overline \pi^{\pi^*_0}$). Since the distribution of the observations conditional to the selection event has no closed form expression, we still need to use sampling methods such as the SEI-SLR algorithm presented in Section~\ref{sec:sampling} to estimate the SIGLE statistics.

\medskip

In this section, we consider the notations and assumptions described at the beginning of Section~\ref{sec:CLT} and that rely on a system of triangular arrays.

\subsection{In the selected model}
\label{sec:selective-inference-selected}


We keep the notations and the assumptions of Theorem~\ref{thm:MLEasymptotic}. Given some~$\theta^*_0 \in \mathds R^s$, we consider the hypothesis test with null and alternative hypotheses defined by
\begin{equation}
\mathds H_0 \; :\; \{\theta^*=\theta^*_0\}\quad \text{and}\quad \mathds H_1 \; :\; \{\theta^* \neq \theta^*_0\}\,.\end{equation}
The CLT from Theorem~\ref{thm:MLEasymptotic} naturally leads us to introduce the ellipsoid $W_N$ given by

\hspace*{-0.6cm}
\setlength\tabcolsep{1.5pt}
\begin{tabular}{rll}
\multirow{2}{*}{$W_N:=\Bigg\{Y \in \{0,1\}^N \; \Bigg| \;$} & $\diamond \, \, \mathbf X_M^{\top} Y \in \mathrm{Im}(\Xi)$&\multirow{2}{2mm}{$\Bigg\}$,}\\
&$\diamond \, \, \left\|[\overline G_N(\theta^*_0)]^{-1/2}H_N(\overline \theta(\theta_0^*))\left( \Psi(\mathbf X_M^{\top}Y)-\overline \theta(\theta^*_0)\right)\right\|^2_2 > \chi^2_{s,1-\alpha}$  &
\end{tabular}

\medskip

\noindent
where~$\chi^2_{s,1-\alpha}$ is the quantile of order~$1-\alpha$ of the~$\chi^2$ distribution with~$s$ degrees of freedom. If~$\overline \pi^{\theta_0^*}$ was known, we could compute~$\overline \theta(\theta_0^*)$ (using Eq.\eqref{eq:gradbar}) and thus~$\overline G_N(\theta_0^*)$. Then the test with rejection region~$W_N$ would be asymptotically of level~$\alpha$ since Theorem~\ref{thm:MLEasymptotic} gives that
\[ \overline{\mathds P}_{\theta^*_0}\left( Y \in W_N\right) \underset{N\to +\infty}{\longrightarrow} \alpha.\]

Based on this result, we construct an asymptotically valid simple hypothesis testing procedure for the test~\eqref{eq:test-selected}. Our method consists in finding an estimate of the parameter~$\overline \pi^{\theta^*_0}$ in order to approximate the rejection region~$W_N$ with a Monte-Carlo approach. From Proposition~\ref{prop:uniform-distribution}, we know that under an appropriate cooling scheme, the asymptotic distribution of the states visited by our SEI-SLR algorithm (cf. Algorithm~\ref{algo1}) is the uniform distribution on the selection event. We deduce that under the null, we are able to estimate~$\overline \pi^{\theta^*}$ and thus~$\overline \theta$ using Eq.\eqref{eq:gradbar}. This leads to the testing procedure presented in Proposition~\ref{prop:test-selected}, whose proof is postponed to Section~\ref{proof:prop:test-selected}.

\begin{proposition}\label{prop:test-selected} We keep notations and assumptions of Theorem~\ref{thm:MLEasymptotic}. We consider two independent sequences of vectors~$(Y^{(t)})_{t\geq1}$ and~$(Z^{(t)})_{t\geq1}$ generated by Algorithm~\ref{algo1}. Let us denote
\[ \widetilde \pi^{\theta^*_0} =  \frac{\sum_{t=1}^T \mathds P_{\theta^*_0}(Y^{(t)}) Y^{(t)}}{\sum_{t=1}^T \mathds P_{\theta^*_0}(Y^{(t)})} ,\quad \widetilde \theta = \Psi(\mathbf X_M^{\top} \widetilde \pi^{\theta^*_0}),\quad \widetilde G_N = \mathbf X_M^{\top} \mathrm{Diag}\left( \widetilde \pi^{\theta^*_0}\odot (1-\widetilde \pi^{\theta^*_0})\right) \mathbf X_M,\]
\begin{tabular}{rll}
\multirow{2}{*}{and $\widetilde W_N:=\Bigg\{Y \in \{0,1\}^N \; \Bigg| \;$} & $\diamond \, \, \mathbf X_M^{\top} Y \in \mathrm{Im}(\Xi)$&\multirow{2}{*}{$\Bigg\}$.}\\
&$\diamond \, \, \left\|\widetilde G_N^{-1/2}H_N(\widetilde \theta)\left( \Psi(\mathbf X_M^{\top}Y)-\widetilde \theta\right)\right\|^2_2 > \chi^2_{s,1-\alpha}$  &
\end{tabular}

\noindent Then the SIGLE procedure consisting in rejecting the null hypothesis $\mathds H_0$ when
\[   \zeta_{N,T}:= \frac{\sum_{t=1}^T \mathds P_{\theta^*_0}(Z^{(t)}) \mathds 1_{Z^{(t)}\in  \widetilde W_N}}{\sum_{t=1}^T \mathds P_{\theta^*_0}(Z^{(t)})} > \alpha,\]
has an asymptotic level lower than~$\alpha$ in the sense that for any~$\epsilon>0$, there exists~$N_0 \in \mathds N$ such that for any $N\geq N_0$ it holds,
\[ \mathds P\big( \bigcup_{T_N\in \mathds N} \bigcap_{T\geq T_N} \{\zeta_{N,T}\leq \alpha +\epsilon\}\big)=1.\]
\end{proposition}

\subsection{In the saturated model}
\label{sec:selective-inference-saturated}

We keep the notations and the assumptions of Theorem~\ref{prop:CLT}. Given some~$\pi^*_0 \in \mathds R^N$, we consider the hypothesis test with null and alternative hypotheses defined by
\begin{equation}
\mathds H_0 \; :\; \{\pi^*=\pi^*_0\}\quad \text{and}\quad \mathds H_1 \; :\; \{\pi^* \neq \pi^*_0\}\,.
\end{equation}
The CLT from Theorem~\ref{prop:CLT} naturally leads us to introduce the ellipsoid $W_N$ given by \[W_N=\left\{ Y \in \{0,1\}^N \; | \; \left\|[\overline G_N(\pi^*_0)]^{-1/2} \mathbf X_M^{\top}\left(Y-\overline {\pi}^{\pi^*_0}\right)\right\|^2_2 \geq \chi^2_{s,1-\alpha} \right\},\]
where~$\chi^2_{s,1-\alpha}$ is the quantile of order~$1-\alpha$ of the~$\chi^2$ distribution with~$s$ degrees of freedom. If~$\overline \pi^{\pi^*_0}$ was known, we could compute~$\overline G_N(\pi_0^*)$. Then the test with rejection region~$W_N$ would be asymptotically of level~$\alpha$ since Theorem~\ref{prop:CLT} gives that
\[ \overline{\mathds P}_{\pi^*_0}\left( Y \in W_N\right) \underset{N\to +\infty}{\longrightarrow} \alpha.\]

Based on this result, we construct an asymptotically valid simple hypothesis testing procedure for the test~\eqref{eq:test}. Our method consists in finding an estimate of the parameter~$\overline \pi^{\pi^*_0}$ in order to approximate the rejection region~$W_N$ with a Monte-Carlo approach. From Proposition~\ref{prop:uniform-distribution}, we know that under an appropriate cooling scheme, the asymptotic distribution of the states visited by the SEI-SLR algorithm (cf. Algorithm~\ref{algo1}) is the uniform distribution on the selection event. We deduce that under the null, we are able to estimate~$\overline \pi^{\pi^*_0}$ and thus $\overline G_N(\pi^*_0)$. This leads to the testing procedure presented in Proposition~\ref{prop:test}, whose proof is strictly analogous to the one of Proposition~\ref{prop:test-selected}.

\begin{proposition}\label{prop:test} We keep notations and assumptions of Theorem~\ref{prop:CLT}. We consider two independent sequences of vectors~$(Y^{(t)})_{t\geq1}$ and~$(Z^{(t)})_{t\geq1}$ generated by Algorithm~\ref{algo1}. Let us denote
\[ \widetilde \pi^{\pi^*_0} =  \frac{\sum_{t=1}^T \mathds P_{\pi^*_0}(Y^{(t)}) Y^{(t)}}{\sum_{t=1}^T \mathds P_{\pi^*_0}(Y^{(t)})} ,\quad \quad \widetilde G_N = \mathbf X_M^{\top} \mathrm{Diag}\left( \widetilde \pi^{\pi^*_0}\odot (1-\widetilde \pi^{\pi^*_0})\right) \mathbf X_M,\]
and $\widetilde W_N:= \left\{  Y \in \{0,1\}^N \; | \; \left\|\widetilde G_N^{-1/2}\mathbf X_M^{\top}\left( Y- \widetilde \pi^{\pi_0^*}\right)\right\|^2_2 > \chi^2_{s,1-\alpha} \right\}.$
Then the SIGLE procedure consisting of rejecting the null hypothesis $\mathds H_0$ when
\[   \zeta_{N,T}:= \frac{\sum_{t=1}^T \mathds P_{\pi^*_0}(Z^{(t)}) \mathds 1_{Z^{(t)}\in  \widetilde W_N}}{\sum_{t=1}^T \mathds P_{\pi^*_0}(Z^{(t)})} > \alpha,\]
has an asymptotic level lower than~$\alpha$ in the sense that for any~$\epsilon>0$, there exists~$N_0 \in \mathds N$ such that for any $N\geq N_0$ it holds,
\[ \mathds P\big( \bigcup_{T_N\in \mathds N} \bigcap_{T\geq T_N} \{\zeta_{N,T}\leq \alpha +\epsilon\}\big)=1.\]
\end{proposition}

\subsection{Calibration of SIGLE}

In the main paper, we have clearly stated that SIGLE procedures are calibrated by sampling under the null using the SEI-SLR algorithm or the rejection sampling method. In this section, we conduct some experiments to study the distribution under the null of the p-values of the SIGLE methods when we calibrate the tests by using the theoretical quantile given by Proposition~\ref{prop:test-selected} and Proposition~\ref{prop:test}.

\medskip 
{\bf A correct calibration under weak dependence.}
Our experiments have shown that calibrating the SIGLE procedures using our conditional CLTs from Section~\ref{sec:CLT} can lead to anti-conservative tests. This undesirable property was still observed when we conducted experiments with large values for $N$ (typically $N=30,000$). Based on our extensive simulations, we strongly believe that our conditional CLTs hold when the entries of the response vector $Y\sim \overline {\mathds P}_{\theta^*}$ are weakly dependent. To illustrate our conclusions, we conducted simulations with different regularization parameters $\lambda$ using the Setting 1 (cf. Table~\ref{table:settings}). In the first situation, a small regularization parameter is chosen (namely $\lambda=0.1$), leading to select $9$ out of the $10$ features. In the second situation, we choose $\lambda = 1$ leading to a set of active variables of size $8$. Figure~\ref{fig:cali-sigle} shows that for $\lambda=0.1$, SIGLE in the saturated model is correctly calibrated while SIGLE in the selected model is anti-conservative. When $\lambda$ is increased to $0.5$, we see that the SIGLE procedure in both the selected and the saturated model is anti-conservative.
\medskip

{\bf Despite the use the conditional CLTs for calibration, one still needs to sample under the null.}
Let us point out that calibrating the SIGLE procedure using the conditional CLTs from Section~\ref{sec:CLT} does not exempt us from sampling states using the rejection method or the SEI-SLR algorithm since we need to estimate $\overline G_N(\pi^* _0)$ (and $\overline \theta(\theta^*_0) $ in the selected model).

\begin{figure}[ht!]
\begin{center}
 \centering
    \begin{subfigure}[b]{0.49\textwidth}
        \centering
        \includegraphics[width=\textwidth]{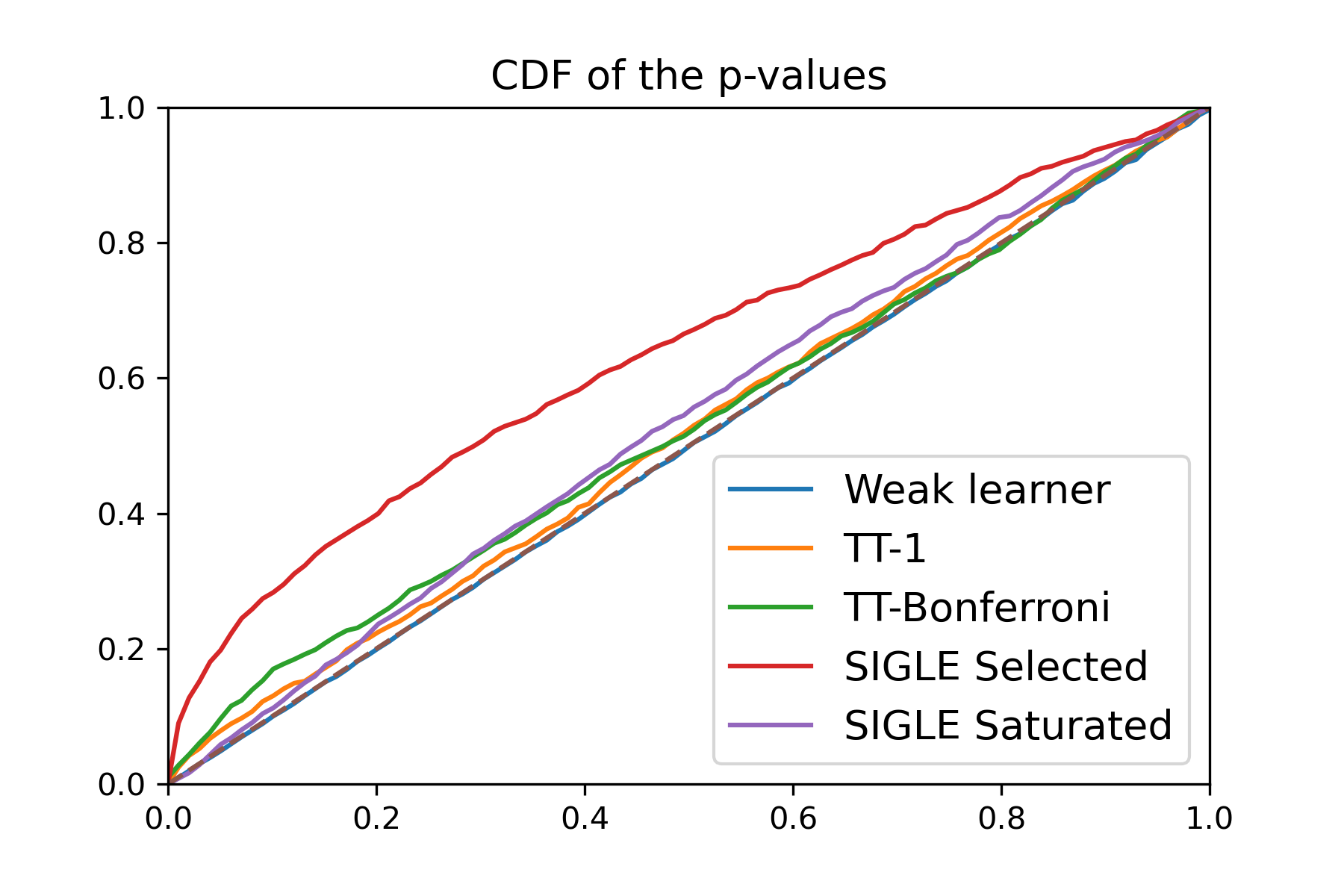}
        \caption[]%
        {{$\lambda=0.1$.}}    
    \end{subfigure}
    \hfill
     \centering
    \begin{subfigure}[b]{0.49\textwidth}
        \centering
        \includegraphics[width=\textwidth]{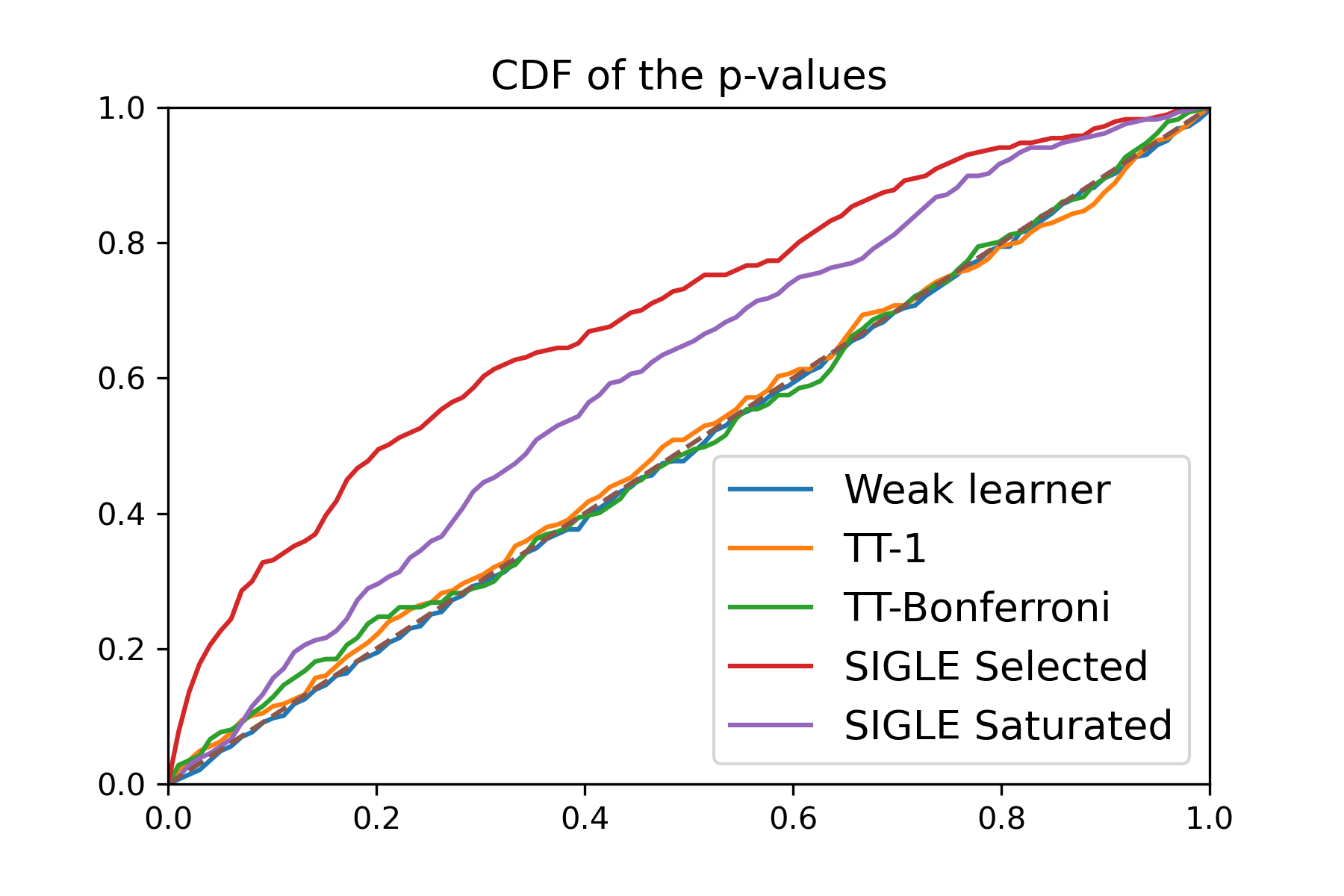}        \caption[]%
        {{\small $\lambda=0.5$.}}    
    \end{subfigure}
\caption{CDF of the p-values of the different testing procedures under the global null with the {\bf Setting 1} (cf. Table~\ref{table:settings}) for different regularization parameters $\lambda$.}
\label{fig:cali-sigle}\end{center}
\vskip -0.2in
\end{figure}

\section{Confidence region}
\label{apdxsec:cr}

\subsection{Asymptotic confidence region in the selected model}
\label{sec:CI-selected}

\subsubsection{Main result}

In the previous section, we proved that the MLE~$\widehat \theta$ satisfies a CLT with a centering vector that is not the parameter of interest~$\theta^*$. Two questions arises at this point. 
\begin{enumerate}
 \item How can we compute a relevant estimate for~$\theta^*$? \item Can we provide theoretical guarantees regarding this estimate?
\end{enumerate}
Proposition~\ref{thm:CI} answers both questions. It provides a valid confidence region with asymptotic level~$1-\alpha$ for any estimate~$\theta^{\bigstar}$ of~$\theta^*$ where the width of the confidence region is asymptotically driven by~$\|\overline \theta(\theta^{\bigstar})-\widehat \theta\|_2$. The proof of Proposition~\ref{thm:CI} can be found in Section~\ref{proof:thm:CI}.

\begin{proposition}
We keep notations and assumptions of Theorem~\ref{thm:MLEasymptotic} and we assume further that there exist $p\in[1,\infty]$ and $\kappa,R>0$ such that 
\[\theta^*\in \mathds B_p(0,R) \quad \text{ and } \quad\forall \theta \in \mathds B_p(0,R), \quad \lambda_{\min}(\overline  \Gamma^{\theta})\geq \kappa,\]
where~$\mathds B_p(0,R):=\{ \theta \in \mathds R^s \; |\; \|\theta\|_p\leq R\}$. Let us consider any estimator~$ \theta^{\bigstar} \in \mathds B_p(0,R)$ of~$\theta^*$. Then the probability of the event 
\[\| \theta^*-\theta^{\bigstar} \|_2\leq C\left(\kappa c \right)^{-1}\left\{  \|\overline \theta(\theta^{\bigstar} )-\widehat \theta\|_2+\|(\sigma^{\overline \theta})^{-2}\|_{\infty}\left(Nc^2/C\right)^{-1/2}\sqrt{\chi^2_{s,1-\alpha}}\right\},\]
tends to $1-\alpha$ as $N\to \infty$. We recall that $(\sigma^{\overline \theta})^2=\sigma'(\mathbf X_M \overline \theta(\theta^*))$.
\label{thm:CI}
\end{proposition}

\paragraph{Remarks.}  In Proposition~\ref{thm:CI}, note that the constants~$c$ and~$C$ can be easily computed from the design matrix. Nevertheless, we point out that the confidence region from Proposition~\ref{thm:CI} involves two constants (namely~$\kappa$ and~$\sigma^{\overline \theta}$) that cannot be {\it a priori} easily computed in practice. \\
Proposition~\ref{thm:CI} proves that when~$N$ is large enough, the size of our confidence region is driven by the distance~$\|\overline \theta(\theta^{\bigstar} )-\widehat \theta\|_2$. This remark motivates us to choose~$\theta^{\bigstar}$ among the minimizers of the function \[m:\theta \mapsto \|\overline \theta (\theta) - \widehat \theta\|_2^2 . \]
In the sake of minimizing~$m$, a large set of methods are at our disposal. In the next section, we propose a deep learning and a gradient descent approach for our numerical experiments.

\subsubsection{Simulations}

\paragraph{Deep learning method}

We train a feed forward neural network with ReLu activation function and three hidden layers. With this network, we aim at estimating any~$\theta \in \mathds R^s$ by feeding as input~$\overline \theta(\theta)$. We generate our training dataset by first sampling~$n_{train}=500$ random vectors~$\theta_i \sim \mathcal N(0,\mathrm{Id}_s)$, $i\in [n_{train}]$. Then, for any $i\in [n_{train}]$ we compute the estimate~$\widetilde \theta(\theta_i)$ of~$\overline \theta(\theta_i)$ as follows
\[ \widetilde \pi^{\theta_i} =  \frac{\sum_{t=1}^T \mathds P_{\theta_i}(Y^{(t)}) Y^{(t)}}{\sum_{t=1}^T \mathds P_{\theta_i}(Y^{(t)})} \quad \text{and} \quad \widetilde \theta(\theta_i) = \Psi(\mathbf X_M^{\top} \widetilde \pi^{\theta_i}),\]where $(Y^{(t)})_{t\geq 1}$ is the sequence generated from the SEI-SLR algorithm (see Algorithm \ref{algo1}). We train our network using stochastic gradient descent with learning rate~$0.01$ and~$500$ epochs. At each epoch, we feed to the network the inputs~$(\widetilde \theta(\theta_i))_{i \in [n_{train}]}$ with the corresponding target values~$(\theta_i)_{i \in [n_{train}]}$. We then compute our estimate~$\theta^{\bigstar}$ of~$\theta^*$ by taking the output of our network when taking as input the unpenalized MLE~$\widehat \theta$ using the design $\mathbf X_M$ (cf. Eq.\eqref{eq:uncondi-MLE-XM}). Figure~\ref{fig:NN} illustrates the result obtained from this deep learning approach. We keep the experiment settings of Section~\ref{ref:experiments} namely, we consider~$\vartheta^*=(1 \; 1 \; 0 \; \dots \; 0)^{\top} \in \mathds R^d$ and we choose the regularization parameter~$\lambda$ so that the selected model corresponds to the true set of active variables, namely $M=\{1,2\}$.

\begin{figure}[!ht]
\centering
\includegraphics[scale=0.4]{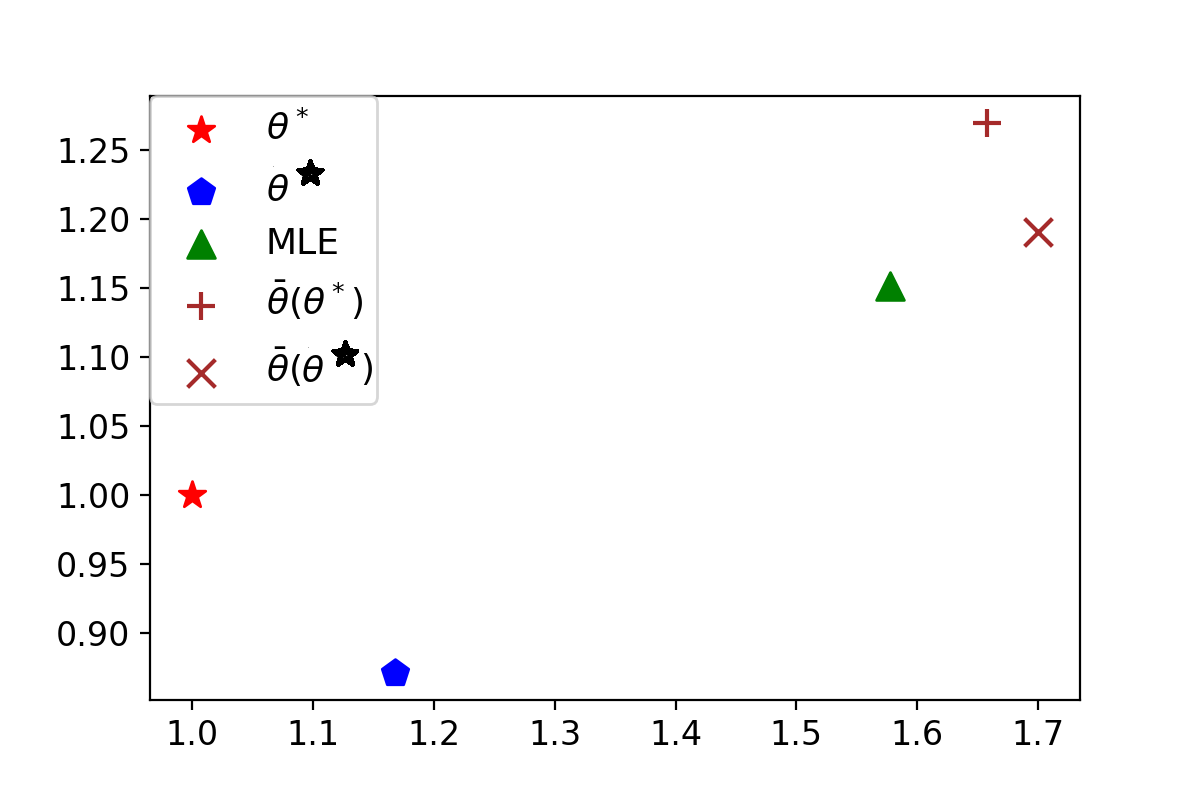}
\caption{Visualization of the results obtained using our deep learning approach to compute an estimate~$\theta^{\bigstar}$ (the blue hexagone) of~$\theta^*$ (the red star). $\theta^{\bigstar}$ corresponds to the output of the neural network when feeding as input the MLE~$\widehat \theta$ (the green triangle). We also plot the parameter~$\overline \theta(\theta^*)$ (the brown plus) and~$\overline \theta(\theta^{\bigstar})$ (the brown cross). }
\label{fig:NN}
\end{figure}

\paragraph{Gradient descent method}

As shown in the proof of the expression of Proposition~\ref{thm:CI} (cf. Eq.\eqref{eq:nable-pibar}), it holds \[\forall \theta \in \mathds R^s, \quad \nabla_{\theta} \overline \pi^{\theta}= \overline \Gamma^{\theta} \mathbf X_{M}.\] Recalling additionally that~$\overline \theta(\theta) = \Psi\left(\mathbf X_M^{\top} \overline \pi^{\theta}\right)$ (cf. Eq.\eqref{eq:gradbar}), we get that for any~$\theta \in \mathds R^s$,
\begin{align*}\nabla_{\theta} m(\theta) &=2\nabla_{\theta} \overline \theta (\theta) (\overline \theta (\theta) - \widehat \theta)\\
&=2\nabla \Psi(\mathbf X_M^{\top} \overline \pi^{\theta}) \mathbf X_M^{\top} \overline \Gamma ^{\theta} \mathbf X_M (\overline \theta (\theta) - \widehat \theta)\\
&=2\nabla \Psi(\mathbf X_M^{\top}  \pi^{\overline \theta(\theta)}) \mathbf X_M^{\top} \overline \Gamma ^{\theta} \mathbf X_M (\overline \theta (\theta) - \widehat \theta)\\
&=2\left( \mathbf X_M^{\top} \mathrm{Diag}( \pi^{\overline \theta(\theta)} \odot (1- \pi^{\overline\theta(\theta)})) \mathbf X_M\right)^{-1}\mathbf X_M^{\top} \overline \Gamma ^{\theta} \mathbf X_M (\overline \theta (\theta) - \widehat \theta).
\end{align*}
Hence,
\[ \nabla_{\theta} m(\theta)=2\left[ H_N(\overline \theta(\theta)) \right]^{-1}\mathbf X_M^{\top}\overline \Gamma ^{\theta} \mathbf X_M (\overline \theta (\theta) - \widehat \theta).\]
Given some $\theta$, $\overline \pi^{\theta}$ and $\overline \Gamma^{\theta}$ can be estimated using samples generated by the SEI-SLR algorithm (and thus the same holds for~$\overline \theta (\theta) = \Psi(\mathbf X_M^{\top} \overline \pi^{\theta})$ and for~$H_N(\overline \theta(\theta))$).

\begin{figure}[!ht]
\centering
\includegraphics[scale=0.36]{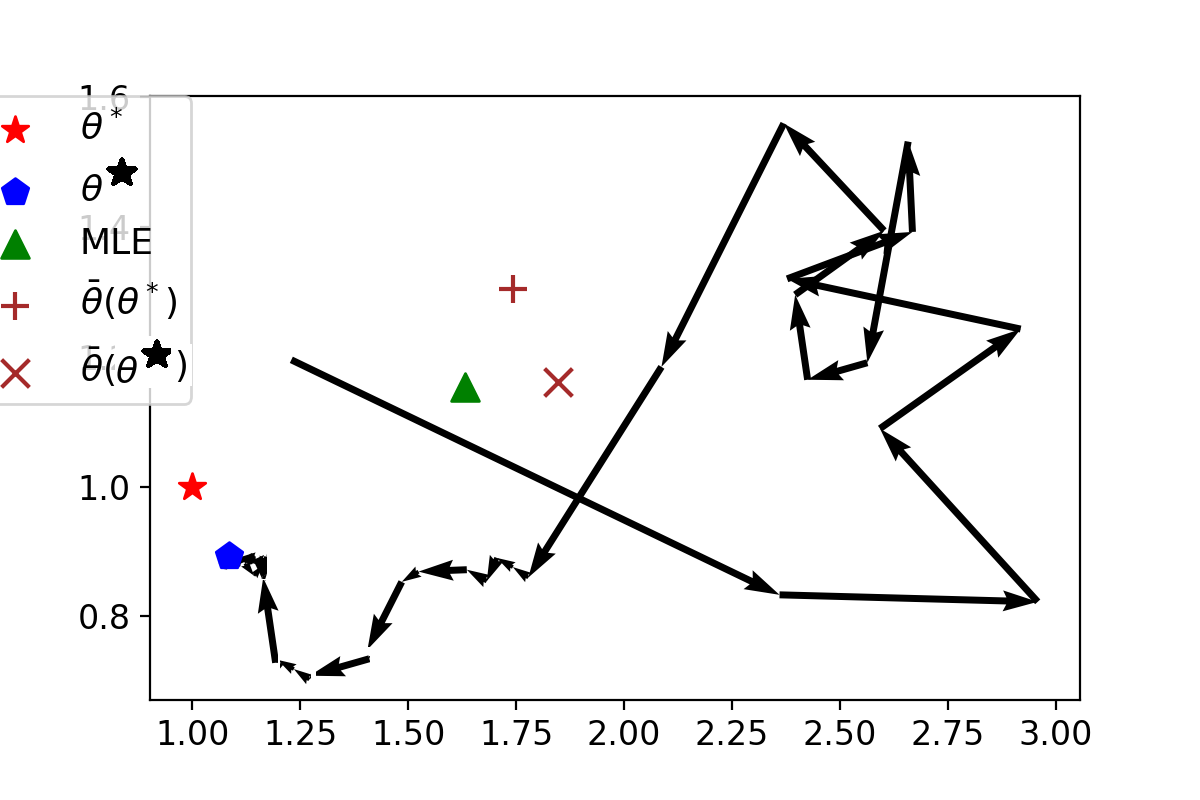}
\caption{Visualization of our gradient descent procedure to compute an estimate~$\theta^{\bigstar}$ (the blue hexagone) of~$\theta^*$ (the red star). The MLE~$\widehat \theta$ is the green triangle. We also plot the parameter~$\overline \theta(\theta^*)$ (the brown plus) and~$\overline \theta(\theta^{\bigstar})$ (the brown cross). }
\label{fig:gradstocha}
\end{figure}

\subsection{Asymptotic confidence region in the saturated model}

With Theorem~\ref{prop:CLT}, we proved that~$\mathbf X_M^{\top}Y$ with $Y$ distributed according to $\overline{\mathds P}_{\pi^*}$ satisfies a CLT with an asymptotic Gaussian distribution centered at $\mathbf X_M^{\top} \overline \pi^{\pi^*}$. Using an approach analogous to Section~\ref{sec:CI-selected}, we propose here to build an asymptotic confidence region for $\pi^*$. The proof of Proposition~\ref{thm:CI-saturated} is postponed to Section~\ref{sec:proof-CI-saturated}.

\begin{proposition}
We keep notations and assumptions of Theorem~\ref{prop:CLT} and we consider $\alpha \in (0,1)$. We assume further that there exist $p\in[1,\infty]$ and $\kappa,R>0$ such that 
\begin{align*}\pi^* &\in \mathds B_{p}\big(\frac{\mathbf 1_N}{2},R\big) \quad \text{ and } \quad\forall \pi \in  \mathds B_{p}\big(\frac{\mathbf 1_N}{2},R\big), \quad \lambda_{\min}(\overline  \Gamma^{\pi})\geq \kappa.\end{align*}
Let us consider any estimator~$ \pi^{\bigstar} \in  \mathds B_{p}(\frac{\mathbf 1_N}{2},R)$ of~$\pi^*$. Then the probability of the event
\[\| \pi^*- \pi^{\bigstar}\|_2\leq  (4\kappa)^{-1} \big\{\|\mathrm{Proj}_{\mathbf X_M}(Y-\overline \pi^{\pi^{\bigstar}})\|_2+Cc^{-1}\sqrt{\chi^2_{s,1-\alpha}}+\|\mathrm{Proj}^{\perp}_{\mathbf X_M}(\overline \pi^{\pi^*}-\overline \pi^{\pi^{\bigstar}})\|_2\big\},\]
tends to $1-\alpha$ as $N\to \infty$.
\label{thm:CI-saturated}
\end{proposition}

\paragraph{Remarks.}  
\begin{itemize}
\item Analogously to Section~\ref{sec:CI-selected}, Proposition~\ref{thm:CI-saturated} motivates us to choose~$\pi^{\bigstar}$ among the minimizers of the function \[M:\pi \mapsto \|\mathbf X_M^{\top}\overline \pi^{\pi}- \mathbf X_M^{\top} Y\|_2^2 . \]
As mentioned in the Section~\ref{sec:CI-selected}, one can rely for example on a deep learning or a gradient descent method in order to reach a local minimum $\pi^{\bigstar}$ for $M$.
\item The term $\|\mathrm{Proj}^{\perp}_{\mathbf X_M}(\overline \pi^{\pi^*}-\overline \pi^{\pi^{\bigstar}})\|_2$ arising in the confidence region from Proposition~\ref{thm:CI-saturated} illustrates that our conditional CLT from Theorem~\ref{prop:CLT} holds on $\mathbf X_M^{\top} Y$ and that we do not control what occurs in the orthogonal complement of the span of the columns of $\mathbf X_M$. Nevertheless, let us comment informally our result in the case where $E_M=\{0,1\}^N$ (meaning that there is no conditioning) and where $\vartheta^*$ is close to $0$ (meaning that $\pi^*$ is close to $\mathbf 1_N /2$). In this framework, $\overline {\Gamma}^{\pi} = \mathrm{Diag}(\pi \odot (1-\pi))$ is close to~$\frac14 \mathrm{Id}_N$ for $\pi$ in a small neighbourhood around $\mathbf 1_N /2$. Hence, we get that $\kappa$ is approximately~$\frac14$. Since it also holds that $\overline \pi^{\pi^*}-\overline \pi^{\pi^{\bigstar}}=\pi^*-\pi^{\bigstar}$ (since $E_M= \{0,1\}^N$),  we obtain from Proposition~\ref{thm:CI-saturated} that a CR for $\mathrm{Proj}_{\mathbf X_M} \pi^*$ with asymptotic coverage $1-\alpha$ is
\[\|\mathrm{Proj}_{\mathbf X_M} (\pi^*-\pi^{\bigstar})\|_2 \leq \|\mathrm{Proj}_{\mathbf X_M}(Y-\overline \pi^{\pi^{\bigstar}})\|_2+Cc^{-1}\sqrt{\chi^2_{s,1-\alpha}}.\]
\end{itemize}

\section{Side notes about SIGLE}
\label{apdx:advanced_discussion}
\subsection{SIGLE for a misspecified model from the start}
\label{apdx:advanced_discussion_misspe}

In this paper, we have considered the case where the observed data~$y_i \in \mathcal Y$ has indeed by generated from the GLM presented in Section~\ref{sec:1.1}. Can we extend the methods presented in this paper when we remove this assumption?

In this section, we consider that the~$y_i$'s are i.i.d. and distributed according to an arbitrary probability distribution~$\mathds P$. 

\subsubsection{SIGLE in the selected model}

In the case of a misspecified model from the start, the assumption made to be in the selected model is 
\[\sigma^{-1}(\mathds E[Y]) \in \mathrm{Im}(\mathbf X_M),\]
where the expectation is taken with respect to~$\mathds P$. 
We define 
\begin{equation}\label{apdx:misspe-theta*}\theta ^* \in \arg \min _{\theta \in \mathds R^s} \overline{\mathds E}\big[- \log \overline{ \mathds P}_{\theta}(Y) \big]. \end{equation}
$\mathds P_{\theta^*}$ can be understood as the probability distribution belonging to the GLM family with design matrix $\mathbf X_M$ leading to the conditional distribution $\overline{\mathds P}_{\theta^*}$ that is the closest possible to $\overline {\mathds P}$. More precisely, for any GLM distribution $\mathds P_{\theta}$, $\theta \in \mathds R^s$, we have
\[\mathrm {KL}(\overline {\mathds P} \mid \overline {\mathds P}_{\theta}) \geq \mathrm {KL}(\overline {\mathds P} \mid \overline {\mathds P}_{\theta^*}) .\] 

In the following, we reinterpret the methods of this paper relaxing the assumption that the model is well-specified from the start, as 
%
%
%
%
it might happen that the true initial distribution of the observation~$\mathds P$ does not belong to the GLM family. More precisely, considering the null hypothesis:
\[\mathds H_0: \quad \{ \overline {\mathds P}\equiv \overline{\mathds P}_{\theta^*_0}\}\,,\]
we can fall into one of the following cases:
\begin{enumerate}
\item If the model was well-specified initially, this means that there exists some~$\vartheta^* \in \mathds R^d$ such that~$\mathds P = \mathds P_{\vartheta^*}$ and thus~$\overline {\mathds P} \equiv \overline {\mathds P}_{\vartheta^*_M}$ (in the selected model). Namely, the null hypothesis is true for at least one parameter vector~$\theta^*_0 \in \mathds R^s$.
\item If the model was not well-specified initially but the null is true for some~$\theta^*_0 \in \mathds R^s$, this means that by conditioning on the selection event, we lost the information regarding the fact that the model was misspecified initially. 
\item If the model was not well-specified initially and the null is false for any~$\theta^*_0$, this means that~$\overline{\mathds P}$ still carries the information of the initial model misspecificity.
\end{enumerate}

\paragraph{A predictive viewpoint on SIGLE in the selective model.}

To obtain the SIGLE statistic in the selected model, we need to compute~$\overline \theta(\theta^*_0)$ which is defined by
\[\overline \theta(\theta_0^*) \in \arg \min_{\theta \in \mathds R^s} \overline{\mathds E}_{\theta^*_0} \big[-\log \mathds P_{\theta}(Y)\big] =\arg \min_{\theta \in \mathds R^s} \mathrm{KL}(\overline{\mathds P}_{\theta^*_0} \mid \mathds P_{\theta}) .\]
The question that we ask is how far is the distribution~$\mathds P_{\overline \theta(\theta^*_0)}$ from~$\overline {\mathds P}$. This can be of interest for a prediction task where one might want to use~$\overline \theta(\theta^*_0)$ to predict the response to new entries.

The best approximation of~$\overline {\mathds P}$ that we can get considering an unconditional GLM distribution of the form $\mathds P_{\theta}$ is~$\mathds P_{\vec \theta}$ where
\[\vec \theta \in \arg \min_{\theta \in \mathds R^s} \overline{\mathds E} \big[-\log \mathds P_{\theta}(Y)\big]= \arg \min_{\theta \in \mathds R^s} \mathrm{KL} (\overline{\mathds P} \mid  \mathds P_{\theta}).\]
Therefore, we want to compare the difference between the KL divergence between~$\overline {\mathds P}$ and~$\mathds P_{\vec \theta}$, and the KL divergence between~$\overline {\mathds P}$ and~$\mathds P_{\overline  \theta(\theta_0^*)}$. It holds

\begin{equation*}
\mathrm{KL}(\overline{\mathds P} \mid \mathds P_{\overline \theta(\theta^*_0)})=  \mathrm{KL}(\overline{\mathds P} \mid \mathds P_{\vec \theta} ) + \overline {\mathds E} \big[\log \frac{\mathds P_{\vec \theta}}{\mathds P_{\overline \theta( \theta^*_0)}}\big],
\end{equation*}
where~$\overline {\mathds E} \big[\log \frac{\mathds P_{\vec \theta}}{\mathds P_{\overline \theta( \theta^*_0)}}\big]\geq 0$ by definition of~$\vec \theta$. Therefore, $\overline {\mathds E} \big[\log \frac{\mathds P_{\vec \theta}}{\mathds P_{\overline \theta( \theta^*_0)}}\big]$ corresponds to the additional error we make in terms of KL divergence by working with the proxy~$\overline {\mathds P}_{\theta^*_0}$ instead of the true conditional distribution of the observations~$\overline{\mathds P}$. 

\bigskip

\tikzstyle{latent} = [circle,fill=white,draw=black,inner sep=1pt,minimum size=20pt, font=\fontsize{10}{10}\selectfont,node distance=3]
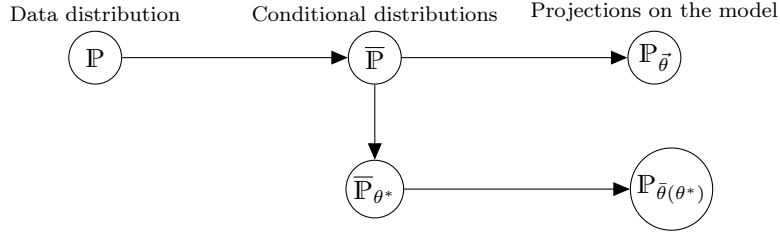
\begin{figure}
\centering
\begin{tikzpicture}
 \node[latent, label=above:{Data distribution}] (P) {$\mathds P$}; %
 \node[latent,right=of P, label=above:{Conditional distributions}] (barP) {$\overline {\mathds P}$};%
  \node[latent,below=of barP,yshift=2cm] (barPtheta) {$\overline {\mathds P}_{\theta^*}$};%
  \node[latent,right=of barP, label=above:{Projections on the model}] (vectheta) {$ {\mathds P}_{\vec \theta}$};%
  \node[latent,right=of barPtheta] (bartheta) {$ {\mathds P}_{\bar \theta( \theta^*)}$};%

 \edge[scale=3] {P} {barP}  
 \edge {barP} {barPtheta}
 \edge {barPtheta} {bartheta}
  \edge {barP} {vectheta}
\end{tikzpicture}
\caption{Visualizations of all distributions that we consider if the model is a priori not necessarily well-specified from the start.}
\end{figure}

\subsection{Inverting the first order optimality condition}
\label{apdx:inverting-PSI}

When characterizing the selection event $E_M^{S_M}$ (see Theorem~\ref{thm:ES}), we highlighted the crucial role of the diffeomorphism $\Xi=\Psi^{-1}$ arising in the first order optimality condition. In this section, we aim at presenting
\begin{itemize}
\item a
different view on $\Psi$ using tools from convex analysis,  
\item the practical methods we use to compute quantities involving $\Psi$ in our simple hypothesis testing method in the selected model.
\end{itemize}

\subsubsection{SIGLE through the lens of convex analysis}

Recalling the definition of the negative log-likelihood $\mathcal L_N(\theta, (Y,\mathbf X_M))$, we will denote in this section
\[\mathcal L_{N,0,M}(\theta) := \mathcal L_N(\theta,(0,\mathbf X_M))=\sum_{i=1}^N \xi(\langle \mathbf X_{i,M},\theta\rangle).\]
Let us recall that the Fenchel conjugate of the map $\mathcal L_{N,0,M}$ is defined by
\[\mathcal L_{N,0,M}^*:\rho\in \mathds R^s \mapsto \sup_{\theta \in \mathds R^s} \left\{ \langle \rho, \theta\rangle -\mathcal L_ {N,0,M}(\theta)\right\}.\]
Since $\xi$ is a convex and $C^{m+1}$ function, $\mathcal L_{N,0,M}$ is also a convex and a $C^{m+1}$ map which implies that $L_{N,0,M} = \nabla \mathcal L_{N,0,M}$ is a homeomorphism. We deduce that for any $\rho \in \mathds R^s,$
\begin{align*} \mathcal L_{N,0,M}^*(\rho ) &= \langle \rho, L_{N,0,M}^{-1}(\rho)\rangle - \mathcal L_{N,0,M}(L_{N,0,M}^{-1}(\rho)),\\
\nabla \mathcal L_{N,0,M}^*(\rho) &= L_{N,0,M}^{-1}(\rho).
\end{align*}
For any $Y\in \{0,1\}^N$, the unpenalized MLE $\widehat \theta$ with the design matrix $\mathbf X_M$ and the observed response $Y$ is given by (using the first order optimality condition)
\[\widehat \theta = L_{N,0,M}^{-1}(\mathbf X_M^{\top}Y).\]
We deduce that 
\[\widehat \theta = \nabla \mathcal L^*_{N,0,M}(\mathbf X_M^{\top}Y) .\]
Similarly, using Eq.\eqref{eq:gradbar} we get
\[\overline \theta(\theta^*) = \nabla \mathcal L_{N,0,M}^*(\mathbf X_M ^{\top} \overline \pi^{\theta^*}).\]
We deduce that $\Psi = \nabla \mathcal L_{N,0,M}^*$.

In order to provide a concrete interpretation of the function~$\Psi$, let us first characterize the Fenchel conjugate $\mathcal L^*_{N,0,M}$:
\begin{align}
\mathcal L^*_{N,0,M}(\rho) &= \sup_{\theta \in \mathds R^s} \left\{ \langle \rho, \theta\rangle -\mathcal L_ {N,0,M}(\theta)\right\}\notag\\
 &= \sup_{\theta \in \mathds R^s} \sum_{i=1}^N \left\{\rho_i\theta_i-\xi(\mathbf X_{i,M}\theta)\right\}\notag\\
&= (\sum_{i=1}^N f_i)^*(\rho)\notag\\
&= (f_1^* \square \cdots \square f_N^*)(\rho)\notag\\
&:= \min_{\rho=\rho^{(1)}+\dots+\rho^{(N)}} \left\{ f_1^*(\rho^{(1)}) + \dots+f_N^*(\rho^{(N)})\right\},\label{eq:infconv-rho}
\end{align}
where in the last equality we used \cite[Theorem 6.5.8]{laurent1972approximation} and where for any $i\in [N]$,
\[f_i:\theta \in \mathds R^s \mapsto \xi(\mathbf X_{i,M}\theta).\] 
Using Lemma~\ref{lemma:infconv}, we obtain that $\mathcal L^* _{N,0,M}(\rho)$ is the minimal entropy obtained among the vectors of probabilities $\pi \in (0,1)^N$ satisfying $\rho = \mathbf X_M^{\top} \pi$.

\begin{Lemma} \label{lemma:infconv}
The inf-convolution in Eq.\eqref{eq:infconv-rho} is attained for $(\rho^{(i)})_{i \in [N]} \in \left(\mathds R^s\right)^N$ such that for any $i\in [N]$, $\rho^{(i)}= \pi_i \mathbf X_{i,M}$ for some $\pi_i \in (0,1)$. Moreover,
\begin{align}
\mathcal L^*_{N,0,M}(\rho) &= \min_{\pi\in(0,1)^N \; s.t. \; \rho=\mathbf X_M^{\top} \pi} H(\pi),\label{eq:lemma:infconv}
\end{align}
where 
\[H(\pi) = \sum_{i=1}^N \{ \pi_i\ln(\pi_i) + (1-\pi_i) \ln(1-\pi_i)\}.\]
\end{Lemma}

\begin{proof}[Proof of Lemma~\ref{lemma:infconv}]$ $\newline
\begin{itemize}
\item {\bf The inf-convolution in Eq.\eqref{eq:infconv-rho} is attained.}\\
First, we know from  \cite[Theorem 6.5.8]{laurent1972approximation} that the minimum in the inf-convolution of Eq.\eqref{eq:infconv-rho} is attained. 
\item {\bf $\rho^{(i)}$ in Eq.\eqref{eq:infconv-rho} can be chosen in the span of $\mathbf X_{i,M}$.}\\Let us consider $i\in [N]$ and some $\rho^{(i)} \in \mathds R^s$. Let us assume by contradiction that $\rho^{(i)} \notin \mathrm{Span}(\mathbf X_{i,M})$. Then considering
\[\theta^{(i)}(t):= t  \big(\mathrm{Id}-\frac{1}{\|\mathbf X_{i,M}\|^2_2}\mathbf X_{i,M}^{\top} \mathbf X_{i,M}\big) \rho^{(i)},\]
we have 
\begin{align*}
&\lim_{t \to +\infty} \left\{ \langle \rho^{(i)}, \theta^{(i)}(t)\rangle -\xi(\mathbf X_{i,M}\theta^{(i)}(t))\right\} =\lim_{t \to +\infty} \left\{ t [\rho^{(i)}]^{\top}\mathrm{Proj}^{\perp}_{\mathbf X_{i,M}}\rho^{(i)}-0\right\} = +\infty,
\end{align*}
which means that $f_i^*(\rho^{(i)})=+\infty$ since for any $t>0$ it holds 
\begin{align*}
f_i^*(\rho^{(i)}) &= \sup_{\theta \in \mathds R^s} \left\{ \langle \rho^{(i)}, \theta\rangle -\xi(\mathbf X_{i,M}\theta)\right\}\\
&\geq  \left\{ \langle \rho^{(i)}, \theta^{(i)}(t)\rangle -\xi(\mathbf X_{i,M}\theta^{(i)}(t))\right\}.
\end{align*}
We deduce that in the inf-convolution of Eq.\eqref{eq:infconv-rho}, we can consider that for any $i\in [N]$, $\rho^{(i)} = \pi_i \mathbf X_{i,M}$ for some $\pi_i \in \mathds R$. 
\item {\bf $\rho^{(i)}$ in Eq.\eqref{eq:infconv-rho} can be chosen as $\pi_i \mathbf X_{i,M}$ with $\pi_i \in (0,1)$.}\\
We have already proved that $\rho^{(i)}$ in Eq.\eqref{eq:infconv-rho} can be chosen as $\rho^{(i)}=\pi_i \mathbf X_{i,M}$. It holds
\begin{align*}
f_i^*(\pi_i \mathbf X_{i,M}) &= \sup_{\theta \in \mathds R^s} \left\{ \langle \pi_i \mathbf X_{i,M}, \theta\rangle -\xi(\mathbf X_{i,M}\theta)\right\}\\
 &=\sup_{\theta \in \mathds R^s} \left\{ \langle \pi_i, \mathbf X_{i,M}\theta\rangle -\xi(\mathbf X_{i,M}\theta)\right\}\\
&= \sup_{ r \in \mathds R} \left\{ \pi_i r - \xi(r)\right\}\\
&= \xi^*(\pi_i)\\
&= H(\pi_i),
\end{align*}
where we used that the Fenchel conjugate of the softmax function is the entropy $H$ defined by 
\[
H(p)= \left\{
    \begin{array}{ll}
        p\ln(p)+(1-p)\ln(1-p) & \mbox{if } p\in (0,1), \\
        +\infty & \mbox{otherwise.}
    \end{array}
\right.\]
\end{itemize}
Since in Eq.\eqref{eq:infconv-rho} we aim a reaching a minimum, we deduce from these computations that one can restrict $\rho^{(i)}$ to be of the form $\pi_i \mathbf X_{i,M}$ with $\pi_i \in (0,1)$.
\end{proof}

\noindent {\bf Interpretation of $\Psi(\rho)$.} Lemma~\ref{lemma:infconv} shows that~$\mathcal L^*_{N,0,M}(\rho)$ is the minimum entropy of a population characterized by~$N$ binary features with the constraint that we have some information on the population given by~$\rho \in \mathds R^s$. We assume that~$\rho$ depends linearly on the proportion of the population with the different features, namely
\[\rho = \mathbf X_M^{\top}\pi,\]
where for all~$i\in[N]$, $\pi_i$ represents the proportion of people with feature~$i$. Hence, given the observation of~$s$ aggregated properties about the population (namely~$\rho$), $\mathcal L^*_{N,0,M}(\rho)$ is the entropy corresponding to the most uniform allocation of the~$N$ binary features in the population. Hence, $\Psi(\rho)=\nabla \mathcal L_{N,0,M}^*(\rho)$ quantifies how much the entropy of this ideal description of the population is changed when a small shift in the observation of the~$s$ properties occurs.

\medskip
Taking a concrete example, one can consider that the $N$ features are the following:  age between 20 and 40, age between 40 and 60, age above 60, manager, manual labourer, lives in a big city, lives in a small town, ... The vector~$\rho$ represents the number of votes obtained by~$s$ different candidates during an election. We assume that the number of votes obtained by each candidate is a linear function of the proportion of the population with the different features. We observe only the number of votes obtained by each candidate. Then $\mathcal L^*_{N,0,M}(\rho)$ represents the entropy of the population assuming that the different features are distributed as uniformly as possible in the population. $\Psi(\rho)$ measures the variation of the entropy of the population when a small change in the number of votes obtained by the different candidates is observed.

\subsubsection{Practical implementation of SIGLE in the selected model}
\label{sec:sigle-graddescent}

The PSI method in the selected model for the $\ell^1$-penalized logistic regression proposed in this paper requires the ability to compute efficiently
\begin{itemize}
\item $\Psi(\mathbf X_M^{\top} Y)$ for any $Y\in \{0,1\}^N$,
\item $\Psi(\mathbf X_M^{\top} \overline \pi^{\theta_0^*})$.
\end{itemize}  
As already mentioned in Eq.\eqref{eq:gradbar}, for any $Y\in \{0,1\}^N$, $\Psi(\mathbf X_M^{\top} Y)$ corresponds to the unpenalized MLE $\widehat \theta$ computed using the design $\mathbf X_M$ (see Eq.\eqref{eq:uncondi-MLE-XM}). As a result, we compute $\Psi(\mathbf X_M^{\top} Y)$ by simply solving the unpenalized MLE for logistic regression using standard open source libraries (such as scikit-learn in Python where we remove the $\ell^2$-regularization which is applied by default). 

Solvers computing the MLE for logistic regression require - as far as we know - the response vector to have binary entries. As a consequence, a different approach is required to compute $\Psi(\mathbf X_M^{\top}\overline \pi^{\theta_0^*})$ since $\overline \pi^{\theta_0^*} \in (0,1)^N$. 
We found our method to be extremely accurate in our numerical experiments and it works as follows. First, we compute 
\[\theta^c \in \arg \min_{\theta \in \mathds R^s} \|\mathbf X_M \theta - \sigma^{-1}( \overline \pi^{\theta_0^*})\|_2^2,\]
and we end up with two possible cases:
\begin{enumerate}
\item Either it holds \begin{equation}\label{eq:condition-bartheta}\mathbf X_M^{\top} \sigma(\mathbf X_M \theta^c) = \mathbf X_M^{\top} \overline \pi ^{\theta_0^*},\end{equation} which is equivalent to $\overline \theta(\theta^*)=\theta^c$ (see Eq.\eqref{eq:gradbar}). In this case, we output $\theta^c$. Note that this situation occurs in particular when 
\[\sigma^{-1}(\mathbf X_M^{\top} \overline \pi^{\theta_0^*}) \in \mathrm{Im}(X_M),\]
which can be understood as a conditional selected model-type assumption.
\item Or Eq.\eqref{eq:condition-bartheta} does not hold and we consider a gradient descent approach using as warm start the vector $\theta^c$ to minimize the map
\[G:\theta \mapsto \|\mathbf X_M^{\top} \sigma(\mathbf X_M \theta) - \mathbf X_M^{\top} \overline \pi^{\theta_0^*}\|^2_2.\]
Note that the gradient of $G$ at $\theta \in \mathds R^s$ is given by
\[\nabla G(\theta) = 2 \mathbf X_M^{\top} \mathrm{Diag}(\sigma'(\mathbf X_M \theta)) \mathbf X_M  \mathbf X_M^{\top} \left( \sigma(\mathbf X_M \theta)-\overline \pi^{\theta_0^*}\right),\]
and satisfies
\[\forall \theta \in \mathds R^s, \quad \|G(\theta)\|_2 \leq \frac14 \|\mathbf X_M^{\top} \mathbf X_M \|  \times \|\mathbf X_M\|_{1,2}=:L_G,\]
where $\|\mathbf X_M\|_{1,2}:= \sqrt{\sum_{i=1}^N  \|\mathbf X_{i,M}\|_1^2}.$
\end{enumerate}
Our method is summarized with Algorithm~\ref{algo:bartheta}. 

\begin{algorithm}
\begin{algorithmic}[1]
    \STATE \text{\bf Input:} $ t_{\max}, \epsilon, \ell_r$
  \STATE $\theta^c \in \arg \min_{\theta \in \mathds R^s} \|\mathbf X_M \theta - \sigma^{-1}( \overline \pi^{\theta_0^*})\|_2^2$
  \IF{$G(\theta^c)<\epsilon$}
  \RETURN{$\theta^c$}
  \ELSE{ \STATE $\theta^{(0)} \leftarrow \theta^c$
  \STATE $t\leftarrow 0$\\
  \WHILE{$t<t_{\max}$ and $G(\theta^{(t)})> \epsilon$}
  \STATE $t \leftarrow t+1$
  \STATE $\theta^{(t)} \leftarrow \theta^{(t-1)} - \frac{\ell_r}{L_G} \nabla G(\theta^{(t-1)})$
  \ENDWHILE}
  \RETURN $\theta^{(t)}$
  \ENDIF
\end{algorithmic}
\caption{Computing $\overline \theta(\theta_0^*)$}
\label{algo:bartheta}
\end{algorithm}

\section{Proofs}
\label{apdx:proofs}

\subsection{Proof of Proposition~\ref{prop:signs}}
\label{proof-prop:signs}

Let us consider~$\vartheta_1,\vartheta_2$ two vectors in~$\Theta$ achieving the minimum in~\eqref{e:generalized_lasso}. Then, denoting~$\vartheta_3= \frac12 \vartheta_1 + \frac12 \vartheta_2$ it holds
\[\frac{\mathcal L_N(\vartheta_1,Z)+\mathcal L_N(\vartheta_2,Z)}{2} + \lambda \frac{\|\vartheta_1\|_1+\|\vartheta_2\|_1}{2} \leq \mathcal L_N(\vartheta_3,Z) + \lambda \|\vartheta_3\|_1.\]
Since the triangle inequality gives $ \|\vartheta_3\|_1 \leq \frac{\|\vartheta_1\|_1+\|\vartheta_2\|_1}{2} $ and since the function~$\xi$ 
is strictly convex, it holds that~$\mathbf X\vartheta_1=\mathbf X \vartheta_2$. Indeed, otherwise we would have by strict convexity
\begin{align*}&\mathcal L_N(\vartheta_3,Z)+\lambda \|\vartheta_3\|_1\\
= \quad & \sum_{i=1}^N \left(
\xi(\langle\mathbf x_i,\vartheta_3\rangle)
-\langle y_i\mathbf x_i,\vartheta_3\rangle \right)+ \lambda \|\vartheta_3\|_1\\
\leq \quad & \sum_{i=1}^N \left(
\xi(\langle\mathbf x_i,\frac{\vartheta_1+\vartheta_2}2\rangle)
-\frac12\langle y_i\mathbf x_i,\vartheta_1\rangle-\frac12\langle y_i\mathbf x_i,\vartheta_2\rangle \right)+ \frac12\lambda \|\vartheta_1\|_1  + \frac12\lambda \|\vartheta_2\|_1\\
<\quad &\frac{\mathcal L_N(\vartheta_1,Z)+\mathcal L_N(\vartheta_2,Z)}{2} + \lambda \frac{\|\vartheta_1\|_1+\|\vartheta_2\|_1}{2} .
\end{align*}
From the KKT conditions, we deduce that for a given $Y \in \mathcal Y^N$, all solutions~$\hat \vartheta^{\lambda}$ of~\eqref{e:generalized_lasso} have the same vector of signs denoted~$\widehat S(Y)$ which is given
\[\widehat S(Y) = \frac{1}{\lambda}\mathbf X^{\top} \left( Y - \sigma(\mathbf X \hat \vartheta^{\lambda}) \right) ,\]
where~$\hat \vartheta^{\lambda}$ is any solution to~\eqref{e:generalized_lasso}.

\subsection{Proof of Proposition~\ref{prop:EMSM-deformed-polytope}}
\label{proof-prop:EMSM-deformed-polytope}

Partitioning the KKT conditions of Eq.\eqref{eq:gKKTconditions} according to the equicorrelation set $\widehat M(Y)$ leads to
\begin{align*}
\mathbf X^{\top}_{\widehat M(Y)} \left( Y -\sigma(\mathbf X_{\widehat M(Y)} \hat{\vartheta}^{\lambda}_{\widehat M(Y)})\right)&=\lambda \widehat{{S}}_{\widehat M(Y)},\\
\mathbf X^{\top}_{-{\widehat M(Y)}} \left( Y -\sigma(\mathbf X_{\widehat M(Y)} \hat{\vartheta}^{\lambda}_{\widehat M(Y)})\right)&=\lambda \widehat{{S}}_{-{\widehat M(Y)}},\\
\mathrm{sign}(\hat{\vartheta}^{\lambda}_{\widehat M(Y)})&=\widehat {S}_{\widehat M(Y)},\\
\|\widehat {S}_{-{\widehat M(Y)}}\|_{\infty}&< 1 .\end{align*}
Since the KKT conditions are necessary and sufficient for a solution, we obtain
that~$Y$ belongs to~$E_M^{S_M}$ if and only if there exists~$\theta \in \Theta_M$ satisfying
 \begin{align*}
\mathbf X^{\top}_{M} \left( Y -\sigma(\mathbf X_{M} \theta)\right)&=\lambda  S_{M} ,\\
\mathrm{sign}(\theta)&= S_{M},\\
\| \mathbf X^{\top}_{-{M}} \left( Y -\sigma(\mathbf X_{M} \theta)\right)\|_{\infty}&< \lambda .
\end{align*}

\subsection{Proof of Proposition~\ref{prop:diffeo-g}}
\label{proof-prop:diffeo-g}

Let us consider~$\theta,\theta' \in \Theta_M$ such that~$\Xi(\theta)=\Xi(\theta')$.
Then we have
\begin{align}
0 &=\mathbf X_M^{\top} \sigma(\mathbf X_M \theta) - \mathbf X_M^{\top} \sigma(\mathbf X_M \theta') \notag\\
&= \Xi(\theta)-\Xi(\theta')\notag\\
&= \int_0^1 \nabla \Xi(\theta t + (1-t) \theta')\cdot (\theta-\theta') dt\notag\\
&= \int_0^1 \mathbf X_M^{\top} \mathrm{Diag}\left[ \sigma'(\mathbf X_M\theta t + (1-t) \mathbf X_M\theta')\right]\mathbf X_M (\theta-\theta') dt\notag\\
&= \mathbf X_M^{\top} \underbrace{\left(\int_0^1 \mathrm{Diag}\left[ \sigma'(\mathbf X_M\theta t + (1-t) \mathbf X_M\theta')\right] dt\right)}_{=:D} \mathbf X_M (\theta-\theta').\label{eq:Xi-proof}
\end{align}
Note that for any~$t\in[0,1]$ and for any~$i\in[N]$, $\left\{\sigma'(\mathbf X_M\theta t + (1-t) \mathbf X_M\theta') \right\}_{i} >0$ since~$\xi''(u)=\sigma'(u) >0$ for any~$u \in \mathds R$. We deduce that~$D \in \mathds R^{N\times N}$ is a diagonal matrix with strictly positive coefficients on the diagonal.
Eq.\eqref{eq:Xi-proof} gives that $\theta - \theta' \in \mathrm{Ker}(\mathbf X_M^{\top} D \mathbf X_M)$ which implies that $(\theta - \theta')^{\top} \mathbf X_M^{\top} D \mathbf X_M(\theta - \theta' )=0$. This means that 
\[\sum_{i =1}^N D_{i,i} \left[\mathbf X_M(\theta - \theta' )\right]_{i}^2=0.\]
Since $D_{i,i}>0$ for all $i\in [N]$, we get that $\mathbf X_M(\theta-\theta')=0$, i.e. $\mathbf X_M \theta = \mathbf X_M \theta'.$ Since $\mathbf X_M$ has full column rank, this leads to $\theta=\theta'.$

Since~$\Xi$ is injective and of class~$\mathcal C^{m}$ with a differential given by~$\nabla_{\theta}\Xi(\theta) = \mathbf X_M^{\top} \mathrm{Diag(\sigma'(\mathbf X_M\theta))}\mathbf X_M$ which is invertible at any~$\theta \in \Theta_M$ under the assumptions of Proposition~\ref{prop:diffeo-g}. Hence the global inversion theorem gives Proposition~\ref{prop:diffeo-g}.

\subsection{Proof of Theorem~\ref{prop:CLT}}
\label{proof:propCLT}

For the sake of brevity, we will simply denote $\overline G_N(\pi^*)$ by $\overline G_N$. Let us further denote $\mathbf X_M^{\top}= \left[ \mathbf w_1 \;|\; \mathbf w_2 \; |\; \dots \;| \; \mathbf w_N \right]$, where $\mathbf w_i = \mathbf x_{i,M}\in \mathds R^s$.

The proof of Theorem~\ref{prop:CLT} relies on~\cite[Theorem 1]{bardet2008dependent}. In the following, we check that all the assumptions of~\cite[Theorem 1]{bardet2008dependent} are satisfied. Denoting for any $i\in[N]$, $\xi_{i,N}= \overline G_N^{-1/2} \mathbf w_i (y_i-\overline \pi^{\pi^*}_i)$, it holds \[\overline G_N^{-1/2}\mathbf X_M^{\top} (Y-\overline \pi^{\pi^*})=\sum_{i=1}^N \overline G_N^{-1/2} \mathbf w_i (y_i-\overline \pi^{\pi^*}_i)=\sum_{i=1}^N \xi_{i,N}.\]
Let us also point that $\overline {\mathds E}_{\pi^*}[\xi_{i,N}]=0$. In the following, we will simply refer to $\xi_{i,N}$ as $\xi_i$ to ease the reading of the proof.
Let us denote further\[A_N = \sum_{i=1}^N \overline{\mathds E}_{\pi^*}\left( \|\xi_{i}\|^{3}_2\right).\] 
One can notice that\begin{align*}
\overline {\mathds E}_{\pi^*} \left(\|\xi_i\|^3_2\right)
=\overline {\mathds E}_{\pi^*}\big[(y_i-\overline \pi^{\pi^*}_i)^3\big] \|\overline G_N^{-1/2}\mathbf w_i\|^3_2
\leq \left( \frac{K}{\sqrt{c}\overline \sigma_{\min}}\right)^3 N^{-3/2} s^{3/2},
\end{align*}
where we used that
\[\|\overline G^{-1/2}_N\mathbf w_i\|_2^2\leq \|\overline G^{-1/2}_N\|^2 \times \|\mathbf w_i\|_2^2 \leq \|\overline G_N^{-1}\| (sK^2)\leq (c\overline \sigma_{\min}^2 N)^{-1}(sK^2).\]
We deduce that \[A_N \leq  \left( \frac{K}{\sqrt{c}\overline \sigma_{\min}}\right)^3 N^{-1/2} s^{3/2}.\]
Hence~$A_N \underset{N\to \infty}{\to} 0$ which the first condition that needed to be checked to apply~\cite[Theorem 1]{bardet2008dependent}.

Let us now check the second condition from that \cite{bardet2008dependent} that consists in identifying the appropriate asymptotic covariance matrix.
\begin{align*}
\sum_{i=1}^N \overline {Cov}_{\pi^*}(\xi_i) &=\sum_{i=1}^N \overline {\mathds E}_{\pi^*}\left[  \overline G_N^{-1/2} \mathbf w_i \mathbf w_i^{\top} \overline G_N^{-1/2} (y_i-\overline \pi^{\pi^*}_i)^2\right]\\
&= \sum_{i=1}^N\overline G_N^{-1/2} \mathbf w_i \underbrace{\overline {\mathds E}_{\pi^*}(y_i-\overline \pi^{\pi^*}_i)^2}_{=(\overline \sigma^{\pi^*}_i)^2} \mathbf w_i^{\top} \overline G_N^{-1/2}\\
&=\overline G_N^{-1/2}   \sum_{i=1}^N \mathbf w_i (\overline \sigma^{\pi^*}_i)^2 \mathbf w_i^{\top}   \overline G_N^{-1/2}\\
&=\overline G_N^{-1/2}  \mathbf X_M^{\top} \mathrm{Diag}\big((\overline \sigma^{\pi^*})^2\big) \mathbf X_M  \overline G_N^{-1/2}\\
&= \overline G_N^{-1/2} \overline G_N \overline G_N^{-1/2}\\
&= \mathrm{Id}_s.
\end{align*}

To apply \cite[Theorem 1]{bardet2008dependent}, it remains to check that the dependent Lindeberg conditions hold. For this, we consider some map $f\in \mathcal C_b^3(\mathds R^s,\mathds R)$ where $\mathcal C_b^3(\mathds R^s,\mathds R)$ is the set of functions from $\mathds R^s$ to $\mathds R$ with bounded and continuous partial derivatives up to order $3$. In the following, we denote \[W_i =  \overline G_N^{-1/2} (\mathbf X_{[i-1],M})^{\top}(Y - \overline \pi^{\pi^*})_{[i-1]}= \sum_{a=1}^{i-1}\xi_a.\]

\noindent {\bf First dependent Lindeberg condition.}

For any $i\in [N]$, let us consider $W_i'$ (resp. $\xi'_i$) an independent copy of the random vector $W_i$ (resp. $\xi_i$). Let us recall the following well-known result
\begin{Lemma}\label{lemma:copy}
Let us consider two real valued random variables $A,B$ on some probability space $(\Omega,\mathcal F,\mathds P)$. Let us consider $(A',B')$ an independent copy of the random vector $(A,B)$. Then it holds,
\[Cov(A,B) = \frac12 \mathds E\big[ (A-A')(B-B')\big].\]
\end{Lemma}Using Lemma~\ref{lemma:copy}, the Cauchy-Schwarz inequality and Jensen's inequalities, we get, 
\begingroup
\allowdisplaybreaks
\begin{align*}
&\sum_{k,l=1}^s \sum_{i=1}^N |\overline{Cov}_{\pi^*}(\frac{\partial^2 f}{\partial x_l\partial x_k}(W_i),(\xi_i)_k(\xi_i)_l)|\\
&=\sum_{k,l=1}^s \sum_{i=1}^N |\overline{Cov}_{\pi^*}(\frac{\partial^2 f}{\partial x_l\partial x_k}(W_i),(\xi_i)_k(\xi_i)_l)|\\
&=\sum_{k,l=1}^s \sum_{i=1}^N \frac12 |\overline{\mathds E}_{\pi^*} \left[ \left(\frac{\partial^2 f}{\partial x_l\partial x_k}(W_i)-\frac{\partial^2 f}{\partial x_l\partial x_k}(W'_i)\right)\left((\xi_i)_k(\xi_i)_l-(\xi'_i)_k(\xi'_i)_l\right) \right]|\\
&\leq \sum_{k,l=1}^s \sum_{i=1}^N\frac12 \|\nabla^3 f\|_{\infty} \overline {\mathds E}_{\pi^*}\left( \|W_i-W'_i\|_2 \times |(\xi_i)_k(\xi_i)_l-(\xi'_i)_k(\xi'_i)_l|\right)\\
&\leq \sum_{k,l=1}^s \sum_{i=1}^N\frac12 \|\nabla^3 f\|_{\infty} \sqrt{\overline {\mathds E}_{\pi^*}\left( \|W_i-W'_i\|_2^2 \right)}\times \sqrt{\overline {\mathds E}_{\pi^*}\left( |(\xi_i)_k(\xi_i)_l-(\xi'_i)_k(\xi'_i)_l|^2\right)}
\\
&\leq \sum_{k,l=1}^s \sum_{i=1}^N \|\nabla^3 f\|_{\infty} \sqrt{ \overline {\mathds Var}_{\pi^*}\left( \|W_i\|_2 \right)}\times  \sqrt{\overline {\mathds Var}_{\pi^*}\left(|(\xi_i)_k(\xi_i)_l|\right)}\\
&\leq s \sum_{i=1}^N \|\nabla^3 f\|_{\infty} \sqrt{ \overline {\mathds Var}_{\pi^*}\left( \|W_i\|_2 \right)}\times  \sqrt{\sum_{k,l=1}^s \overline {\mathds Var}_{\pi^*}\left(|(\xi_i)_k(\xi_i)_l|\right)},
\end{align*}
\endgroup
where in the last inequality we used Jensen's inequality. Let us upper-bound the terms $ \overline {\mathds Var}_{\pi^*}\left( \|W_i\|_2 \right)$ and $\sum_{k,l=1}^s \overline {\mathds Var}_{\pi^*}\left(|(\xi_i)_k(\xi_i)_l|\right)$ independently. We have
\begin{align*}
& \overline {\mathds Var}_{\pi^*}\left( \|W_i\|_2 \right)\\
&\leq\overline {\mathds E}_{\pi^*}\left( \|W_i\|_2^2 \right) \\
& = \overline {\mathds E}_{\pi^*}\left[(Y- \overline {\pi}^{\pi^*})^{\top}_{[i-1]} \mathbf X_{[i-1],M}  \overline G^{-1/2}_N \overline G^{-1/2}_N (\mathbf X_{[i-1],M})^{\top}(Y- \overline {\pi}^{\pi^*})_{[i-1]}\right]\\
&= \overline {\mathds E}_{\pi^*} \left[ \mathrm{Tr}\left(   \overline G^{-1/2}_N (\mathbf X_{[i-1],M})^{\top}(Y- \overline {\pi}^{\pi^*})_{[i-1]} (Y- \overline {\pi}^{\pi^*})^{\top}_{[i-1]} \mathbf X_{[i-1],M}  \overline G^{-1/2}_N \right) \right]\\
&= \mathrm{Tr}\left(   \overline G^{-1/2}_N (\mathbf X_{[i-1],M})^{\top}\overline \Gamma^{\pi^*}_{[i-1],[i-1]} \mathbf X_{[i-1],M}  \overline G^{-1/2}_N \right),
\end{align*}
and
\begin{align*}
&\sum_{k,l=1}^s\overline {\mathds Var}_{\pi^*}\left( |(\xi_i)_k(\xi_i)_l| \right) \\
& = \sum_{k,l=1}^s((\overline G_N^{-1/2})_{k,:}\mathbf w_i)^2 ((\overline G_N^{-1/2})_{l,:}\mathbf w_i)^2 \left\{\overline {\mathds E}_{\pi^*}\left[(y_i- \overline \pi^{\pi^*}_i)^4\right]-\overline {\mathds E}_{\pi^*}\left[(y_i- \overline \pi^{\pi^*}_i)^2\right]^2\right\}\\
&=\sum_{k,l=1}^s ((\overline G_N^{-1/2})_{k,:}\mathbf w_i)^2 ((\overline G_N^{-1/2})_{l,:}\mathbf w_i)^2 (\overline \sigma^{\pi^*}_i)^2 (1-2\overline \pi^{\pi^*}_i)^2\\
&= \|\overline G_N^{-1/2}\mathbf w_i\|_2^4(\overline \sigma^{\pi^*}_i)^2 (1-2\overline \pi^{\pi^*}_i)^2\\
&\leq K^4(c \overline \sigma_{\min}^2)^{-2}\frac{s^2}{N^2} (\overline \sigma^{\pi^*}_i)^2 (1-2\overline \pi^{\pi^*}_i)^2,
\end{align*}
where $(\overline \sigma^{\pi^*}_i)^2 = \overline \pi^{\pi^*}_i(1-\overline \pi^{\pi^*}_i)$. Hence, coming back the first Lindeberg condition, we have (forgetting to mention the constants $K,s,c,\overline \sigma_{\min}^2$ that do not depend on $N$, which is the sense of the symbol $\lesssim$),
\begin{align*}
&\sum_{k,l=1}^s \sum_{i=1}^N |\overline {Cov}_{\pi^*}(\frac{\partial^2 f}{\partial x_l\partial x_k}(W_i),(\xi_i)_k(\xi_i)_l)|\\
&\lesssim \frac{1}{N} \sum_{i=1}^N \|\nabla^3 f\|_{\infty} \sqrt{  \mathrm{Tr}\left(  \overline G^{-1/2}_N (\mathbf X_{[i-1],M})^{\top}\overline \Gamma^{\pi^*}_{[i-1],[i-1]} \mathbf X_{[i-1],M} \overline G^{-1/2}_N \right)(1-2\overline \pi^{\pi^*}_i)^2 (\overline \sigma^{\pi^*}_i)^2}  \\
&\leq  \frac{1}{N} \sum_{i=1}^N \|\nabla^3 f\|_{\infty} \sqrt{  \| \overline G^{-1}_N\|_{F} \|(\mathbf X_{[i-1],M})^{\top}\overline \Gamma^{\pi^*}_{[i-1],[i-1]} \mathbf X_{[i-1],M} \|_F(1-2\overline \pi^{\pi^*}_i)^2 (\overline \sigma^{\pi^*}_i)^2}  \\
&\lesssim  \frac{1}{N} \sum_{i=1}^N \|\nabla^3 f\|_{\infty} \sqrt{  \frac{1}{N} \|(\mathbf X_{[i-1],M})^{\top}\overline \Gamma^{\theta^*}_{[i-1],[i-1]} \mathbf X_{[i-1],M}  \|_F(1-2\overline \pi^{\pi^*}_i)^2 (\overline \sigma^{\pi^*}_i)^2}  \\
&\leq  \frac{1}{N^{3/2}} \|\nabla^3 f\|_{\infty} \sum_{i=1}^N  \sqrt{  \|(\mathbf X_{[i-1],M})^{\top}\overline \Gamma^{\pi^*}_{[i-1],[i-1]} \mathbf X_{[i-1],M} \|_F(1-2\overline \pi^{\pi^*}_i)^2 (\overline \sigma^{\pi^*}_i)^2}  ,
\end{align*}
where we used that $\|\overline G_N^{-1}\|_F \leq \sqrt s \|\overline G_N^{-1}\|\lesssim N^{-1} $ (since $\overline G_N^{-1}$ has rank $s$, see Section~\ref{assumption:design}).
Hence, the first dependent Lindeberg condition from \cite{bardet2008dependent} holds thanks to the assumptions made in Theorem~\ref{prop:CLT}.

\medskip

\noindent {\bf Second dependent Lindeberg condition.}\\
Using an approach analogous to the one conducted for the first dependent Lindeberg condition, one can obtain
\begin{align*}
&\sum_{l=1}^s \sum_{i=1}^N |\overline{Cov}_{\pi^*}(\frac{\partial f}{\partial x_l}(W_i),(\xi_i)_l)|\\
& \leq \sqrt s \sum_{i=1}^N \|\nabla^2 f\|_{\infty} \sqrt{ \overline {\mathds Var}_{\pi^*}\left( \|W_i\|_2 \right)}\times  \sqrt{\sum_{l=1}^s \overline {\mathds Var}_{\pi^*}\left(|(\xi_i)_l|\right)}\\
&\lesssim \frac{1}{\sqrt N}\|\nabla^2 f\|_{\infty} \sum_{i=1}^N  \sqrt{   \mathrm{Tr}\left(   \overline G^{-1/2}_N (\mathbf X_{[i-1],M})^{\top}\overline \Gamma^{\pi^*}_{[i-1],[i-1]} \mathbf X_{[i-1],M}  \overline G^{-1/2}_N \right)\left(1- 2 \overline \pi_i^{\pi^*}\right)^2(\overline \sigma^{\pi^*}_i)^2} \\
&\lesssim \frac{1}{\sqrt N}\|\nabla^2 f\|_{\infty} \sum_{i=1}^N  \sqrt{   \| \overline G^{-1}_N\|_F \|(\mathbf X_{[i-1],M})^{\top}\overline \Gamma^{\pi^*}_{[i-1],[i-1]} \mathbf X_{[i-1],M} \|_F\left(1- 2 \overline \pi_i^{\pi^*}\right)^2(\overline \sigma^{\pi^*}_i)^2} \\
&\lesssim   \frac{1}{ N} \|\nabla^2 f\|_{\infty} \sum_{i=1}^N  \sqrt{  \|(\mathbf X_{[i-1],M})^{\top}\overline \Gamma^{\pi^*}_{[i-1],[i-1]} \mathbf X_{[i-1],M} \|_F\left(1- 2 \overline \pi_i^{\pi^*}\right)^2(\overline \sigma^{\pi^*}_i)^2} ,
\end{align*}
where we used that
\begin{align*}
&\overline {\mathds Var}_{\pi^*}\left(|(\xi_i)_l|\right)\\
&= \overline {\mathds E}_{\pi^*}\left(|(\xi_i)_l|^2\right)-\left( \overline {\mathds E}_{\pi^*}|(\xi_i)_l|\right)^2\\
&= ((\overline G_N^{-1/2})_{l,:}\mathbf w_i)^2  \left\{\overline {\mathds E}_{\pi^*}\left((y_i- \overline \pi_i^{\pi^*})^2\right)- \left(\overline {\mathds E}_{\pi^*}|y_i- \overline \pi_i^{\pi^*}|\right)^2\right\}\\
&= ((\overline G_N^{-1/2})_{l,:}\mathbf w_i)^2  \left\{\overline \pi_i^{\pi^*}(1-\overline \pi_i^{\pi^*}) - \left( \overline \pi_i^{\pi^*}(1-\overline \pi_i^{\pi^*}) + (1-\overline \pi_i^{\pi^*}) \overline \pi_i^{\pi^*}\right)^2\right\}\\
&= ((\overline G_N^{-1/2})_{l,:}\mathbf w_i)^2  \overline \pi_i^{\pi^*}(1-\overline \pi_i^{\pi^*}) \left(1- 4 (1-\overline \pi_i^{\pi^*}) \overline \pi_i^{\pi^*}\right)\\
&= ((\overline G_N^{-1/2})_{l,:}\mathbf w_i)^2  (\overline \sigma_i^{\pi^*})^2 \left(1- 2 \overline \pi_i^{\pi^*}\right)^2\\
&\lesssim \frac{1}{N} (\overline \sigma_i^{\pi^*})^2 \left(1- 2 \overline \pi_i^{\pi^*}\right)^2.
\end{align*}
Assuming that \[\sum_{i=1}^N  \sqrt{  \|(\mathbf X_{[i-1],M})^{\top}\overline \Gamma^{\pi^*}_{[i-1],[i-1]} \mathbf X_{[i-1],M} \|_F\left(1- 2 \overline \pi_i^{\pi^*}\right)^2} \underset{N\to\infty}{=} o ( N),\]
we obtain applying~\cite[Theorem 1]{bardet2008dependent} the following CLT
\[\overline G_N^{-1/2} \mathbf X_M^{\top}(Y-\overline \pi^{\pi^*})\overset{(d)}{\underset{N\to + \infty}{\longrightarrow}} \mathcal N(0,\mathrm {Id}_s).\]

\subsection{Proof of Theorem~\ref{thm:MLEasymptotic}}
\label{proof:thmMLE}


To make the notations less cluttered, we will simply denote in the following~$\overline G_N(\theta^*)$ by~$\overline G_N$ and~$\overline \theta (\theta^*)$ by $\overline \theta$.

\paragraph{First step.}

We use Theorem~\ref{prop:CLT} where we established a CLT for
\[-L_{N}(\overline \theta,(Y,\mathbf X_M))= \mathbf X_M^{\top}(Y-\pi^{\overline \theta})=\mathbf X_M^{\top}(Y- \overline \pi ^{\theta^*})=\mathbf X_M^{\top}(Y-\overline \pi^{\pi^*}).\] 
Let us highlight that the first equality comes directly from the definition of $L_{N}(\overline \theta,(Y,\mathbf X_M))$ (see Section~\ref{sec:intro-conditional-MLE}), the second equality comes from Eq.\eqref{eq:gradbar} and the last equality holds since we work under the selected model meaning that $\pi^*=\sigma(\mathbf X\vartheta^*)=\sigma(\mathbf X_M\theta^*)$ (and thus that $\overline{\mathds P}_{\theta^*} \equiv \overline{\mathds P}_{\pi^*}$). Let us recall that to prove Theorem~\ref{prop:CLT}, we used a variant of the Linderberg CLT for dependent random variables proved by~\cite{bardet2008dependent}. The proof of Theorem~\ref{prop:CLT} is given in Section~\ref{proof:propCLT}.

\paragraph{Second step.}

We now prove that for any~$\epsilon> 0$ there is some~$\delta > 0$ such that when~$N$ is large
enough
\[\overline {\mathds P}_{\theta^*}\left( \text{there is }
\widehat \theta \in \mathcal N_N(\overline \theta,\delta)\text{ such that } L_{N}(\widehat \theta,(Y,\mathbf X_M))= 0\right) > 1 - \epsilon,\]
with $\mathcal N_N(\overline \theta,\delta) = \{\theta \;  :\;  \|\overline G_N^{1/2} (\theta - \overline \theta)\|_2 \leq \delta \}$. Stated otherwise, we will prove that there exist a constant $\delta>0$ and an integer $N_{\delta}\in \mathds N$ such that for any $N\geq N_{\delta}$, the following holds with high probability,
\begin{itemize}
\item the conditional MLE $\widehat \theta$ exists, 
\item the conditional MLE $\widehat \theta$ is contained in the ellipsoid $\mathcal N_N(\overline \theta, \delta)$ centered at $\overline \theta$.
\end{itemize}
Let us denote 
\begin{align*}F:\theta \in \mathds R^s \mapsto&\,\, \overline G_N^{-1/2} (L_{N}(\overline \theta,(Y,\mathbf X_M))-L_{N}(\theta,(Y,\mathbf X_M)))\\&=\overline G_N^{-1/2} \mathbf X_M^{\top} (\pi^{\overline \theta} - \pi^{\theta}).\end{align*} Note that $F$ is a deterministic function and does not depend on the random variable $Y$. Moreover we choose to leave implicit the dependence on $N$ of $F$. We also point out that it holds for any $\theta \in \mathds R^s$, \[ \nabla_{\theta} F( \theta)= -\overline G^{-1/2}_N \mathbf X_M^{\top} \mathrm{Diag}(\sigma'(\mathbf X_M \theta)) \mathbf X_M =-\overline G_N^{-1/2} H_N( \theta).\]
Hence~$F$ is a~$\mathcal C^1$ map with invertible Jacobian at any~$ \theta\in \mathds R^s$ and is injective (thanks to Proposition~\ref{prop:diffeo-g}). Applying the global inversion theorem, we deduce that~$F$ is a~$\mathcal C^1$-diffeomorphism from~$\mathds R^s$ to~$\mathds R^s$.

\bigskip

\begin{changebar}
\underline{Sketch of proof.}\\
In the following, we prove that for any $\epsilon$, we can choose $\delta>0$ such that for some $N_{\delta} \in \mathds N$ and for any $N\geq N_{\delta}$, it holds on some event $E_N$ satisfying $\overline {\mathds P}_{\theta^*}(E_N)\geq 1- \epsilon$,
\begin{align}&\overline G_N^{-1/2} L_N(\overline \theta,(Y,\mathbf X_M)) \in F(\mathcal N_N(\overline \theta,\delta))\notag\\
\Leftrightarrow \quad & \overline G_N^{-1/2}  (\underbrace{\mathbf X_M^{\top}\overline \pi^{\theta^*} }_{=\mathbf X_M^{\top}\pi^{\overline \theta}}-\mathbf X_M^{\top}Y) \in F(\mathcal N_N(\overline \theta,\delta)).\label{eq:proof-existenceMLE}\end{align}
This would mean (by definition of $F$) that on $E_N$, there exists some $\widehat \theta  \in \mathcal N_N(\overline \theta, \delta)$ such that $\overline G_N^{-1/2}L_N(\widehat \theta,(Y,\mathbf X_M))=0$ or equivalently that $L_N(\widehat \theta,(Y,\mathbf X_M))=0$. A sufficient condition for Eq.\eqref{eq:proof-existenceMLE} to hold is to check that on the event $E_N$ it holds
\begin{equation}\label{eq:in-ball}\|\overline G_N^{-1/2}L_N(\overline \theta,(Y,\mathbf X_M))\|_2 < \inf_{\theta \in \partial \mathcal N_N(\overline \theta,\delta)} \|F(\theta)\|_2, \end{equation}
where $\partial \mathcal N_N(\overline \theta,\delta) := \{ \theta \in \mathds R^s \, |\, \|\overline G_N^{1/2}(\theta - \overline \theta)\|_2=\delta\}.$ This sufficient condition is a direct consequence of Lemma~\ref{lemma:boundary} and Figure~\ref{fig:existenceMLE} gives a visualization of our proof strategy.
\end{changebar}

\begin{Lemma}\label{lemma:boundary}
Let $f: \mathds R ^s \to \mathds R^s$ be a $\mathcal C^1$-diffeomorphism from $\mathds R^s$ to $f(\mathds R^s)$. Then for any closed space $D \subset \mathds R^s$ it holds \[f(\partial D) = \partial f(D),\]
where for any set $U\subseteq \mathds R^s$, $\partial U= \overline U \backslash \mathring{U}$ with $\overline U$ the closure of the set $U$ and $\mathring U$ the interior of the set $U$.
\end{Lemma}
\begin{proof}
As a $\mathcal C^1$-diffeomorphism, $f$ is in particular a homeomorphism, and as such, it preserves the topological structures.
\end{proof}

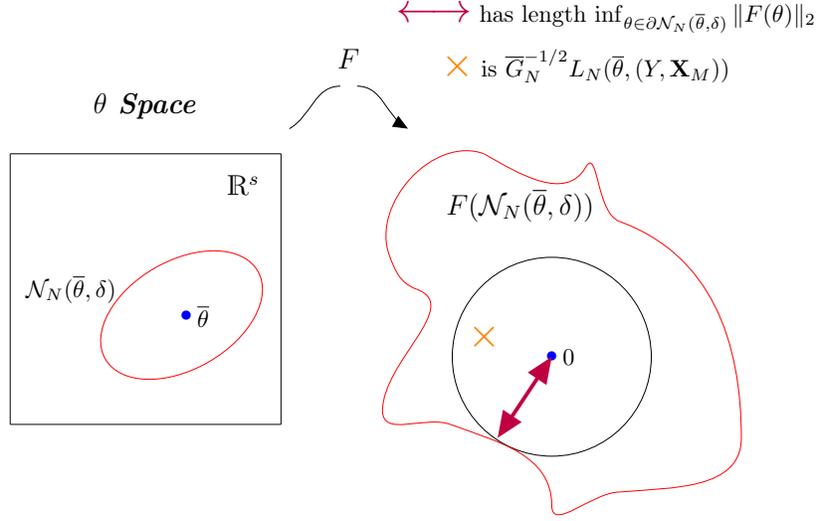
\begin{figure}[!ht]
\centering
\scalebox{0.9}{
\begin{tikzpicture}
    \node  (r3) at (0,5)  {};
    \node[label={[label distance=0.1cm]-135:{\large $\mathds R^s$}}]  (r2) at (4,5)  {};
    \node  (r1) at (4,1)  {};
    \node  (r0) at (0,1)  {};  
    \draw (r0.center) to (r1.center);
    \draw (r1.center) to (r2.center);
    \draw (r2.center) to (r3.center);
    \draw (r3.center) to (r0.center);

    \node (thetaspace) at (2,5.7) {{\bf \large $\theta$ \textit{Space}}};
    \draw[rotate=30,color=red] (3.5,1) ellipse (1.3cm and 0.8cm);
    \node[color=blue] (overlinepoint) at (2.6,2.6) {$\bullet$};
    \node  at ([xshift=0.25cm]$(overlinepoint)$) {$\overline \theta$};
    \node (X12) at (0.9,3) {$ \mathcal N_N(\overline \theta, \delta)$};

    \node  (c4) at (7,5)  {};
    \node[label={[label distance=0.8cm]70:{\large $F(\mathcal N_N(\overline \theta,\delta))$}}]   (c3) at (6,3)  {};
    \node  (c2) at (6.5,1)  {};
    \node  (c1) at (9,0)  {};
    \node  (c0) at (9,4)  {};  
    \node[color=blue] (X12) at (8,2) {$ \bullet$};
    \node (X12point) at ([xshift=0.25cm]$(X12)$) {$0$};
    \draw (8,2) ellipse (1.47cm and 1.47cm);
    \node[color=orange] (gn) at (7,2.3) {{\LARGE$ \times$}};
    \node (gnpt) at (8.5,6.3) {{\LARGE${\color{orange} \times}$} is $\overline G_N^{-1/2} L_N(\overline \theta,(Y,\mathbf X_M))$};
    \node (gnpt2) at (8.8,7) {{\LARGE${\color{purple} \longleftrightarrow}$} has length $\inf_{\theta \in \partial \mathcal N_N(\overline \theta,\delta)} \|F(\theta)\|_2$};
    \draw[color=purple,ultra thick,<->] (7.2,0.8) -- (8,2);

    \draw[color=red] (c1.center) to [out=5,in=-90]++(1.8,0.8) to[out=90,in=-20](c0.center);
    \draw[color=red] (c1.center) to [out=-180,in=-90]++(-1,-0.2) to[out=90,in=-20](c2.center);
    \draw[color=red] (c2.center) to [out=160,in=-90]++(-1,0.2) to[out=90,in=-20](c3.center);
    \draw[color=red] (c3.center) to [out=160,in=-70]++(-0.4,0.5) to[out=110,in=160](c4.center);
    \draw[color=red] (c4.center) to [out=-30,in=-120]++(1.5,-0.2) to[out=60,in=160](c0.center);

    \node (R) at ([xshift=-1cm,yshift=0.3cm]$(c4)$) {};
    \node (C) at ([xshift=-3cm,yshift=0.3cm]$(c4)$) {};
    \node[label={[label distance=0.02cm]90:{\large $F$}}]  (CR) at ([yshift=0.7cm]$(C)!0.5!(R)$) {};
    \draw[->] (C) to[out=30,in=-180]  (CR)
    to [out=0,in=150](R) ; 
    
\end{tikzpicture}
}
\caption{Visualization support for the proof of the existence of the MLE with large probability in a neighbourhood of $\overline \theta$.  We show that with large probability, the orange cross is in the black circle ({\it i.e.,} Eq.\eqref{eq:in-ball} holds) which implies that the orange cross belongs to $F(\mathcal N_N(\overline \theta, \delta))$ ({\it i.e.,} Eq.\eqref{eq:proof-existenceMLE} holds). The MLE is then defined as $\widehat \theta= F^{-1}(\overline G_N^{-1/2}L_N (\overline \theta, (Y,\mathbf X_M)) \in \mathcal N_N(\overline \theta,\delta)$.}
\label{fig:existenceMLE}
\end{figure}

\bigskip

Let $\epsilon>0$ and let us consider 
\begin{equation} \label{eq:defdelta}\delta:= \frac{ \mathfrak K^{1/2}}{\epsilon^{1/2}  2C^{-1}c \overline \sigma^2_{\min}},\end{equation} (the reason of this choice will become clear with Eq.\eqref{eq:choice-delta}). Let us first notice that for any $\theta \in \mathds R^s$,
\begin{align}
&   L_{N}(\overline \theta,(Y,\mathbf X_M)) - L_{N}( \theta,(Y,\mathbf X_M)) \\
&= \mathbf X_M^{\top}(\pi^{\overline \theta}-\pi^{ \theta})\\
&=  \underbrace{\int_0^1 H_N(t \overline \theta+(1-t)  \theta)dt}_{=:Q_N( \theta)} \;(\overline \theta -  \theta),\label{eq:intGN}\end{align}
where we used that the Jacobian of the map $\theta \mapsto \mathbf X_M^{\top} \pi^{\theta} = \mathbf X_M ^{\top}\sigma(\mathbf X_M\theta)$ is $\mathbf X_M \mathrm{Diag}(\sigma'(\mathbf X_M \theta)) \mathbf X_M = H_N(\theta)$. Recalling further that $\|\overline G_N^{-1/2}(\theta-\overline \theta)\|_2=\delta$ for any $\theta \in \partial \mathcal N_N(\overline \theta, \delta)$, it holds,
\begin{align}
&\inf_{\theta \in  \partial \mathcal N_N(\overline \theta,\delta)} \|F(\theta)\|_2 \notag\\
&= \inf_{\theta \in  \partial \mathcal N_N(\overline \theta,\delta)} \|\overline G_N^{-1/2} Q_N(\theta) (\theta- \overline \theta)\|_2\quad \text{(using Eq.\eqref{eq:intGN})} \notag\\
&= \inf_{\theta \in \partial \mathcal N_N(\overline \theta,\delta)} \|\overline G_N^{-1/2} Q_N(\theta) (\theta- \overline \theta) \|_2 \times \frac{\|\overline G_N^{1/2} (\theta-\overline \theta)\|_2}{\|\overline G_N^{1/2}  (\theta-\overline \theta)\|_2} \notag\\
&\geq \inf_{\theta \in \partial \mathcal N_N(\overline \theta, \delta)} \frac{(\theta-\overline \theta)^{\top} Q_N(\theta)(\theta-\overline \theta)}{\|\overline G_N^{1/2}  (\theta-\overline \theta)\|_2} \quad \text{(using the Cauchy Schwarz's inequality)} \notag\\
& =\delta \inf_{\theta \in \partial \mathcal N_
N(\overline \theta, \delta)} \frac{(\theta-\overline \theta)^{\top}\overline G_N^{1/2}}{\|\overline G_N^{1/2}  (\theta-\overline \theta)\|_2}\overline G_N^{-1/2}  Q_N(\theta)\overline G_N^{-1/2} \frac{\overline G_N^{1/2} (\theta-\overline \theta)}{\|\overline G_N^{1/2}  (\theta-\overline \theta)\|_2} \notag\\
&\geq \delta  \inf _{\|e\|_2=1,\theta \in \partial \mathcal N_N(\overline \theta,\delta)} e^{\top} \overline G_N^{-1/2} Q_N(\theta) \overline G_N^{-1/2} e \notag\\
&= \delta \inf _{\|e\|_2=1,\theta \in \partial \mathcal N_N(\overline \theta,\delta)} e^{\top} \overline G_N^{-1/2} \int_0^1H_N(t\overline \theta+(1-t) \theta)dt \overline G_N^{-1/2} e \notag\\
&= \delta \inf _{\|e\|_2=1,\theta \in \partial \mathcal N_N(\overline \theta,\delta)} \int_0^1 \left(e^{\top} \overline G_N^{-1/2} H_N(t\overline \theta+(1-t) \theta) \overline G_N^{-1/2} e\right) dt \notag\\
&\geq \delta  \inf _{\|e\|_2=1,\theta \in  \mathcal N_N(\overline \theta,\delta)} e^{\top} \overline G_N^{-1/2} H_N(\theta) \overline G_N^{-1/2} e \notag\\
&\geq \delta \left\{\inf _{\|e\|_2=1} e^{\top} \overline G_N^{-1/2} H_N(\overline \theta) \overline G_N^{-1/2} e - \mathcal C\frac{\delta}{N^{1/2}}\right\}=: \mathcal I_N(\delta,\overline \theta),\label{eq:lowerboundF}
\end{align}
where in the penultimate inequality we used that $\overline \theta \in \mathcal N_N(\overline \theta,\delta)$ and the convexity of $ \mathcal N_N(\overline \theta,\delta)$. In the last inequality, we used Lemma~\ref{lemma:supN} whose proof is postponed to Section~\ref{proof:lemmasupN}.

\begin{Lemma}\label{lemma:supN}Let us consider some $\delta>0$. Then for any $N\in\mathds N$ and for any unit vector $u \in \mathds R^s$, it holds
\[\sup_{\theta \in \mathcal N_N(\overline \theta,\delta)} |u^{\top}\overline  G_N^{-1/2}(H_N(\theta)-H_N(\overline \theta))\overline  G_N^{-1/2}
u | \leq \mathcal C \frac{\delta}{N^{1/2}},\]
where $\mathcal N_N(\overline \theta,\delta) = \{\theta \in \mathds R^s\;  :\;  \|\overline  G_N^{1/2} (\theta - \overline \theta)\|_2 \leq \delta \}$ and where $\mathcal C$ is a constant that only depends on the quantities $s,K,c,\overline \sigma^2_{\min}$ (that do not depend on $N$).
\end{Lemma}
To lower bound uniformly in $N$ the term $\mathcal I_N(\delta,\overline \theta)$, we notice that
\begin{align*}
&\inf _{\|e\|_2=1} e^{\top} \overline G_N^{-1/2} H_N(\overline \theta) \overline G_N^{-1/2} e \\
&= \inf _{\|e\|_2=1} \frac{e^{\top} \overline G_N^{-1/2}}{\|\overline G_N^{-1/2}e\|_2} H_N(\overline \theta)\frac{ \overline G_N^{-1/2} e}{\|\overline G_N^{-1/2}e\|_2} \|\overline G_N^{-1/2}e\|^2_2\\
&\geq \lambda_{\min}(H_N(\overline \theta)) \inf_{\|e\|_2=1} \|\overline G_N^{-1/2}e\|^2_2\\
&\geq\lambda_{\min}(H_N(\overline \theta)) \, \lambda_{\min}(\overline G_N^{-1})\\
&\geq \left(\overline \sigma^2_{\min} cN \right) \times \left(4C^{-1}N^{-1}\right) \\
&\geq 4C^{-1}c \overline \sigma^2_{\min},
\end{align*}
where we used that for any $i \in [N]$, $\sigma'(\mathbf x_{i,M} \overline \theta) \geq \overline \sigma^2_{\min}$.
Let us denote $N_{\delta}:=\ceil{\big(\frac{\mathcal C \delta}{2C^{-1}c \overline \sigma^2_{\min}}\big)^2}$ so that for any $N\geq N_{\delta}$ it holds
\[ \mathcal I_N(\delta,\overline \theta) \geq \delta2C^{-1}c \overline \sigma^2_{\min}.\]
Using Markov's inequality, we get that for any $N\geq N_{\delta}$,
\begin{align*}
&\overline {\mathds P}_ {\theta^*} ( \|\overline G_N^{-1/2} L_N(\overline \theta,(Y,\mathbf X_M))\|_2 \geq \mathcal I_N(\delta,\overline \theta))\\
&\leq (\mathcal I_N(\delta,\overline \theta))^{-2} \overline {\mathds E}_{\theta^*} ( \|\overline G_N^{-1/2} L_N(\overline \theta,(Y,\mathbf X_M))\|_2^2 )\\
&\leq (\mathcal I_N(\delta,\overline \theta))^{-2} \overline {\mathds E}_{\theta^*} ( (Y- \overline {\pi}^{\theta^*})^{\top} \mathbf X_M \overline G_N^{-1}\mathbf X_M^{\top} (Y- \overline  \pi^{\theta^*}) )\\
&= (\mathcal I_N(\delta,\overline \theta))^{-2} \overline {\mathds E}_{\theta^*} ( \mathrm{Tr}\left[(Y- \overline {\pi}^{\theta^*})^{\top} \mathbf X_M \overline G_N^{-1}\mathbf X_M^{\top} (Y- \overline  \pi^{\theta^*})\right] )\\
&= (\mathcal I_N(\delta,\overline \theta))^{-2} \overline {\mathds E}_{\theta^*} ( \mathrm{Tr}\left[\mathbf X_M \overline G_N^{-1}\mathbf X_M^{\top} (Y- \overline  \pi^{\theta^*})(Y- \overline {\pi}^{\theta^*})^{\top} \right] )\\
&= (\mathcal I_N(\delta,\overline \theta))^{-2}  \mathrm{Tr}\left[\mathbf X_M \overline G_N^{-1}\mathbf X_M^{\top} \overline {\Gamma}^{\theta^*} \right] \\
&=  (\mathcal I_N(\delta,\overline \theta))^{-2}   \mathrm{Tr}\left[\overline G_N^{-1}\mathbf X_M^{\top} \overline {\Gamma}^{\theta^*}\mathbf X_M  \right].
\end{align*}
Hence, it holds for any $N\geq N_{\delta}$,
\begin{align}
&\overline {\mathds P}_ {\theta^*} ( \|\overline G_N^{-1/2} L_N(\overline \theta,(Y,\mathbf X_M))\|_2 \geq \mathcal I_N(\delta,\overline \theta))\notag\\
&\leq  \frac{\mathrm{Tr}\left[\overline G_N^{-1}\mathbf X_M^{\top} \overline {\Gamma}^{\theta^*}\mathbf X_M  \right]}{ \mathcal I_N(\delta,\overline \theta)^2 }\notag\\
&< \frac{\mathfrak K }{\delta^2 (2C^{-1}c \overline \sigma^2_{\min})^2}\notag\\
&\leq \epsilon, \label{eq:choice-delta}
\end{align}
where the last inequality comes from the choice of $\delta$ (see Eq.\eqref{eq:defdelta}). From Eq.\eqref{eq:lowerboundF} and Eq.\eqref{eq:choice-delta}, we deduce that for any $N \geq N_{\delta}$, it holds 
\[\overline {\mathds P}_{\theta^*}(E_N) \geq 1-\epsilon,\]
where
\[E_N := \left\{\|\overline G_N^{-1/2}L_N(\overline \theta,(Y,\mathbf X_M))\|_2 < \inf_{\theta \in \partial \mathcal N_N(\overline \theta,\delta)} \|F(\theta)\|_2\right\}.\]
Hence, on the event $E_N$, we define $\widehat \theta = F^{-1}(\overline G_N^{-1/2}L_N(\overline \theta,(Y,\mathbf X_M)))$ which means by definition of $F$ that $\widehat \theta$ is the conditional MLE, namely
\[L_N(\widehat \theta, (Y, \mathbf X_M)) = 0.\]

\paragraph{Third and final step.} In the previous step, we proved that for~$N$ large enough, the MLE exists and is contained in an ellipsoid centered at~$\overline \theta$ with vanishing volume with high probability. Now we show how using this result to turn the CLT on~$L_{N}(\overline \theta,(Y,\mathbf X_M))$ from Theorem~\ref{prop:CLT} into a CLT for~$\widehat \theta$. 

We consider $N \geq N_{\delta}$ and we work on the event~$E_N$ of the previous step. Since~$L_{N}(\widehat \theta,(Y,\mathbf X_M))=0$ by definition of~$\widehat \theta$, we get that
\begin{align*}
L_{N}(\overline \theta,(Y,\mathbf X_M))& =   L_{N}(\overline \theta,(Y,\mathbf X_M)) - L_{N}(\widehat \theta,(Y,\mathbf X_M)) \\
&= \mathbf X_M^{\top}(\pi^{\overline \theta}-\pi^{\widehat \theta})\\
&=  \underbrace{\int_0^1 H_N(t \overline \theta+(1-t) \widehat \theta)dt}_{=Q_N(\widehat \theta)} \;(\overline \theta - \widehat \theta),\end{align*}
where we used that the Jacobian of the map $\theta \mapsto \mathbf X_M^{\top} \pi^{\theta} = \mathbf X_M \sigma(\mathbf X_M\theta)$ is $\mathbf X_M \mathrm{Diag}(\sigma'(\mathbf X_M \theta)) \mathbf X_M = H_N(\theta)$. From the Portmanteau Theorem \citep[cf.][Lemma 2.2]{van2000asymptotic}), we know that a sequence of $\mathds R^s$-valued random vectors $(X_n)_n$ converges weakly to a random vector $X$ if and only if for any Lipschitz and bounded function $h:\mathds R^s\to \mathds R$ it holds
\[\mathds E h(X_n) \underset{n\to \infty}{\to} \mathds E h(X).\] Hence, we consider a Lipschitz and bounded function $h:\mathds R^s\to \mathds R$. We denote by $L_h>0$ the Lipschitz constant of $h$. It holds for any $N\geq N_{\delta}$,
\begin{align}
&|\overline{\mathds E}_{\theta^*} [h(\overline G_N^{-1/2} H_N(\overline \theta)(\overline \theta-\widehat \theta))]-\overline{\mathds E}_{\theta^*}\big[h\big(\overline G_N^{-1/2}L_N(\overline \theta,(Y,\mathbf X_M))\big)\big]| \notag\\
&= |\overline{\mathds E}_{\theta^*} [h(\overline G_N^{-1/2} H_N(\overline \theta)(\overline \theta-\widehat \theta))]-\overline{\mathds E}_{\theta^*}[h(\overline G_N^{-1/2}Q_N(\widehat \theta)(\overline \theta - \widehat \theta))]| \notag\\
&\leq |\overline{\mathds E}_{\theta^*} \left[\mathds 1_{E_N}\left\{h(\overline G_N^{-1/2} H_N(\overline \theta)(\overline \theta-\widehat \theta))-h\big(\overline G_N^{-1/2}Q_N(\widehat \theta)(\overline \theta - \widehat \theta)\big)\right\}\right]|+ 2\|h\|_{\infty} \overline {\mathds P}_{\theta^*}(E_N^c) \notag\\
&\leq \overline{\mathds E}_{\theta^*} \big[ L_h \mathds 1_{E_N} \|\overline G_N^{-1/2}H_N(\overline \theta)(\overline \theta-\widehat \theta)-\overline G_N^{-1/2}Q_N(\widehat \theta)(\overline \theta - \widehat \theta)\|_2 \big]+2\|h\|_{\infty}\epsilon \notag\\
&\leq L_h\overline{\mathds E}_{\theta^*} \big[ \mathds 1_{E_N}\|\overline G_N^{-1/2}(H_N(\overline \theta)-Q_N(\widehat \theta))\overline G_N^{-1/2}\| \|\overline G_N^{1/2}(\overline \theta - \widehat \theta)\|_2 \big]+2\|h\|_{\infty}\epsilon \notag\\
&\leq L_h\delta \sup_{\theta \in \mathcal N_N(\overline \theta,\delta)} \|\overline G_N^{-1/2}(H_N(\overline \theta)-Q_N( \theta))\overline G_N^{-1/2}\| + 2\|h\|_{\infty}\epsilon, \label{eq:portemanteau1}
\end{align}
where we used that on the event $E_N$, $\widehat \theta \in \mathcal N_N(\overline \theta,\delta)$, i.e. $\|\overline G_N^{1/2}(\overline \theta - \widehat \theta)\|_2\leq \delta$. Moreover, for any $\theta' \in \mathcal N_N(\overline \theta,\delta)$ we have,

\begin{align}
&\|\overline G_N^{-1/2}(H_N(\overline \theta)-Q_N( \theta'))\overline G_N^{-1/2}\|\notag\\
&=\sup_{\|u\|_2=1} |u^{\top}\overline G_N^{-1/2}(H_N(\overline \theta)-Q_N( \theta'))\overline G_N^{-1/2}u|\notag\\
&\leq \sup_{\|u\|_2=1}\int_0^1 \left|u^{\top}\overline G_N^{-1/2}(H_N(\overline \theta)-H_N(t\overline \theta + (1-t)  \theta'))\overline G_N^{-1/2} u\right|dt\notag\\
&\leq \sup_{\|u\|_2=1}  \sup_{\theta \in \mathcal N_N(\overline \theta,\delta)}|u^{\top}\overline G_N^{-1/2}(H_N(\overline \theta)-H_N( \theta))\overline G_N^{-1/2}u|\notag\\
&\leq \mathcal C \frac{\delta}{N^{1/2}},\label{eq:portemanteau2}
\end{align}
where in the penultimate inequality we used the convexity of the set $\mathcal N_N(\overline \theta,\delta))$ and in the last inequality we used Lemma~\ref{lemma:supN} (which is proved in Section~\ref{proof:lemmasupN}). Using Eq.\eqref{eq:portemanteau1} and Eq.\eqref{eq:portemanteau2}, we deduce that for $G\sim \mathcal N(0,\mathrm{Id}_s)$ we have
\begin{align}
& |\overline{\mathds E}_{\theta^*} [h(\overline G_N^{-1/2} H_N(\overline \theta)(\overline \theta-\widehat \theta))] - \mathds E[h(G)]|\notag\\ 
&\leq |\overline{\mathds E}_{\theta^*} [h(\overline G_N^{-1/2} H_N(\overline \theta)(\overline \theta-\widehat \theta))]-\overline{\mathds E}_{\theta^*}\big[h\big(\overline G_N^{-1/2}L_N(\overline \theta,(Y,\mathbf X_M))\big)\big]|\notag\\
&\qquad +|\overline{\mathds E}_{\theta^*}\big[h\big(\overline G_N^{-1/2}L_N(\overline \theta,(Y,\mathbf X_M))\big)\big]-\mathds E[h(G)]|\notag\\
&\leq L_h\delta \mathcal C \frac{\delta}{N^{1/2}} + 2\|h\|_{\infty}\epsilon + |\overline{\mathds E}_{\theta^*}\big[h\big(\overline G_N^{-1/2}L_N(\overline \theta,(Y,\mathbf X_M))\big)\big]-\mathds E[h(G)]|.\label{eq:endproofCLT}
\end{align}
The CLT from Theorem~\ref{prop:CLT} states that
\[\overline G_N^{-1/2}L_N(\overline \theta,(Y,\mathbf X_M)) \overset{(d)}{\underset{N\to \infty}{\longrightarrow}} \mathcal N(0,\mathrm{Id_s}),\]
which means by the Portmanteau Theorem \citep[cf.][Lemma 2.2]{van2000asymptotic}) that
\[|\overline{\mathds E}_{\theta^*}\big[h\big(\overline G_N^{-1/2}L_N(\overline \theta,(Y,\mathbf X_M))\big)\big]-\mathds E[h(G)]| \underset{N\to +\infty}{\to}0.\]
We deduce that for any $\epsilon>0$ and for any Lipschitz and bounded function $h: \mathds R^s \to \mathds R$, one can choose $N$ large enough to ensure that the right hand side of Eq.\eqref{eq:endproofCLT} is smaller than $4\|h\|_{\infty} \epsilon$. Note that this is true since the constant $\delta$ does not depend on $N$. This concludes the proof thanks to the Portmanteau Theorem.

\subsection{Proof of Lemma~\ref{lemma:supN}}
\label{proof:lemmasupN}

Let us first recall that $ H_N(\overline \theta)=\mathbf X_M^{\top} \mathrm{Diag}(\sigma'(\mathbf X_M \overline \theta)) \mathbf X_M$ and that $\mathbf X_M^{\top}= \left[ \mathbf w_1 \;|\; \mathbf w_2 \; |\; \dots \;| \; \mathbf w_N \right]$, where $\mathbf w_i = \mathbf x_{i,M}\in \mathds R^s$. Let us consider some $\theta \in \mathcal N_N(\overline \theta,\delta)$. We have that
\begin{align}
H_N(\theta) - H_N(\overline \theta)& = \sum_{i=1}^N \mathbf w_i \left[ \sigma'(\mathbf w_i^{\top} \theta) -\sigma'(\mathbf w_i^{\top}\overline \theta)\right]\mathbf w_i^{\top}\notag\\
&=\sum_{i=1}^N \mathbf w_i \underbrace{\int _0^1 \sigma''(t\mathbf w_i^{\top} \theta + (1-t)\mathbf w_i^{\top} \overline \theta) dt}_{=: H_i} \mathbf w_i^{\top} (\theta-\overline \theta)\mathbf w_i^{\top}.\label{eq:GN-GNbar}
\end{align}
We get using Eq.\eqref{eq:GN-GNbar} that for any unit vector~$u \in \mathds R^s$,
\begin{align}
&|u^{\top}\overline G_N^{-1/2}
(H_N(\theta) - H_N(\overline \theta))\overline G_N^{-1/2}u|
\notag\\
&=\left|\sum_{i=1}^N u^{\top}\overline G_N^{-1/2} \mathbf w_iH_i\mathbf w_i^{\top} (\theta-\overline \theta)\mathbf w_i^{\top}
\overline G_N^{-1/2}u \right|\notag\\
&=\left|\sum_{i=1}^N \mathbf w_i^{\top} (\theta-\overline \theta) \times  u^{\top}\overline G_N^{-1/2} \mathbf w_iH_i\mathbf w_i^{\top}
\overline G_N^{-1/2}u\right|\notag\\
&=\left|\sum_{i=1}^N \mathbf w_i^{\top} (\theta-\overline \theta) \times  H_i |\mathbf w_i^{\top}
\overline G_N^{-1/2}u|^2\right|\notag\\
&\leq \max_{1\leq j\leq N} |\mathbf w_j^{\top} (\theta-\overline \theta)| \sum_{i=1}^N  |H_i| |\mathbf w_i^{\top}
\overline G_N^{-1/2}u|^2\notag\\
&= \max_{1\leq j\leq N} |\mathbf w_j^{\top} (\theta-\overline \theta)| \; \|\mathbf H^{1/2} \mathbf X_M^{\top}
\overline G_N^{-1/2}u\|_2^2,\label{eq:lemma-GN}
\end{align}
where $\mathbf H^{1/2} := \mathrm{Diag}((|H_i|^{1/2})_{i\in [N]})$. The proof is concluded by upper-bounding both terms involved in the product of the right hand side of Eq.\eqref{eq:lemma-GN}.
Using the assumption of the design matrix presented in Section~\ref{assumption:design} and recalling that $\theta \in \mathcal N_N(\overline \theta,\delta)$, we have
\begin{align*}
\max_{1\leq j\leq N} |\mathbf w_j^{\top} (\theta-\overline \theta)| &\leq \max_{1\leq j\leq N} \|\overline G_N^{-1/2}\mathbf w_j\|_2  \underbrace{\|\overline G_N^{1/2}(\theta-\overline \theta)\|_2}_{\leq \delta}\\
&= \delta  K \sqrt{(\overline \sigma^2_{\min} c)^{-1}s}  N^{-1/2},
\end{align*}
where we used that $\|\overline G_N^{-1/2}\|^2=\|\overline G_N^{-1}\|\leq (c\overline \sigma^2_{\min}N)^{-1}$ and that for any $i\in [N]$, $\|\mathbf w_i\|_2^2 \leq sK^2.$ Since $|H_i|\leq 1$ for any $i\in [N]$,
\begin{align*}
\|\mathbf H^{1/2} \mathbf X_M^{\top}
\overline G_N^{-1/2}u\|_2^2&\leq \| \mathbf X_M^{\top}
\overline G_N^{-1/2}u\|_2^2\\
&=\sum_{i=1}^N ( \mathbf w_{i}^{\top}
\overline G_N^{-1/2}u)^2\\
&\leq \sum_{i=1}^N \|
\overline G_N^{-1/2}\mathbf w_{i}\|_2^2 \leq (\overline \sigma_{\min}^2c)^{-1} s K^2,
\end{align*}
where in the penultimate inequality we used Cauchy-Schwarz inequality.

\subsection{Proof of Proposition~\ref{prop:test-selected}}
\label{proof:prop:test-selected}

For any $N \in \mathds N$, let us denote 
\begin{equation}\label{def:EN-MLE}\mathcal E_N:=\{ Z\in \{0,1\}^N\, |\, \mathbf X_M^{\top}Z \in \mathrm{Im}(\Xi)\}.\end{equation}
In order to clarify the notations of this proof, let us stress that we denote in the following by $\overline {\mathds P} _{\theta^*_0}$ the distribution of $Y$, $\mathds P_1$ the distribution of the sequence $(Y^{(t)})_{t\geq 1}$ and $\mathds P_2$ the distribution of $(Z^{(t)})_{t\geq 1}$. Let us consider some~$\epsilon>0$. 

\bigskip

\noindent {\bf Step 1:  $\mathds P_1$ almost sure convergences.}\\
From Proposition~\ref{prop:point-alternative}, we know that under the null $\mathds H_0$
 \begin{equation}\label{eq:cv-asp1}\frac{\sum_{t=1}^T Y^{(t)} \mathds P_{\theta^*_0}(Y^{(t)})}{\sum_{t=1}^T  \mathds P_{\theta^*_0}(Y^{(t)})} \underset{T \to \infty}{\to} \overline {\mathds E}_{\theta^*_0}\left[Y \right]=\overline \pi^{\theta^*_0} \quad \mathds P_1- \text{almost surely}.\end{equation}
Since~$\widetilde \pi^{\theta_0^*} \underset{T \to \infty}{\to} \overline \pi^{\theta^*_0}$ $\mathds P_1$-a.s., we know that $\mathds P_1$-a.s, there exists some~$T_1 \in\mathds N$ such that for any~$T\geq T_1$ it holds
\[\|\widetilde \pi^{\theta_0^*}\odot (1-  \widetilde \pi^{\theta^*_0}) - \overline \pi^{\theta^*_0}\odot (1-  \overline \pi^{\theta^*_0})\|_{\infty} <\epsilon,\] and since $(\overline \sigma^{\theta^*_0})^2\geq(\sigma_{\min})^2>0$, we get by continuity of the inverse of a matrix that $\mathds P_1$-a.s, there exists some~$T_2\in \mathds N$ such that for any~$T\geq T_2$, it holds \[\|\widetilde G_N^{-1} - \overline G_N^{-1} \|<\epsilon^2,\]
where we recall that 
\[\widetilde G_N = \mathbf X_M^{\top}\mathrm{Diag}(\widetilde \pi^{\theta^*_0} \odot (1-  \widetilde \pi^{\theta^*_0}) )\mathbf X_M,\]
and  \[\overline G_N = \mathbf X_M^{\top}\mathrm{Diag}(\overline \pi^{\theta^*_0} \odot (1-  \overline \pi^{\theta^*_0}) )\mathbf X_M. \]
From Eq.\eqref{eq:cv-asp1} and by continuity of the map~$\Psi$, we get that $\mathds P_1$-a.s.~$\widetilde \theta =\Psi(\mathbf X_M^{\top} \widetilde \pi^{\theta^*_0})\underset{T \to \infty}{\to}\Psi(\mathbf X_M^{\top} \overline \pi^{\theta^*_0})=\overline \theta(\theta^*_0)$ (see Eq.\eqref{eq:gradbar}). Hence, $\mathds P_1$-a.s, there exists some~$T_3\in \mathds N$ such that for any~$T\geq T_3$, it holds
\[\|\widetilde \theta - \overline \theta\|_2 \leq \epsilon.\]
Note that we left the dependence of~$\widetilde \pi^{\theta^*_0}$ and~$\widetilde \theta$ on~$T$ implicit. 
\medskip

\noindent {\bf Step 2: Comparing $\widetilde W_N$ and $W_N$.}\\
It holds for any $Z \in \mathcal E_N$,
\begin{align*}
&\left| \left\|\widetilde G_N^{-1/2}H_N(\widetilde \theta)\left( \Psi(\mathbf X_M^{\top}Z)-\widetilde \theta\right)\right\|_2-\left\|\overline G_N^{-1/2}H_N(\overline \theta)\left( \Psi(\mathbf X_M^{\top}Z)-\overline \theta\right)\right\|_2\right|\\
&\leq \left| \left\|\widetilde G_N^{-1/2}H_N(\widetilde \theta)\left( \Psi(\mathbf X_M^{\top}Z)-\overline \theta\right)\right\|_2-\left\|\overline G_N^{-1/2}H_N(\overline \theta)\left( \Psi(\mathbf X_M^{\top}Z)-\overline \theta\right)\right\|_2\right|\\
&\quad +\left\|\widetilde G_N^{-1/2}H_N(\widetilde \theta)\left(\overline \theta-\widetilde \theta\right)\right\|_2\\
&\leq \|\widetilde G_N^{-1/2}-\overline G_N^{-1/2}\|\|H_N( \widetilde \theta)\| \left\| \Psi(\mathbf X_M^{\top}Z)-\overline \theta\right\|_2\\
&\quad + \|\overline G_N^{-1/2}\|\|H_N( \widetilde \theta)-H_N(\overline \theta)\| \left\| \Psi(\mathbf X_M^{\top}Z)-\overline \theta\right\|_2 +\left\|\mathbf X_M^{\top}\mathbf X_M\right\| \|\overline \theta-\widetilde \theta\|_2.
\end{align*}
Using the Powers–Størmer inequality \citep[cf.][Lemma 4.1]{powersinequality} and denoting $\|M\|_1$ the Schatten 1-norm of any matrix $M$, it holds
\begin{align*}
&\|\widetilde G_N^{-1/2}-\overline G_N^{-1/2}\|^2\leq \|\widetilde G_N^{-1/2}-\overline G_N^{-1/2}\|^2_F\leq \|\widetilde G_N^{-1}-\overline G_N^{-1}\|_1\leq 2s \|\widetilde G_N^{-1}-\overline G_N^{-1}\|,
\end{align*}
where in the last inequality we used that $\widetilde G_N$ and $\overline G_N$ have rank at most $s$. Hence, $\mathds P_1$-a.s, for any~$T\geq T_N(\epsilon):= \max(T_1,T_2,T_3)$ it holds
\begin{align*}&\left| \left\|\widetilde G_N^{-1/2}H_N(\widetilde \theta)\left( \Psi(\mathbf X_M^{\top}Z)-\widetilde \theta\right)\right\|_2-\left\|\overline G_N^{-1/2}H_N(\overline \theta)\left( \Psi(\mathbf X_M^{\top}Z)-\overline \theta\right)\right\|_2\right|\\
&\leq \left\| \Psi(\mathbf X_M^{\top}Z)-\overline \theta\right\|_2\left\{\epsilon 2sCN +(c(\overline \sigma_{\min})^2N)^{-1/2}CN  \epsilon\right\}+CN\epsilon  =: \mathcal C_N(Z,\epsilon).
\end{align*}
We get that $\mathds P_1$-a.s, for any~$T\geq T_N(\epsilon)$ it holds
\begin{align*}&\sup_{Z \in \mathcal E_N}\left| \left\|\widetilde G_N^{-1/2}H_N(\widetilde \theta)\left( \Psi(\mathbf X_M^{\top}Z)-\widetilde \theta\right)\right\|_2-\left\|\overline G_N^{-1/2}H_N(\overline \theta)\left( \Psi(\mathbf X_M^{\top}Z)-\overline \theta\right)\right\|_2\right|\\
&\leq \sup_{Z \in \mathcal E_N}\mathcal C_N(Z,\epsilon)=:\mathcal C_N(\epsilon).\end{align*}

\medskip
\noindent {\bf Step 3: Conclusion.}\\
Let us consider some $\eta \in (0,1-\alpha)$. Since~$\mathcal C_N(\epsilon)$ goes to~$0$ as~$\epsilon\to 0$, we deduce that we can choose $\epsilon$ small enough such that $\mathds P_1$-a.s., for any~$T\geq T_N(\epsilon)$ it holds 
\begin{equation}\label{eq:indicator-rejection}\forall Z \in \mathcal E_N,\quad \mathds 1_{Z \in \widetilde W_N} \leq \mathds 1_{Z \in W_N(\alpha+\eta)},\end{equation}
where 

\setlength\tabcolsep{1.5pt}
\begin{tabular}{rll}
\multirow{2}{*}{$W_N(\alpha+\eta):=\Bigg\{Z \in \{0,1\}^N \; \Bigg| \;$} & $\diamond \, \, \mathbf X_M^{\top} Z \in \mathrm{Im}(\Xi)$&\multirow{2}{2mm}{$\Bigg\}$,}\\
&$\diamond \, \, \left\|[\overline G_N]^{-1/2}H_N(\overline \theta)\left( \Psi(\mathbf X_M^{\top}Z)-\overline \theta\right)\right\|^2_2 > \chi^2_{s,1-\alpha-\eta}$  &
\end{tabular}
\smallskip

\noindent Recalling the definition of $\mathcal E_N$ from Eq.\eqref{def:EN-MLE} and using the definitions of $W_N(\alpha+\eta)$ and $\widetilde W_N$, it also holds trivially
\begin{equation}\label{eq:indicator-rejection2}\forall Z \in \{0,1\}^N \backslash \mathcal E_N,\quad 0= \mathds 1_{Z \in \widetilde W_N} \leq \mathds 1_{Z \in W_N(\alpha+\eta)}=0.\end{equation}
Using both Eq.\eqref{eq:indicator-rejection} and Eq.\eqref{eq:indicator-rejection2}, we deduce that
\[\forall Z \in \{0,1\}^N,\quad \mathds 1_{Z \in \widetilde W_N} \leq \mathds 1_{Z \in W_N(\alpha+\eta)},\]
and we then get that $\mathds P_1$-a.s., for any~$T\geq T_N(\epsilon)$, we have
\begin{align*}
&\zeta_{N,T}=\frac{\sum_{t=1}^T \mathds P_{\theta^*_0}(Z^{(t)}) \mathds 1_{Z^{(t)}\in  \widetilde W_N}}{\sum_{t=1}^T \mathds P_{\theta^*_0}(Z^{(t)})}\leq \frac{\sum_{t=1}^T \mathds P_{\theta^*_0}(Z^{(t)}) \mathds 1_{Z^{(t)}\in  W_N(\alpha+\eta)}}{\sum_{t=1}^T \mathds P_{\theta^*_0}(Z^{(t)})}.
\end{align*}
The right hand side of the previous inequality converges $\mathds P_2$-a.s. to $\overline {\mathds P}_{\theta^*_0}(Y\in W_N(\alpha+\eta))$ as $T\to +\infty$ thanks to Proposition~\ref{prop:point-alternative}. Since from Theorem~\ref{thm:MLEasymptotic} it holds,
\[\underset{N \to +\infty}{\lim \sup}\, \, \overline  {\mathds P}_{\theta^*_0}(Y\in  W_N(\alpha+\eta)) \leq \alpha+\eta,\] we get that for any~$\epsilon>0$, there exists~$N_0 \in \mathds N$ such that for any $N\geq N_0$ it holds,
\[ \mathds P\big( \bigcup_{T_N\in \mathds N} \bigcap_{T\geq T_N} \{\zeta_{N,T}\leq \alpha +\epsilon\}\big)=1.\]

\subsection{Proof of Proposition~\ref{thm:CI}}
\label{proof:thm:CI}
Let us denote~$\mathcal M: \theta \in \mathds R^s \mapsto \mathbf X_M^{\top} \overline \pi^{\theta}.$ Since for any~$z \in \{0,1\}^N$, $\mathds P_{\theta}(z) = \exp( - \mathcal L_{N}(\theta,(z,\mathbf X_M)))$, we get~$\nabla_{\theta} \mathds P_{\theta}(z)=-L_{N}(\theta,(z,\mathbf X_M))\mathds P_{\theta}(z)$. Recalling that $\overline \pi^{\theta} = \overline {\mathds E}_{\theta}[Y]$, we have for any $k \in [s]$,
\begin{align} \frac{\partial \overline \pi^{\theta}}{\partial \theta_k} &= \left( \sum_{w \in E_M} \mathds P_{\theta}(w)  \right)^{-2}\sum_{w,z \in E_M}\mathds P_{\theta}(z) \mathds P_{\theta}(w)z\left\{ L_{N}(\theta,(w,\mathbf X_M))-L_{N}(\theta,(z,\mathbf X_M)) \right\}_k\notag \\
&= \overline {\mathds E}_{\theta}\left[Z\left\{ L_{N}(\theta,(W,\mathbf X_M))-L_{N}(\theta,(Z,\mathbf X_M)) \right\}_k \right]\notag \\
&= \overline {\mathds E}_{\theta}\left[Z\left\{ \mathbf X_M^{\top}(Z-W)\right\}_k \right]\notag \\
&=  \overline \Gamma^{\theta} \mathbf X_{:,M[k]},\label{eq:nable-pibar}
\end{align}
where~$Z$ and~$W$ are independent random vectors valued in $\{0,1\}^N$ and distributed according to $\overline {\mathds P}_{\theta }$. Note that we used that for any $W \in \{0,1\}^N$, it holds
\[L_N(\theta,(W,\mathbf X_M)) = \mathbf X_M^{\top}(\sigma(\mathbf X_M \theta)-W).\]
Hence it holds \[\forall \theta \in \mathds R^s, \quad \nabla \mathcal M(\theta) =  \mathbf X_M^{\top} \overline \Gamma^{\theta} \mathbf X_M.\]
Suppose that we are able to compute an estimate~$ \theta^{\bigstar}\in \mathds B_p(0,R)$ of~$\theta^*$. Using that $\theta^* \in \mathds B_p(0,R)$ and that
\[\inf_{\theta\in\mathds B_p(0,R)}\lambda_{\min}\left(\nabla \mathcal M(\theta)\right)\geq \kappa\lambda_{ \min}\left(\mathbf X_M^{\top} \mathbf X_M\right)\geq c\kappa N ,\]
it holds 
\begin{align*}
\|\mathcal M( \theta^{\bigstar})-\mathcal M(\theta^*)\|_2^2&= \| \int_0^1 \nabla \mathcal M (t\theta^{\bigstar}+(1-t)\theta^*) ( \theta^{\bigstar}-\theta^*) dt\|_2^2\\
&=(\theta^{\bigstar}-\theta^*)^{\top} \left\{ \int_0^1 \nabla \mathcal M (t\theta^{\bigstar}+(1-t)\theta^*) dt \right\}^2(\theta^{\bigstar}-\theta^*)\\
&\geq \|\theta^{\bigstar}-\theta^*\|_2^2 \inf_{\theta \in \mathds B_p(0,R)} \lambda_{\min}( \nabla \mathcal M(\theta))^2\\
&\geq (c\kappa N)^2\|\theta^{\bigstar}-\theta^*\|_2^2.
\end{align*}
Noticing further that 
\[\sup_{\theta\in\mathds R^s}\|\nabla \Psi^{-1}(\theta)\| = \sup_{\theta\in\mathds R^s}\|\mathbf X_M^{\top} \mathrm{Diag}(\sigma'(\mathbf X_M\theta)) \mathbf X_M\| \leq \frac14 CN,\]
we get
\begin{align*}
\| \theta^*-\theta^{\bigstar}\|_2&\leq   \left(\kappa c N\right)^{-1} \|\mathbf X_M^{\top}\overline \pi^{\theta^{\bigstar}}-\mathbf X_M^{\top}\overline \pi^{\theta^*}\|_2\\
&=   \left(\kappa c N\right)^{-1} \|\mathbf X_M^{\top} \pi^{\overline \theta(\theta^{\bigstar})}-\mathbf X_M^{\top} \pi^{\overline \theta(\theta^*)}\|_2\quad \text{(using Eq.\eqref{eq:gradbar})}\\
&\leq   \left(\kappa c N\right)^{-1} \sup_{\theta\in\mathds R^s}\|\nabla \Psi^{-1}(\theta)\|\|\Psi\left(\mathbf X_M^{\top} \pi^{\overline \theta (\theta^{\bigstar})}\right)-\Psi\left(\mathbf X_M^{\top} \pi^{\overline \theta(\theta^*)}\right)\|_2\\
&\leq    C\left(\kappa c \right)^{-1} \|\Psi\left(\mathbf X_M^{\top} \pi^{\overline \theta(\theta^{\bigstar})}\right)-\Psi\left(\mathbf X_M^{\top} \pi^{\overline  \theta(\theta^*)}\right)\|_2 \\
&=    C\left(\kappa c \right)^{-1} \|\overline \theta(\theta^{\bigstar})-\overline \theta(\theta^*)\|_2\\
&\leq C\left(\kappa c \right)^{-1} \left[\|\overline \theta(\theta^{\bigstar})-\widehat \theta\|_2+ \|\widehat \theta-\overline \theta(\theta^*)\|_2\right],
\end{align*}
where we used that $\mathbf X_M^{\top} \pi^{\overline \theta (\theta^*)} = \mathbf X_M^{\top} \sigma\big(\mathbf X_M\overline \theta (\theta^*)\big) = \Xi\big(\overline \theta (\theta^*)\big) \in \mathrm{Im}(\Xi)$ and thus $\Psi(\mathbf X_M^{\top} \pi^{\overline \theta (\theta^*)} )$ is well-defined. Similarly, we have that  $\mathbf X_M^{\top} \pi^{\overline \theta (\theta^{\bigstar})} \in \mathrm{Im}(\Xi)$. 
Since Theorem~\ref{thm:MLEasymptotic} gives that
\[\overline {\mathds P}_{\theta^*}\left( \|V_N(\theta^*)(\widehat \theta - \overline \theta)\|_2^2\leq \chi^2_{s,1-\alpha}\right) \underset{N\to + \infty}{\to}  1-\alpha,\] with $V_N(\theta^*):=[\overline G_N(\theta^*)]^{-1/2}H_N(\overline \theta(\theta^*))$, we deduce (using the assumption of the design matrix from Section~\ref{assumption:design}) that the event  \[
\|\widehat \theta-\overline \theta(\theta^*)\|_2\leq \|[V_N(\theta^*)]^{-1}\| \|V_N(\theta^*)(\widehat \theta - \overline \theta)\|_2\leq\|(\sigma^{\overline \theta})^{-2}\|_{\infty}c^{-1}\left(N/C\right)^{-1/2}\sqrt{\chi^2_{s,1-\alpha}} ,\]
holds with probability tending to $1-\alpha$ as $N\to+\infty$. Note that we used that 
\[\|H_N(\overline \theta(\theta^*))^{-1}\|\leq  (cN)^{-1} \|(\sigma^{\overline \theta})^{-2}\|_{\infty},\]
and that
\[\|[\overline G_N(\theta^*)]^{1/2}\|\leq (CN)^{1/2}.\]
Hence we obtain an asymptotic confidence region for~$\theta^*$ of level~$1-\alpha$.

\subsection{Proof of Proposition~\ref{thm:CI-saturated}}

\label{sec:proof-CI-saturated}
Let us denote~$\mathcal R: \pi \in (0,1)^N \mapsto  \overline \pi^{\pi}.$ It holds for any $i\in[N]$,
\begin{align}  \frac{\partial \overline \pi^{\pi}}{\partial \pi_i} &= \left( \sum_{w \in E_M} \mathds P_{\pi}(w)  \right)^{-2}\sum_{w,z \in E_M}\mathds P_{\pi}(z) \mathds P_{\pi}(w)z\left\{ z-w \right\}_i \big( \pi_i(1-\pi_i)\big)^{-1}\notag \\
&= \overline {\mathds E}_{\pi}\left[Z(Z-W)^{\top}_i \right] \big( \pi_i(1-\pi_i)\big)^{-1}, \notag 
\end{align}
where~$Z$ and~$W$ are independent random vectors valued in $\{0,1\}^N$ and distributed according to $\overline {\mathds P}_{\pi}$. Hence it holds \[\forall \pi \in (0,1)^N, \quad \nabla \mathcal R(\pi) =  \overline \Gamma^{\pi} \mathrm{Diag}(\pi \odot (1-\pi))^{-1}.\]
Suppose that we are able to compute an estimate~$\pi^{\bigstar}\in \mathds B_p(\frac{\mathbf 1_N}2,R)$ of~$\pi^*$. Then since it holds for any $v\in \mathds R^N$,
\[\inf_{\pi\in\mathds B_p(\frac{\mathbf 1_N}2,R)}\|\nabla \mathcal R(\pi)v\|_2\geq  4\kappa \|v\|_2,\]
we get that
\begin{align*}
\|\mathcal R(\pi^{\bigstar})-\mathcal R(\pi^*)\|_2&= \| \int_0^1 \nabla \mathcal R (t\pi^{\bigstar}+(1-t)\pi^*) (\pi^{\bigstar}-\pi^*) dt\|_2\\
&\geq  4\kappa \|\pi^{\bigstar}-\pi^*\|_2.
\end{align*}
Hence we have that
\begin{align*}
\| \pi^*-\pi^{\bigstar}\|_2&\leq  (4\kappa)^{-1} \|\overline \pi^{\pi^{\bigstar}}-\overline \pi^{\pi^*}\|_2\\
&\leq  (4\kappa)^{-1} \big\{\|\mathrm{Proj}_{\mathbf X_M}(\overline \pi^{\pi^{\bigstar}}- Y)\|_2 + \|\mathrm{Proj}_{\mathbf X_M}(Y-\overline \pi^{\pi^*})\|_2\\
&\qquad +\|\mathrm{Proj}^{\perp}_{\mathbf X_M}(\overline \pi^{\pi^{\bigstar}}-\overline \pi^{\pi^*})\|_2\big\}.
\end{align*}
Since Theorem~\ref{prop:CLT} gives that
\[\overline {\mathds P}_{\pi^*}\left( \|[\overline G_N(\pi^*)]^{-1/2}(\mathbf X_M^{\top} Y -\mathbf X_M^{\top} \overline \pi^{\pi^*})\|_2^2\leq \chi^2_{s,1-\alpha}\right) \underset{N\to + \infty}{\to}  1-\alpha,\] we deduce that the event  \begin{align*}
\|\mathbf X_M^{\top} Y-\mathbf X_M^{\top}\overline \pi^{\pi^*}\|_2&\leq \|[\overline G_N(\pi^*)]^{1/2}\| \|[\overline G_N(\pi^*)]^{-1/2}\mathbf X_M^{\top}( Y - \overline \pi^{\pi^*})\|_2\\
&\leq (CN)^{1/2} \sqrt{\chi^2_{s,1-\alpha}} ,\end{align*}holds with probability tending to $1-\alpha$ as $N\to+\infty$. Noticing further that for any vector $v \in \mathds R^N$,
\[\|\mathrm{Proj}_{\mathbf X_M}v\|_2 \leq \|\mathbf X_M \left( \mathbf X_M^{\top} \mathbf X_M \right)^{-1}\| \times \| \mathbf X_M^{\top}v\|_2 \leq (CN)^{1/2} (cN)^{-1} \|\mathbf X_M^{\top}v\|_2,\]
we get that for any $\epsilon>0$, there exists $N_0 \in \mathds N$ such that for any $N\geq N_0$, it holds with at least $1-\alpha-\epsilon$,
\begin{align*}
\| \pi^*-\pi^{\bigstar}\|_2 &\leq  (4\kappa)^{-1} \big\{\|\mathrm{Proj}_{\mathbf X_M}(Y-\overline \pi^{\pi^{\bigstar}})\|_2+Cc^{-1}\sqrt{\chi^2_{s,1-\alpha}} \\
&\qquad +\|\mathrm{Proj}^{\perp}_{\mathbf X_M}(\overline \pi^{\pi^{\bigstar}}-\overline \pi^{\pi^*})\|_2\big\}.
\end{align*}
Hence we obtain an asymptotic confidence region for~$\pi^*$ of level~$1-\alpha$.

\section{Inference conditional on the signs}
\label{sec:conditional-signs}

\subsection{Leftover Fisher information}

As highlighted in~\cite{fithian2014optimal}, conducting inference conditional on some random variable prevents the use of this variable as evidence against a hypothesis. Selective inference should be understood as partitioning the observed information in two sets: the one used to select the model and the one used to make inference. This communicating vessels principle is illustrated with the following inclusions borrowed from~\cite{fithian2014optimal}.
\[\mathcal F_0 \underbrace{\subset}_{\text{used for selection}} \mathcal F(\mathds 1_{Y \in \mathcal M}) \underbrace{\subset}_{\text{used for inference}} \mathcal F(Y). \]
Typically, let us assume that we condition on both the selected support~$\widehat M(Y)=M$ and the observed vector of signs~$\widehat S_M(Y)=S_M\in \{0,1\}^{|M|}$, meaning that~$\mathcal M = E_M^{S_M}$ (cf. Eq.\eqref{def:EMSM}). Even if the vector of signs~$S_M$ is surprising under~$\mathds H_0$, we will not reject unless we are surprised anew by observing the response variable~$Y$. Stated otherwise, when we condition on both the selected support and the vector of signs, we cannot take advantage of the possible unbalanced probability distribution of the vector of signs~$\widehat S_M(Y)$ conditionally on~$E_M$. Hence, conditioning on a finer $\sigma$-algebra results in some information loss. \cite{fithian2014optimal} explain that we can actually quantify this waste of information.
The Hessian of the log-likelihood can be decomposed as
\begin{equation}\label{eq:hessian-likelihood}\nabla^2_{\vartheta}\mathcal L_N(\vartheta,Y \, |\, E_M) =\nabla^2_{\vartheta}\mathcal L_N(\vartheta,\widehat S_M(Y) \,|\, E_M) +\nabla^2_{\vartheta}\mathcal L_N(\vartheta,  Y \, | \, \{ E_M, \widehat S_M(Y)\}).\end{equation}
For any~$\sigma$-algebra~$\mathcal F \subseteq\sigma(Y)$, we consider the conditional expectation
\[\mathcal I_{Y \, |\, \mathcal F} (\vartheta):= -\mathds E  \left[ \nabla^2_\vartheta \mathcal L_N(\vartheta,  Y \, | \,\mathcal F) \,|\,\mathcal F \right].\]
The {\it leftover Fisher information} after selection at~$\widehat S_M(Y)$ is defined by~$\mathcal I_{Y \, |\, \{E_M,\widehat S_M(Y)\}} (\vartheta).$ Taking expectation in both sides of Eq.\eqref{eq:hessian-likelihood} leads to
\begin{align*}\mathds E\left[\mathcal I_{Y \, |\, \{E_M, \widehat S_M(Y)\}} (\vartheta)\right] &= \mathds E \, \mathcal I_{Y \, |\, E_M}(\vartheta) - \mathds E \, \mathcal I_{\widehat S_M(Y) \, |\, E_M}(\vartheta)\\
&\preceq   \mathds E \, \mathcal I_{Y \, |\, E_M}(\vartheta)
,\end{align*}
which can also be written as
\[\sum_{S_M \in \{\pm1\}^s} \mathds P(\widehat S_M(Y)=S_M \, |\, E_M)\, \mathds E \mathcal I_{Y \,|\, E_M^{S_M}}(\vartheta)\preceq   \mathds E \, \mathcal I_{Y \, |\, E_M}(\vartheta).\]
In expectation, the loss of information induced by conditioning further on the vector of signs is quantified by the information~$\widehat S_M(Y)$ carries about~$\vartheta$. Let us stress that this conclusion is only true in expectation and it may exist some vector of signs~$S_M \in \{-1,+1\}^s$ such that
\[\mathcal I_{Y \, |\, E_M} (\vartheta) \preceq   \mathcal I_{Y \, |\, E_M^{S_M}}(\vartheta).\]

Hence, conditioning on the signs will generally lead to wider confidence intervals. Nevertheless, let us stress that inference procedures correctly calibrated conditional on $E_M^{S_M}$ will be also valid conditional on $E_M$. More precisely, considering some transformation $T:\mathds R^N \to \mathds R$ and real valued random variables $L(Y,S_M)<U(Y,S_M)$ such that for any vector of signs $S_M \in \{-1,+1\}^s$ it holds
\[{\mathds P}\left(T(\pi^*) \in [L(Y,S_M),U(Y,S_M)] \, | \, E_M^{S_M}\right)=1-\alpha ,\]
the confidence interval has also~$(1-\alpha)$ coverage conditional on the~$E_M=\{\widehat M (Y)=M\}$ since
\begin{align*}
&\mathds P( T(\pi^*) \in [L(Y,\widehat S_M(Y)),U(Y,\widehat S_M(Y))]  \; |\; E_M) \\
&=\sum_{S_M \in \{\pm1\}^s} \mathds P(\widehat S_M(Y)=S_M \, |\, E_M)\underbrace{\mathds P( T(\pi^*) \in [L(Y,S_M),U(Y, S_M)]  \; |\; E_M^{S_M})}_{=1-\alpha} \\
&=1-\alpha.
\end{align*}

\subsection{Discussion}
\label{comparison:taylorGLM}

Let us recall that in~\cite{taylorGLM}, the authors work in the selected model for logistic regression. They consider a selected model $M\subseteq[d]$ associated to a response vector $Y=(y_i)_{i\in[n]}\in \{0,1\}^N$ where for any~$i\in[N]$, $y_i$ is a Bernoulli random variable with parameter~$\{\sigma(\mathbf X_M \theta^*)\}_i$ for some $\theta^* \in \mathds R^s$ ($s=|M|$). As presented in Section~\ref{sec:MLEvsBIAS}, in~\cite{taylorGLM} the authors claim the following asymptotic distribution
\begin{equation}\label{eq:taylor}\underline{\theta} \sim \mathcal N(\vartheta^*_M,H_N( \vartheta^{*}_M)^{-1}),\end{equation}
 where $\underline{\theta} =  \hat \vartheta^{\lambda}_M + \lambda H_N(\widehat \vartheta^{\lambda}_M)^{-1}\widehat S_M(Y).$
Note that this approximation corresponds to the one usually made to form Wald tests and confidence intervals in
generalized linear models. They claim that the selection event~$\{Y \in \{0,1\}^N\; :\; \widehat M(Y)=M, \widehat S_M(Y)=S_M\}$ can be asymptotically approximated by
\[\{ Y \; : \; \mathrm{Diag}(S_M) \left(  \underline{\theta} - H_N(\vartheta^{*}_M)^{-1}\lambda S_M\right)\geq 0  \}.\] Let us denote by~$F^{[a,b]}_{\mu,\sigma^2}$ the CDF of a $\mathcal N(\mu,\sigma^2)$ random variable
truncated to the interval $[a,b]$. Then they use the polyhedral lemma to state that for some random variables $\mathcal V^-$ and $\mathcal V^+$ it holds
\[\left[F^{[\mathcal V^-_{S_M},\mathcal V^+_{S_M}]}_{\vartheta^*_{M[j]},\left[H_N(\vartheta^*_M)^{-1}\right]_{j,j}}(\underline \theta_j) \; | \; \widehat M(Y)=M, \; \widehat S_M(Y)=S_M\right]\; \sim \; \mathcal U([0,1]).\]
Several problems arise at this point.
\begin{enumerate}
\item {\bf Lack of theoretical guarantee due to the use of Monte-Carlo estimates.}\\
The first problem is that both~$\underline \theta$ and the selection event~$\{\widehat M(Y)=M, \; \widehat S_M(Y)=S_M\}$ involve the unknown parameter~$\vartheta^*_M$ through~$H_N(\vartheta^*_M)$. Taylor and al. propose to use a Monte-Carlo estimate for~$H_N(\vartheta^*_M)$ by replacing it with~$H_N(\widehat \theta^{\lambda})$. Using this Monte-Carlo estimate, one can compute~$L$ and~$U$ such that
\[F^{[\mathcal V^-_{S_M},\mathcal V^+_{S_M}]}_{L,\left[H_N(\vartheta^*_M)^{-1}\right]_{j,j}}(\underline \theta_j)=1-\frac{\alpha}{2} \quad \text{and} \quad F^{[\mathcal V^-_{S_M},\mathcal V^+_{S_M}]}_{U,\left[H_N(\vartheta^*_M)^{-1}\right]_{j,j}}(\underline \theta_j)=\frac{\alpha}2.\]
Then,~$[L,U]$ is claimed to be a confidence interval with (asymptotic)~$(1-\alpha)$ coverage for~$\vartheta^*_{M[j]}$ conditional on~$\{\widehat M(Y) = M, \widehat S_M(Y)=S_M\}$, that is,
\[\mathds P(\vartheta^*_{M[j]}\in [L,U]\; |\; \widehat M(Y) = M,\; \widehat S_M(Y)=S_M)= 1 -\alpha.\]

\item {\bf Their approach is not well suited to provide more powerful inference procedures by conditioning only on $E_M$.}\\In the linear model,~\cite{sun16} also start by deriving a pivotal quantity by conditioning on both the selected variables and the vector of signs. However, in the context of linear regression, the vector of signs only appears in the threshold values~$\mathcal V^-$ and~$\mathcal V^+$. Hence, conditioning only on the selected variables~$\{\widehat M(Y)=M\}$ simply reduces to take the union~$\cup_{S_M\in \{\pm1\}^s} [\mathcal V^-_{S_M},\mathcal V^+_{S_M}]$ for the truncated Gaussian. In the method proposed by~\cite{taylorGLM}, the vector of signs also appears in the computation of~$\underline \theta$. The consequence is that the (asymptotic) distribution of~$\underline{\theta}$ conditional on~$\{\widehat M(Y)=M\}$ is not a truncated Gaussian anymore but a mixture of truncated Gaussians. In this situation, it seems unclear how to take advantage of this structure to provide more powerful inference procedures.
\end{enumerate}

\end{document}